\documentclass[12pt]{amsart}

\usepackage{enumerate, amsmath, amsthm, amsfonts, amssymb, xy,  mathrsfs, graphicx, paralist, fancyvrb,ytableau, tikz, tikz-cd,wrapfig, ytableau, multicol, float}
\PassOptionsToPackage{usenames, dvipsnames}{xcolor}
\usepackage[margin=1in]{geometry} 
\usepackage[bookmarks, colorlinks=true, linkcolor=blue, citecolor=blue, urlcolor=blue]{hyperref}
\usetikzlibrary{cd}

\usepackage{adjustbox}
\usepackage{etoolbox}

\input xy
\xyoption{all}

\numberwithin{equation}{subsection}
\newtheorem{theorem}[equation]{Theorem}

\newtheorem{proposition}[equation]{Proposition}
\newtheorem{lemma}[equation]{Lemma}
\newtheorem{corollary}[equation]{Corollary}

\theoremstyle{definition}
\newtheorem{rmk}[equation]{Remark}
\newenvironment{remark}[1][]{\begin{rmk}[#1] \pushQED{\qed}}{\popQED \end{rmk}}
\newtheorem{eg}[equation]{Example}
\newenvironment{example}[1][]{\begin{eg}[#1] \pushQED{\qed}}{\popQED \end{eg}}
\newtheorem{defn}[equation]{Definition}
\newenvironment{definition}[1][]{\begin{defn}[#1]\pushQED{\qed}}{\popQED \end{defn}}
\newtheorem{notn}[equation]{Notation}
\newenvironment{notation}[1][]{\begin{notn}[#1]\pushQED{\qed}}{\popQED \end{notn}}

\newcommand{\rE}{\mathrm{E}}

\newcommand{\cH}{\mathcal{H}}

\newcommand{\rH}{\mathrm{H}}

\newcommand{\cK}{\mathcal{K}}

\newcommand{\bN}{\mathbf{N}}

\newcommand{\cO}{\mathcal{O}}

\newcommand{\rP}{\mathrm{P}}

\newcommand{\bQ}{\mathbf{Q}}

\newcommand{\rQ}{\mathrm{Q}}

\newcommand{\bS}{\mathbf{S}}

\newcommand{\bV}{\mathbf{V}}

\newcommand{\bZ}{\mathbf{Z}}

\newcommand{\bk}{\mathbf{k}}

\newcommand{\defi}[1]{{\bf\upshape\sffamily #1}}

\renewcommand{\phi}{\varphi}
\renewcommand{\emptyset}{\varnothing}
\newcommand{\eps}{\varepsilon}

\newcommand{\DS}{\displaystyle}

\makeatletter
\def\Ddots{\mathinner{\mkern1mu\raise\p@
\vbox{\kern7\p@\hbox{.}}\mkern2mu
\raise4\p@\hbox{.}\mkern2mu\raise7\p@\hbox{.}\mkern1mu}}
\makeatother

\DeclareMathOperator{\im}{image} 

\DeclareMathOperator{\coker}{coker}

\DeclareMathOperator{\supp}{Supp}

\DeclareMathOperator{\rank}{rank}

\DeclareMathOperator{\End}{End}

\DeclareMathOperator{\Sym}{Sym}

\DeclareMathOperator{\Tor}{Tor}

\DeclareMathOperator{\Spec}{Spec}

\DeclareMathOperator{\sgn}{sgn}

\DeclareMathOperator{\gr}{gr}

\newcommand{\GL}{\mathbf{GL}}

\newcommand{\Gr}{\mathbf{Gr}}

\usepackage{tikz}
\usepackage{makecell}

\newcommand{\run}{\preceq_{\rm run}}
\newcommand{\strun}{\prec_{\rm run}}
\newcommand{\smallhooktableau}[1]{%
  \vcenter{\hbox{\scalebox{1.14}{$\ytableaushort{#1}$}}}}

\usetikzlibrary{arrows.meta}

\newcommand{\HookPanelThreeTwo}[5]{%
 \begin{scope}[shift={(#1,#2)}]
  \foreach \r/\len in {1/3,2/2}{
    \foreach \c in {1,...,\len}{
      \fill[gray!15] (\c-1,1-\r) rectangle ++(1,-1);
    }
  }
  \ifstrempty{#4}{}{%
    \foreach \c/\r in {#4}{
      \fill[blue!28] (\c-1,1-\r) rectangle ++(1,-1);
    }
  }
  \ifstrempty{#5}{}{%
    \foreach \c/\r in {#5}{
      \fill[yellow!60] (\c-1,1-\r) rectangle ++(1,-1);
    }
  }
  \foreach \r/\len in {1/3,2/2}{
    \foreach \c in {1,...,\len}{
      \draw[line width=.45pt] (\c-1,1-\r) rectangle ++(1,-1);
    }
  }
  \foreach \c/\r in {1/1,2/2}{
    \fill (\c-0.5,0.5-\r) circle (1.1pt);
  }
  \node[font=\scriptsize] at (1.5,-2.55) {#3};
 \end{scope}
}

\title[Hook Algebras and Determinantal Varieties]{Hook Algebras and Determinantal Varieties}

\author{Steven V Sam}
\address{Department of Mathematics, University of California, San Diego}
\email{ssam@ucsd.edu}
\thanks{SS was partially supported by NSF grant DMS-2302149.}

\author{Keller VandeBogert}
\address{Department of Mathematics, University of Kentucky}
\email{keller.v@uky.edu}

\author{Jerzy Weyman}
\address{Institute of Mathematics, Jagiellonian University}
\email{jerzy.weyman@uj.edu.pl}

\date{October 7, 2026}

\begin{document}

\begin{abstract}
In characteristic zero, we prove that the Tor algebras of the generic,
symmetric, and skew-symmetric determinantal rings are Koszul with respect
to a natural regrading. In particular, their Eilenberg--Moore spectral sequence degenerates, which allows us to explicitly determine the Poincar\'e series of their residue fields and answers a problem raised
by Avramov. To do this, we construct the universal hook algebra and prove that it has a quadratic Gr\"obner basis and thus is Koszul. This leads to a new standard monomial construction of the irreducible polynomial \(\GL\)-representations, alongside the Schur and Weyl functors, which is of independent interest.
\end{abstract}

\maketitle

\vspace{-.2in}

\setcounter{tocdepth}{1}
\tableofcontents

\section{Introduction}

\subsection{The Poincar\'e series problem for determinantal rings}

For a Noetherian local ring \(R\) with residue field \(\bk\), the
Poincar\'e series of \(\bk\) is
\[
  \rP^R_\bk(t)
  =
  \sum_{i\ge0}\dim_\bk\Tor_i^R(\bk,\bk)t^i.
\]
The goal of this paper is to compute this series for the three classical
families of determinantal rings. More precisely, as a special case, we prove the following:

\begin{theorem}\label{thm:intro-main}
Assume \(\operatorname{char}\bk=0\). Let \(S\) be \(\Sym(V\otimes W)\),
\(\Sym(S^2V)\), or \(\Sym(\bigwedge^2V)\), and let \(R=S/I\), where \(I\)
is the ideal of \(r\times r\) minors, \(r\times r\) symmetric minors, or
\(2r\times2r\) Pfaffians, with \(r\ge2\) and \(I\ne0\). Then
\[
  \rP^R_\bk(t)
  =
  \frac{(1+t)^{\dim_\bk S_1}}
  {\displaystyle
   \sum_{i,d\ge0}
   (-1)^d
   (\dim_\bk\Tor_i^S(R,\bk)_{i+(r-1)d})
   t^{i+d}}.
\]
In particular, \(\rP^R_\bk(t)\) is rational.
\end{theorem}

As \(\GL\)-representations, the groups \(\Tor_i^S(R,\bk)_j\) are known
from the work of Lascoux in the generic case~\cite{Lascoux1978} and of
J\'ozefiak, Pragacz, and Weyman in the symmetric and Pfaffian cases
\cite{Jozefiak1978,JozefiakPragacz1979,JozefiakPragaczWeyman1981}, so the
Betti numbers appearing in the denominator may be read off from explicit
combinatorial formulas; Corollary~\ref{cor:poincare-applications}
refines the formula above by retaining both the internal grading and
the group action. The \emph{multiplicative} structure of the Koszul
homology algebra \(T=\Tor^S(R,\bk)\) is far more subtle and is precisely
what Theorem~\ref{thm:intro-main} depends on; its study occupies the
bulk of this paper.

For generic determinantal rings, Avramov singled out the computation of
\(\rP_\bk^R(t)\) as ``a very difficult
problem''~\cite[\S\S6--8]{Avramov1978}. He settled the cases of maximal
minors and of the submaximal minors of a square matrix
\cite[Theorem~6.7, Remark~6.8, Theorem~8.9]{Avramov1978}; see
also~\cite{Gover1973}. Beyond these cases and the Koszul case
\(r=2\)~\cite{BackelinFroberg1985,DeConciniEisenbudProcesi1982}, the
computation has remained open;
Theorem~\ref{thm:intro-main} fully resolves this question in characteristic zero.

\begin{example}\label{ex:intro-middle-minors}
For a generic \(5\times5\) matrix \(X\), let
\[
  R=\bk[x_{ij}:1\le i,j\le5]/I_3(X).
\]
This is the first case outside the two that Avramov settled: the minors
are neither maximal nor submaximal, and since \(R\) is defined by cubics
it is not Koszul. Theorem~\ref{thm:intro-main} yields
\[
\begin{aligned}
  \rP^R_\bk(t)
  &=
  \frac{(1+t)^{25}}
  {1-100t^2-450t^3-828t^4-700t^5
    +700t^7+828t^8+450t^9+100t^{10}-t^{12}} \\[2pt]
  &=1+25t+400t^2+5250t^3+64{,}728t^4+779{,}530t^5
    +9{,}361{,}100t^6+\cdots.
  \qedhere
\end{aligned}
\]
\end{example}

Whether or not \(\rP^R_\bk(t)\) is always rational was for decades the
Serre--Kaplansky problem \cite{Serre2000LocalAlgebra,Kaplansky1974CommutativeRings}, resolved negatively by Anick~\cite{Anick1982,Anick1982Note}; the question is
therefore not whether the series is always rational, but to identify the
structural classes for which it is.

The passage from the finite \(S\)-free resolution of \(R\) to the
infinite resolution of \(\bk\) over \(R\) is governed by the algebraic
Eilenberg--Moore spectral sequence
\[
  \Tor^T\!\bigl(\Tor^S(\bk,\bk),\bk\bigr)
  \Longrightarrow \Tor^R(\bk,\bk),
\]
where \(T\) is regarded as a graded algebra by Tor homological degree.
The Hilbert series of the \(E^2\)-page is a coefficientwise upper bound
for \(\rP_\bk^R(t)\), with equality precisely when the spectral sequence
degenerates, and in one form or another essentially every closed-form
Poincar\'e series computation amounts to establishing such a
degeneration. Whether the spectral sequence degenerates genuinely
depends on the multiplication on \(T\): Roos exhibited quadratic
quotients with identical graded Betti tables over their ambient
polynomial rings for which the spectral sequence degenerates for one but
not the other~\cite[\S\S0,3]{Roos2016}.

The main insight of this paper is that in the right grading this algebra
structure is as homologically simple as it can be. The classical
resolutions impose the rigid constraint
\[
  \Tor_i^S(R,\bk)_j=0
  \qquad\text{unless}\qquad
  j-i=(r-1)d
\]
for some \(d\ge0\). We define the \defi{strand grading} by
\[
  \deg_{\mathrm{str}}
  \Tor_i^S(R,\bk)_{i+(r-1)d}=d.
\]
Its degree one component is the first nonzero strand.

With this grading in hand, we prove the following:

\begin{theorem}\label{thm:intro-structure}
Assume \(\operatorname{char}\bk=0\). For each of the three determinantal
rings above, \(T=\Tor^S(R,\bk)\) is generated in strand degree one and is
Koszul with respect to the strand grading
\textup{(Theorem~\ref{thm:koszul-strands})}.
\end{theorem}

\begin{example}\label{ex:intro-strand-table}
Continuing Example~\ref{ex:intro-middle-minors}, the Betti table of \(R\), with rows indexed by \(j-i\) and columns by homological degree \(i\), is
\[
\tiny \begin{array}{r|cccccccccc|l}
 & 0 & 1 & 2 & 3 & 4 & 5 & 6 & 7 & 8 & 9 & \\\cline{1-11}
0 & 1 & \cdot & \cdot & \cdot & \cdot & \cdot & \cdot & \cdot & \cdot
& \cdot & \text{strand degree }0\\
1 & \cdot & \cdot & \cdot & \cdot & \cdot & \cdot & \cdot & \cdot &
\cdot & \cdot & \\
2 & \cdot & 100 & 450 & 828 & 700 & 225 & \cdot & \cdot & \cdot &
\cdot & \text{strand degree }1\\
3 & \cdot & \cdot & \cdot & \cdot & \cdot & \cdot & \cdot & \cdot &
\cdot & \cdot & \\
4 & \cdot & \cdot & \cdot & \cdot & 225 & 700 & 828 & 450 & 100 &
\cdot & \text{strand degree }2\\
5 & \cdot & \cdot & \cdot & \cdot & \cdot & \cdot & \cdot & \cdot &
\cdot & \cdot & \\
6 & \cdot & \cdot & \cdot & \cdot & \cdot & \cdot & \cdot & \cdot &
\cdot & 1 & \text{strand degree }3
\end{array} \qedhere
\]
\end{example}

In Example~\ref{ex:intro-strand-table}, the first strand has total
dimension \(100+450+828+700+225=2303\): Theorem~\ref{thm:intro-structure}
asserts that the Koszul homology algebra of the \(3\times3\) minors of a
generic \(5\times5\) matrix is minimally generated by \(2303\) elements;
in general, the number of generators grows rapidly with the size of the
matrix.
By Proposition~\ref{prop:em-rational-form}, if \(T\) is Koszul with
respect to the strand grading, then the Eilenberg--Moore spectral
sequence degenerates, and the resulting Euler characteristic identity is
precisely the formula of Theorem~\ref{thm:intro-main}.

Known proofs of such a degeneration use additional structure, most
notably through \emph{Avramov's machine}
\cite{Kustin1993TorACI,Avramov1998Infinite}, which starts with an
associative differential graded algebra (dga) structure on the minimal free
resolution of \(R\) over \(S\). Buchsbaum and Eisenbud constructed such
structures for resolutions of length at most three
\cite[Propositions~1.1 and~1.3]{BuchsbaumEisenbud1977}, and this
machinery underlies most known positive results
\cite{AvramovKustinMiller1988,Weyman1989,
KustinMiller1980AlgebraStructure,KustinMiller1985TorAlgebras,
Eisenbud1980Homological,
JacobssonKustinMiller1985Codim4,KustinPalmerSlattery1994ACI,
RossiSega2014Compressed,KustinSegaVraciu2018CompressedOdd,
Gupta2017GoodBad,Gupta2026Criterion,Lescot1985Asymptotic,
Lescot1990Inert,AvramovHenriquesSega2013QCI}. A more recent method
combines formality of the derived fiber \(R\otimes_S^{\mathrm L}\bk\)
with Koszulity of its homology~\cite{BriggsCameronLetzPollitz2025}.
Neither is available here: a differential graded algebra structure need
not exist~\cite{Avramov1981Obstructions}, and Srinivasan proved that it
does not for the \(4\times4\) Pfaffians of a generic \(6\times6\)
alternating matrix~\cite{Srinivasan1992}; the derived fiber is likewise
not formal in general (Remark~\ref{rem:formality-not-needed}). Our
degeneration is instead a form of \emph{quasi-formality} in the sense of
Positselski~\cite[Corollary~2.2]{Positselski2017Quasiformal}, closely
related to the vanishing of higher Massey products.

Moreover, in almost every previously settled class, setting aside
complete intersections and Koszul algebras, the augmentation ideal of
the Koszul homology algebra cubes to zero. The determinantal Tor
algebras are, by contrast, of unbounded multiplicative complexity
(Proposition~\ref{prop:generic-strand-generation}), and to our knowledge
the results below give the first closed-form Poincar\'e series for such
a family beyond these settings.

Theorem~\ref{thm:intro-structure} is proved using a construction that is
independent of determinantal rings and of interest in its own right: the
universal hook algebra.

\subsection{The universal hook algebra}

Every partition \(\lambda\) decomposes uniquely along its main diagonal
into nested hooks, its diagonal hooks, the \(i\)th having arm length
\(a_i\) and leg length \(b_i\); we write
\(\lambda=(a_1,\ldots,a_d\mid b_1,\ldots,b_d)\), where
\(a_1>\cdots>a_d\ge0\), \(b_1>\cdots>b_d\ge0\), and \(d=\rank(\lambda)\)
is the number of hooks. Two partitions \(\lambda\) and \(\mu\) are
\defi{mergeable} if their combined diagonal hooks can still be nested;
in that case the hooks assemble into a single partition, the
\defi{Frobenius union} \(\lambda\sqcup\mu\)
(Definition~\ref{def:frobenius-union}), and \(\bS_{\lambda\sqcup\mu}\)
occurs with multiplicity one in \(\bS_\lambda\otimes\bS_\mu\). For
instance, \((2,1)\) and \((1)\) have hooks \((1\mid1)\) and
\((0\mid0)\); the arms \(1>0\) and legs \(1>0\) interleave, so the pair
is mergeable, and the hooks reassemble:
\[
\ytableausetup{boxsize=0.8em}
\ydiagram{2,1}\;\sqcup\;\ydiagram{1}\;=\;\ydiagram{2,2}.
\]
Two copies of \((1)\) are not mergeable, since their arm lengths
coincide.

It turns out that these diagonal hooks generate a single algebra
controlling all three determinantal families at once.
The Pfaffian case already makes this generation visible: in
Figure~\ref{fig:pfaff}, the hooks in the first strand nest to give every
shape in the higher strands.

\begin{example}
  For \(r=2\) and \(\dim V=6\), we depict the Betti table for \(S/I_4\) in Figure~\ref{fig:pfaff}; a Young diagram for $\lambda$ is used in place of the Schur module $\bS_\lambda(V)$. The color-coding makes the relevance of merging hook shapes self-evident.
\begin{figure}
  \[
    \ytableausetup{boxsize=0.5em}
    \setlength{\arraycolsep}{4pt}
\begin{array}{r|ccccccc}
 & 0 & 1 & 2 & 3 & 4 & 5 & 6 \\
\hline
0 & 1 & \cdot & \cdot & \cdot & \cdot & \cdot & \cdot \\
1 & \cdot
  & \begin{ytableau} *(red!30) \\ *(red!30) \\ *(red!30) \\ *(red!30) \end{ytableau}
  & \begin{ytableau} *(blue!30) & *(blue!30) \\ *(blue!30) \\ *(blue!30) \\ *(blue!30) \\ *(blue!30) \end{ytableau}
  & \begin{ytableau} *(green!45) & *(green!45) & *(green!45) \\ *(green!45) \\ *(green!45) \\ *(green!45) \\ *(green!45) \\ *(green!45) \end{ytableau}
  & \cdot & \cdot & \cdot \\
2 & \cdot & \cdot & \cdot
  & \begin{ytableau} *(blue!30) & *(blue!30) \\ *(blue!30) & *(red!30) \\ *(blue!30) & *(red!30) \\ *(blue!30) & *(red!30) \\ *(blue!30) & *(red!30) \end{ytableau}
  & \begin{ytableau} *(green!45) & *(green!45) & *(green!45) \\ *(green!45) & *(red!30) \\ *(green!45) & *(red!30) \\ *(green!45) & *(red!30) \\ *(green!45) & *(red!30) \\ *(green!45) \end{ytableau}
  & \begin{ytableau} *(green!45) & *(green!45) & *(green!45) \\ *(green!45) & *(blue!30) & *(blue!30) \\ *(green!45) & *(blue!30) \\ *(green!45) & *(blue!30) \\ *(green!45) & *(blue!30) \\ *(green!45) & *(blue!30) \end{ytableau}
  & \cdot \\
3 & \cdot & \cdot & \cdot & \cdot & \cdot & \cdot
  & \begin{ytableau} *(green!45) & *(green!45) & *(green!45) \\ *(green!45) & *(blue!30) & *(blue!30) \\ *(green!45) & *(blue!30) & *(red!30) \\ *(green!45) & *(blue!30) & *(red!30) \\ *(green!45) & *(blue!30) & *(red!30) \\ *(green!45) & *(blue!30) & *(red!30) \end{ytableau}
\end{array}
\]
\caption{Betti table of $4 \times 4$ Pfaffians}
\label{fig:pfaff}
\end{figure}
\end{example}

The \defi{universal hook algebra} \(\cH(V)\) is generated by the hook Schur
functors \(\bS_{(a+1,1^b)}(V)\), one for each \defi{hook type}
\((a,b)\); the product of two hooks is zero when they do not merge, and
otherwise it is the projection onto the Schur functor of their Frobenius
union. Since that summand occurs with multiplicity one in the tensor
product of the two hook Schur functors, the resulting quadratic
relations are canonical; the precise presentation, together with our
conventions on Schur functors, is given in
\S\S\ref{sec:background}--\ref{sec:hook-algebra}. The main properties of
this algebra are summarized in the following (see Theorems~\ref{thm:main-hook} and~\ref{thm:quad-gb}):

\begin{theorem}\label{thm:intro-hook}
Assume \(\operatorname{char}\bk=0\). Let \(\cH(V)\) be the universal hook algebra.
\begin{enumerate}[\rm (1)]
\item As a bigraded \(\GL(V)\)-representation,
\[
  \cH(V)\cong
  \bigoplus_{\ell(\lambda)\le\dim V}\bS_\lambda(V),
  \qquad
  \deg\bS_\lambda(V)=(\rank(\lambda),|\lambda|).
\]
Here \(\ell(\lambda)\) and \(|\lambda|\) denote the length and size of
\(\lambda\).
Thus every irreducible polynomial \(\GL(V)\)-representation occurs with
multiplicity one.
\item The defining ideal has a quadratic Gr\"obner basis. Its standard
monomials are precisely the hook products whose assembled filling is
semistandard, see \S\ref{sec:term-orders}.
\item The algebra \(\cH(V)\) is Koszul with respect to the rank grading.
\end{enumerate}
\end{theorem}

\begin{example}\label{ex:intro-three-diagonal-hooks}
Let \(\dim V=3\), and let \(A\subseteq\cH(V)\) be the subalgebra generated
by the hook summands of types \((0,0)\), \((1,1)\), and \((2,2)\). Its
rank-graded components, with diagonal hooks color-coded by type, are
\[
\ytableausetup{boxsize=0.6em}
\begin{aligned}
A_0&=\bk, &
A_1&\cong
\begin{array}[t]{@{}c@{}}\bS_{(1)}(V)\\[2pt]
\begin{ytableau}*(green!50)\end{ytableau}\end{array}
\oplus
\begin{array}[t]{@{}c@{}}\bS_{(2,1)}(V)\\[2pt]
\begin{ytableau}*(red!35)&*(red!35)\\ *(red!35)\end{ytableau}\end{array}
\oplus
\begin{array}[t]{@{}c@{}}\bS_{(3,1,1)}(V)\\[2pt]
\begin{ytableau}*(blue!35)&*(blue!35)&*(blue!35)\\
*(blue!35)\\ *(blue!35)\end{ytableau}\end{array},\\[2pt]
A_2&\cong
\begin{array}[t]{@{}c@{}}\bS_{(2,2)}(V)\\[2pt]
\begin{ytableau}*(red!35)&*(red!35)\\
*(red!35)&*(green!50)\end{ytableau}\end{array}
\oplus
\begin{array}[t]{@{}c@{}}\bS_{(3,2,1)}(V)\\[2pt]
\begin{ytableau}*(blue!35)&*(blue!35)&*(blue!35)\\
*(blue!35)&*(green!50)\\ *(blue!35)\end{ytableau}\end{array}
\oplus
\begin{array}[t]{@{}c@{}}\bS_{(3,3,2)}(V)\\[2pt]
\begin{ytableau}*(blue!35)&*(blue!35)&*(blue!35)\\
*(blue!35)&*(red!35)&*(red!35)\\ *(blue!35)&*(red!35)\end{ytableau}\end{array}, &
A_3&\cong
\begin{array}[t]{@{}c@{}}\bS_{(3,3,3)}(V)\\[2pt]
\begin{ytableau}*(blue!35)&*(blue!35)&*(blue!35)\\
*(blue!35)&*(red!35)&*(red!35)\\
*(blue!35)&*(red!35)&*(green!50)\end{ytableau}\end{array},\\[2pt]
\operatorname{Hilb}_A(q)&=1+17q+17q^2+q^3.
\end{aligned}
\]
Writing a hook product as the filling obtained by assembling its diagonal
hooks, one quadratic straightening relation is
\[
\ytableausetup{boxsize=0.9em}
\vcenter{\hbox{\begin{ytableau}
*(blue!35)1 & *(blue!35)2 & *(blue!35)3\\
*(blue!35)2 & *(green!50)1\\
*(blue!35)3
\end{ytableau}}}
\;=\;-\,\vcenter{\hbox{\begin{ytableau}
*(blue!35)1 & *(blue!35)1 & *(blue!35)3\\
*(blue!35)2 & *(green!50)2\\
*(blue!35)3
\end{ytableau}}}
\;-\;\vcenter{\hbox{\begin{ytableau}
*(blue!35)1 & *(blue!35)1 & *(blue!35)2\\
*(blue!35)2 & *(green!50)3\\
*(blue!35)3
\end{ytableau}}}.
\]
Since each tableau involves exactly two colors, this is a quadratic
relation with respect to the rank grading. The filling on the left is
nonsemistandard,
whereas the two fillings on the right are semistandard. The unique rank-three standard monomial of \(A\) is the
\(3\times3\) tableau whose \(i\)th row is filled with \(i\); it spans
\(A_3\cong\bS_{(3,3,3)}(V)\cong\det(V)^3\).
\end{example}

Thus \(\cH(V)\) provides a standard monomial theory for products of hook
Schur functors, with the semistandard hook tabloids playing the role
that semistandard tableaux play in the classical straightening laws of
Akin--Buchsbaum--Weyman
\cite{AkinBuchsbaumWeyman1981,AkinBuchsbaumWeyman1982} and
De Concini--Eisenbud--Procesi~\cite{DeConciniEisenbudProcesi1980};
Example~\ref{ex:straighten-33} straightens the same tableau in the hook,
Schur, and Weyl models side by side.
The universal hook algebra is a hook analogue of the classical shape algebra
studied by Towber~\cite{Towber1977,LancasterTowber1979,TaftTowber1991}.
Here the Cartan product is replaced by Frobenius union. Inamdar and Mehta
proved that the classical shape algebra, and more generally its Schubert
quotients, are Koszul~\cite{InamdarMehta1994,Inamdar1994Note}.
The hook relations do not follow formally from these theories and give a new
system of Garnir relations for hook products: the proposed lift of the usual
Garnir relations in
\cite{GorodentsovKhoroshkinRudakov2007} already fails for the pair of
hooks in Remark~\ref{rmk:gorodentsov}.

The main difficulty in constructing the hook algebra is to produce, for every
nonsemistandard quadratic hook tabloid, a relation having that tabloid as its
leading term. In
\S\S\ref{sec:hook-algebra}--\ref{sec:elim-systems}, we obtain these
relations from explicit row and column Pieri maps and solve the resulting
elimination systems; this proves triangularity for the term order of
\S\ref{sec:term-orders}.

The universal hook algebra captures only the leading part of the
multiplication on the Tor algebra \(T\): a product in \(T\) can contain
nonzero irreducible components not retained by the corresponding product
in \(\cH(V)\), and such components genuinely occur in general
(Remark~\ref{rem:filtration-subtlety}), so passing to an associated
graded algebra is unavoidable. We construct a finite multiplicative
filtration whose associated graded algebra is a finite model built from
\(\cH(V)\) (see Propositions~\ref{prop:diagonal-hook-degeneration}
and~\ref{prop:generic-degeneration} and
Corollary~\ref{cor:hook-degreewise-compatibility}), and we prove the
following:

\begin{theorem}\label{thm:intro-degeneration}
Assume \(\operatorname{char}\bk=0\). Then \(T\) is a filtered
deformation either of a finite algebra retract of \(\cH(V)\) or of an
algebra built out of Veronese and Segre operations of such retracts, and
is therefore Koszul with respect to the strand grading.
\end{theorem}

To summarize, the main steps of our argument are as follows:
\begin{enumerate}
\item \(T\) is generated by its first-strand hook summands,
\item \(\gr T\) is isomorphic to a Koszul algebra built from \(\cH(V)\),
\item Koszulity of \(\gr T\) implies that \(T\) is Koszul with respect
to the strand grading,
\item the Eilenberg--Moore spectral sequence degenerates, which gives
the Poincar\'e series.
\end{enumerate}

Although the Poincar\'e series results are stated in characteristic zero,
the numerical formulas hold in every sufficiently large characteristic
for fixed matrix dimensions and~\(r\). The uniformity is striking:
the polynomial functors appearing in the minimal resolution of the residue
field have unbounded degree, yet a single characteristic bound gives the
entire Poincar\'e series, not just a finite truncation. The characteristic
restrictions come instead from the finite-dimensional Tor algebra~\(T\)
and its finite straightening model, where only finitely many partitions
enter the bounds of Remark~\ref{rmk:general-rank-straightening}.
Once \(T\) is strand-Koszul, the spectral-sequence argument controls all
homological degrees at once.

\subsection{Related work}

The determinantal rings above have well-known geometric incarnations:
\(\Spec R\) is the variety of matrices of rank less than \(r\) (rank
less than \(2r\) in the alternating case), its \defi{rank locus}. After
base change to an algebraic closure, the corresponding projective
schemes \(\operatorname{Proj}(R\otimes_\bk\overline{\bk})\) are,
respectively, the secant varieties
\[
\sigma_{r-1}\!\left(
  \operatorname{Seg}(\mathbf P(V^*)\times\mathbf P(W^*))\right),
\qquad
\sigma_{r-1}\!\left(\nu_2(\mathbf P(V^*))\right),
\qquad
\sigma_{r-1}\!\left(\Gr(2,V^*)\right)
\]
of the Segre variety, the quadratic Veronese variety, and the
Grassmannian in its Pl\"ucker
embedding~\cite{ManivelMichalek2015}, so Theorem~\ref{thm:intro-main}
computes the Poincar\'e series of the local ring at the vertex of the
affine cone over each of these varieties.

Avramov asked whether the Koszul homology algebra of a Koszul algebra is
generated by its linear strand
\(\bigoplus_i\Tor_i^S(R,\bk)_{i+1}\). This is known to fail in
general~\cite[Theorem~3.15]{BoocherDAliGrifoMontanoSammartano2015}, so
generation by the linear strand is a genuinely special property;
Conca, Katth\"an, and Reiner proved that it holds for the second
Veronese and asked if it holds for Segre products of polynomial
rings~\cite[Question~1.1 and Corollary~2.4]{ConcaKatthanReiner2021}.
When \(r=2\), the
generic determinantal ring is the Segre product of two polynomial rings
and the strand grading is \(j-i\). Thus the generation statement in
Theorem~\ref{thm:koszul-strands} gives an affirmative answer in
characteristic zero for these Segre products while recovering their
result for the second Veronese.

\subsection{Organization}

The paper is organized as follows. In Section~\ref{sec:background} we fix
the partition conventions. In Sections~\ref{sec:hook-algebra}--\ref{sec:elim-systems} we introduce the
universal hook algebra and its term order, develop the hook Pieri maps
and the associated elimination systems, and prove the quadratic
straightening law and Koszulity (Theorems~\ref{thm:main-hook}
and~\ref{thm:quad-gb}). Finally, in Section~\ref{sec:applications} we
prove strand generation, construct the three filtered models, and derive
the Poincar\'e series formulas (Theorem~\ref{thm:koszul-strands} and
Corollary~\ref{cor:poincare-applications}), concluding with the geometric
and representation-theoretic consequences. Essentially every
construction in this paper is implemented in the Macaulay2 package
\texttt{HookPieri}~\cite{HookPieri}.

\section{Background}\label{sec:background}

We work over a field~\(\bk\), with all vector spaces finite-dimensional. 

\subsection{Partitions and Schur Functors}\label{subsec:partitions-schur}

A \defi{partition} \(\lambda=(\lambda_1\ge\cdots\ge\lambda_\ell>0)\)
is identified with its Young diagram \(\{(i,j)\in\mathbb Z_{>0}^2:\, j\le\lambda_i\}\), drawn in the English
convention.  The \defi{conjugate} partition is \(\lambda^{\mathsf t}\), obtained by transposing the diagram.  We write
\(|\lambda|:=\sum\lambda_i\) for the size and
\(\ell(\lambda)\) for the number of parts.

\begin{definition}\label{def:frobenius-coords}
The \defi{rank} of~\(\lambda\), denoted \(\rank(\lambda)\), is the
side length of the largest square fitting inside its diagram (the
\defi{Durfee square}).  The \defi{Frobenius coordinates} of~\(\lambda\)
are
\[
\lambda\;\longleftrightarrow\;
(a_1,\dots,a_d\mid b_1,\dots,b_d),
\qquad d:=\rank(\lambda),
\]
where \(a_i:=\lambda_i-i\) and
\(b_i:=\lambda_i^{\mathsf t}-i\) are the arm and leg lengths along
the main diagonal.  Thus \(a_1>\cdots>a_d\ge 0\) and
\(b_1>\cdots>b_d\ge 0\).

The \defi{\(i\)th diagonal hook} of~\(\lambda\) is the hook-shaped subdiagram
\[
H_i(\lambda):=
\bigl(\{(i,j):j\ge i\}\cup\{(j,i):j\ge i\}\bigr)\cap\lambda,
\]
which has shape \((a_i{+}1,1^{b_i})\).  These hooks tile the
diagram:
\(\lambda=H_1(\lambda)\sqcup\cdots\sqcup H_d(\lambda)\).
\end{definition}

  For instance, \(\lambda=(4,3,2)\) has rank \(d=2\) and Frobenius coordinates
\((a_1,a_2\mid b_1,b_2)=(3,1\mid2,1)\). Its two shaded diagonal cells anchor
the arms extending to the right and the legs extending downward (see below, left). The first diagonal hook \(H_1\) has shape \((4,1,1)\) and is anchored at \((1,1)\), while \(H_2\) has shape \((2,1)\)
and is anchored at \((2,2)\) (see below, right):
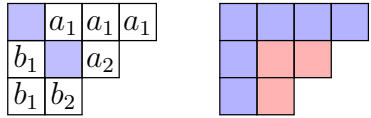
\[
\ytableausetup{boxsize=1.15em}
\begin{ytableau}
*(blue!25) & a_1 & a_1 & a_1 \\
b_1 & *(blue!25) & a_2 \\
b_1 & b_2
\end{ytableau}
\qquad
\begin{ytableau}
*(blue!30) & *(blue!30) & *(blue!30) & *(blue!30) \\
*(blue!30) & *(red!30) & *(red!30) \\
*(blue!30) & *(red!30)
\end{ytableau} 
\]

\begin{definition}\label{def:frobenius-union}
Two partitions \(\lambda\leftrightarrow(a_1,\dots,a_r\mid
b_1,\dots,b_r)\) and \(\mu\leftrightarrow(c_1,\dots,c_s\mid d_1,\dots,d_s)\) are
\defi{mergeable} if their combined diagonal hooks nest: the arm lengths are distinct, and when the hooks
\((a_i\mid b_i)\) and \((c_j\mid d_j)\) are listed in decreasing order of arm, their legs strictly decrease as well. In that case the \defi{Frobenius union}
\(\lambda\sqcup\mu\) is the partition with exactly these diagonal hooks.
Each hook stays intact: an arm \(a_i\) keeps its own leg \(b_i\); arms
and legs are not rearranged separately.
\end{definition}

For example,
$\ytableausetup{boxsize=.6em}
\underbrace{
\begin{ytableau}
*(blue!35) & *(blue!35) & *(blue!35) \\
*(blue!35) \\
*(blue!35)
\end{ytableau}}_{(3,1,1)}
\;\sqcup\;
\underbrace{
\begin{ytableau}
*(red!35) & *(red!35) \\
*(red!35) & *(green!50)
\end{ytableau}}_{(2,2)}
\;=\;
\underbrace{
\begin{ytableau}
*(blue!35) & *(blue!35) & *(blue!35) \\
*(blue!35) & *(red!35) & *(red!35) \\
*(blue!35) & *(red!35) & *(green!50)
\end{ytableau}}_{(3,3,3)}$, but
$\underbrace{
\begin{ytableau}
{} & *(yellow!70)
\end{ytableau}}_{(2)}
\;\not\sqcup\;
\underbrace{
\begin{ytableau}
{} & *(yellow!70) \\
{}
\end{ytableau}}_{(2,1)}$.

\begin{definition}\label{def:row-insertion}
For a partition~\(\lambda\) of rank~\(d\) and an integer \(r\ge 0\),
define \(\lambda^{(r)}\) by inserting \(r\) rows of length~\(d\)
immediately below the Durfee square:
\[
\lambda^{(r)}
:=(\lambda_1,\dots,\lambda_d,\,d^r,\,\lambda_{d+1},\dots).
\]
If \(\lambda\leftrightarrow(\underline a\mid b_1,\dots,b_d)\), then
\(\lambda^{(r)}\leftrightarrow
(\underline a\mid b_1+r,\dots,b_d+r)\); in particular,
\(\varnothing^{(r)}=\varnothing\).
\end{definition}

For example, \(\lambda=(3,2,1)\) has rank \(d=2\), and \(\lambda^{(1)}\) inserts one row of length \(2\) immediately below the Durfee square, giving \(\lambda^{(1)}=(3,2,2,1)\):
\[
\ytableausetup{boxsize=.6em}
\begin{ytableau}
{} & {} & {} \\
{} & {} \\
{}
\end{ytableau}
\quad\xrightarrow{\;(-)^{(1)}\;}\quad
\begin{ytableau}
{} & {} & {} \\
{} & {} \\
*(yellow!50) & *(yellow!50) \\
{}
\end{ytableau}.
\]

\begin{definition}\label{def:Q-sets}
For \(\delta\in\mathbb Z\) and \(d\ge 0\), let
\[
Q_\delta^d
:=
\bigl\{
(a_1,\dots,a_d\mid b_1,\dots,b_d):\,
a_j-b_j=\delta\;\text{for all }j
\bigr\}
\]
denote the set of rank \(d\) partitions whose diagonal hooks all satisfy
\(\text{arm length}-\text{leg length}=\delta\). Set
\(Q_\delta:=\bigcup_{d\ge 0}Q_\delta^d\) and
\(Q_\delta^{\mathrm{ev}}:=\bigcup_{d\;\mathrm{even}}Q_\delta^d\).

Thus \(Q_0\) consists of self-conjugate partitions. Consequently,
\(|\lambda|-\rank(\lambda)\) is even on \(Q_0\), while \(|\lambda|\) is even
on \(Q_{-1}\).
\end{definition}

For a partition~\(\lambda\), write \(\bS_\lambda\) for the Schur functor;
see~\cite[\S2.1]{Weyman2003} and~\cite{AkinBuchsbaumWeyman1982}. In
characteristic zero, the nonzero modules \(\bS_\lambda(V)\) form a complete
set of irreducible polynomial \(\GL(V)\)-representations; in arbitrary
characteristic they are the universal Schur modules of
\cite{AkinBuchsbaumWeyman1982} and need not be irreducible. Note that
\(\bS_\lambda(V)=0\) if and only if \(\ell(\lambda)>\dim V\).

A \defi{semistandard Young tableau} (SSYT) of shape~\(\lambda\) and
entries in~\(\{1,\dots,n\}\) is a filling of the Young diagram
of~\(\lambda\) which is weakly increasing along rows and strictly
increasing down columns~\cite{Macdonald1995, Fulton1997}.  When \(\dim V=n\), the SSYTs of
shape~\(\lambda\) index a weight basis of~\(\bS_\lambda(V)\) (see~\cite{AkinBuchsbaumWeyman1982}) with respect to a choice of maximal torus.

Finally, we will refer to the Littlewood--Richardson rule throughout, which computes the multiplicity of a Schur functor in the tensor product of two other Schur functors. We specifically use the version of this rule which counts Littlewood--Richardson tableaux. For details, see \cite[\S 5.2, Proposition 3]{Fulton1997}.

\subsection{Poincar\'e Series}

Fix a group $G$. For a finite-dimensional $G$-module $W$, let $[W]$ denote its class in the Grothendieck ring of finite-dimensional $G$-representations.  In our applications, $\operatorname{char}\bk=0$ and $G$ is a product of general linear groups.  If $G$ is trivial, then $[W]=\dim_\bk W$.

Let $R = \bigoplus_{d\in\mathbb Q_{\ge 0}} R_d$ be a connected ($R_0=\bk$) finitely generated graded $\bk$-algebra on which $G$ acts by graded algebra automorphisms.  Set $R_+ = \bigoplus_{d > 0} R_d$. We refer to this grading as the \defi{internal grading} and define the \defi{equivariant Poincar\'e series}
\[
  \rP_R(t,u) = \sum_{i\in\mathbb Z_{\ge 0},\,d\in\mathbb Q_{\ge 0}} [\Tor_i^R(\bk,\bk)_d] \cdot t^i u^d,
\]
where $t$ records homological degree and $u$ internal degree; here $\bk=R/R_+$.

Let $A\to B$ be a $G$-equivariant homomorphism of connected, finitely generated $\mathbb Q_{\ge0}$-graded $\bk$-algebras. Assume that $A$ is generated in degree~$1$ and that $B$ is finite as an $A$-module.  Set
\[
  T_{i,j}=\Tor_i^A(B,\bk)_j,
  \qquad T=\bigoplus_{i,j}T_{i,j}.
\]
Assume that some nonzero $T_{i,j}$ has $j-i>0$, and let $\ell>0$ be the least such difference.  Giving $T_{i,j}$ degree $(j-i)/\ell$ defines the \defi{strand grading}, a priori a rational grading. If $T$ is Koszul with respect to this grading, set
\[
  H(x,y) = \sum_{i,d\ge 0} [T_{i,\,i+\ell d}] \cdot x^i y^d.
\]

We say that \(A\) is \defi{Koszul}~\cite{Priddy1970,PolishchukPositselski2005} if it is generated in a single degree, and after rescaling this degree to 1, we have
\[
\Tor_i^A(\bk,\bk)_j=0\qquad\text{for all }i\ne j,
\]
where \(i\) is homological degree and \(j\) is internal degree. Equivalently,
\(\bk\) admits a linear graded free resolution over \(A\). 

\begin{proposition}\label{prop:em-rational-form}
  Suppose that $A$ is Koszul, that $\ell\ge1$, and that $T$ is Koszul with respect to the strand grading.
  Then
  \[
    \rP_{B}(t,u)  = \frac{\rP_A(t,u)}{H(tu,-tu^\ell)}.
  \]
\end{proposition}

\begin{proof}  
  The bar construction gives the Eilenberg--Moore homology spectral sequence \cite[Proposition 3.2.4]{Avramov1998Infinite}; see also~\cite{GugenheimMay1974}:
  \[
    \rE^2_{p,q,d} = \Tor_p^T(\Tor^A(\bk,\bk), \bk)_{q,d}
    \Longrightarrow \Tor^{B}_{p+q}(\bk,\bk)_d.
  \]
  The subscripts \(q,d\) record the Tor  homological and internal gradings.
  
  Set $M:=\Tor^A(\bk,\bk)$.  Since $A$ is Koszul, $M_{i,j}\ne0$ only when $i=j$.  Moreover,
  \[
    T_{i,j}\cdot M_{q,q}\subseteq M_{i+q,j+q}.
  \]
  Thus only the part of $T$ of strand degree zero can act nontrivially on
  $M$; this part is $\bk$ because the strand grading is connected. Hence $M$
  is a trivial $T$-module and
  \[
    \Tor^T_p(M,\bk)_{q,d}    =\bigoplus_{q'}\Tor_{q'}^A(\bk,\bk)_{q'}      \otimes\Tor^T_p(\bk,\bk)_{q-q',d-q'}.
  \]

  Koszulity with respect to the strand grading gives
  \[
    \Tor^T_p(\bk,\bk)_{q,d}\ne0
    \quad\Longrightarrow\quad d=q+\ell p.
  \]
  Consequently, $\rE^2_{p,q}$ is concentrated in internal degree $q+\ell p$, as is every subquotient $\rE^r_{p,q}$.  The source and target internal degrees of
  \[
    d^r\colon\rE^r_{p,q}\longrightarrow\rE^r_{p-r,q+r-1}
  \]
  differ by $1+r(\ell-1)>0$.  Hence every higher differential vanishes.

  Degeneration gives a filtration on $\Tor_i^B(\bk,\bk)_d$ whose associated
  graded consists of the terms on the $\rE^2$ page. Therefore, in the
  Grothendieck ring,
  \[
    [\Tor_i^B(\bk,\bk)_d]
    =\sum_{j,k}[\Tor_k^A(\bk,\bk)_k]\cdot  [\Tor_j^T(\bk,\bk)_{i-j-k,d-k}].
  \]

  The Euler characteristic of the minimal free $T$-resolution of $\bk$, which
  is linear for the strand grading, gives
  \[
    H(x,y) \cdot \sum_{a,b} [\Tor^T_b(\bk,\bk)_{a,a+b\ell}] \cdot x^a (-y)^b  = 1.
  \]
  Substituting \((x,y)=(tu,-tu^\ell)\) and combining with the identity
  above proves the result.
\end{proof}

\begin{remark}\label{rem:formality-not-needed}
  \begin{enumerate}[\noindent (1)]
  \item A standard route to Eilenberg--Moore degeneration requires the derived
  fiber $R\otimes_S^{\mathrm L}\bk$, for $R=S/I$, to be formal as a
  differential graded algebra; together with Koszulity of $\Tor^S(R,\bk)$,
  this gives what Briggs, Cameron, Letz, and Pollitz call the
  \defi{Cohen Koszul} property~\cite{BriggsCameronLetzPollitz2025}. For
  $S=\Sym(\bigwedge^2\bk^6)$ and $I$ the ideal of $4\times4$ Pfaffians,
  the derived fiber is not formal, and Srinivasan proved that the minimal
  $S$-free resolution admits no differential graded algebra  structure~\cite{Srinivasan1992}. Our results include this algebra as a special case.

  \item Koszulity in a coarser regrading is also insufficient.  In characteristic~$0$, for
  \[
    R=\bk[x_1,\ldots,x_7]/(x_1x_2,x_2x_3,\ldots,x_6x_7,x_7x_1),
  \]
  the algebra $T=\Tor^{\bk[x_1,\ldots,x_7]}(R,\bk)$ is Koszul after each of
  its minimal algebra generators is assigned degree~$1$. It is not Koszul
  with respect to the strand grading: an indecomposable generator in
  bidegree $(5,7)$ has
  strand degree~$2$. Since $R$ is Koszul, the Eilenberg--Moore spectral
  sequence does not collapse at $\rE^2$
  by~\cite[Theorem 2.5.4]{Myers2021}.
  \end{enumerate}
  Thus the essential input is Koszulity specifically in the strand grading, which pins the internal degrees of all higher differentials.
\end{remark}

\section{Koszulity and the Universal Hook Algebra}\label{sec:hook-algebra}

We introduce the universal hook algebra \(\cH(V)\), whose multiplication is based on nesting diagonal hooks and is relevant for the study of Tor algebras of determinantal rings. We prove that it has a quadratic Gr\"obner basis with semistandard hook tabloids as standard monomials, and hence is Koszul. The straightening construction occupies \S\S\ref{sec:hook-schur}--\ref{sec:elim-systems};
\S\ref{sec:applications} transfers the result to the Tor algebras.

\subsection{Universal hook algebra}\label{ssec:universal-hook-algebra}

Throughout this section, assume \(\operatorname{char}\bk=0\).
Set
\[
\bV:=\bigoplus_{a,b\ge 0}\bS_{(a+1,1^b)}(V).
\]
We call \((a,b)\) the \defi{hook type} of the hook shape
\((a+1,1^b)\). For each hook type, choose a parity
\(\eps(a,b)\in\bZ/2\), making \(\bV\) a \(\bZ/2\)-graded vector space,
and let \(\Sym_\eps(\bV)\) denote its supersymmetric algebra. On hook types, use the strict product order
\[
(a,b)>(c,d)\quad\Longleftrightarrow\quad a>c\ \text{ and }\ b>d.
\]
Visually, this distinguishes nestable and non-nestable hook pairs:
\[
\ytableausetup{boxsize=0.6em}
\begin{array}{c@{\qquad\qquad}c}
\begin{ytableau}
*(blue!30) & *(blue!30) & *(blue!30) \\
*(blue!30) \\
*(blue!30) \\
*(blue!30)
\end{ytableau}
\;\;,\;\;
\begin{ytableau}
*(red!30) & *(red!30) \\
*(red!30) \\
*(red!30)
\end{ytableau}
&
\begin{ytableau}
*(blue!30) & *(blue!30) & *(blue!30) \\
*(blue!30) \\
*(blue!30) \\
*(blue!30)
\end{ytableau}
\;\;,\;\;
\begin{ytableau}
*(red!30) & *(red!30) \\
*(red!30) \\
*(red!30) \\
*(red!30)
\end{ytableau}
\\[1.5ex]
\text{strictly comparable: }(2,3)>(1,2) &
\text{incomparable: }(2,3),(1,3)\text{ share leg}
\end{array}
\]
Let $q_\eps\colon \bV\otimes\bV\twoheadrightarrow\Sym_\eps^2(\bV)$
be the canonical quotient. Define \(\mathbb H_\eps(V)\) to be the quotient
of \(\Sym_\eps(\bV)\) by the quadratic ideal generated by
\[
q_\eps\bigl(\bS_{(a+1,1^b)}(V)\otimes
\bS_{(c+1,1^d)}(V)\bigr), \qquad (a,b), (c,d) \text{ incomparable}.
\]
We write \(\mathbb H(V):=\mathbb H_0(V)\) for the purely even choice.

For \(a>c\) and \(b>d\), let \(\lambda\) have Frobenius coordinates
\((a,c\mid b,d)\). By the Littlewood--Richardson rule, \(\bS_\lambda(V)\) occurs with multiplicity one in
\(\bS_{(a+1,1^b)}(V)\otimes\bS_{(c+1,1^d)}(V)\). Set
\[
R_{(a,b),(c,d)}
:= \ker \Bigl(
\bS_{(a+1,1^b)}(V)\otimes\bS_{(c+1,1^d)}(V)
\twoheadrightarrow \bS_\lambda(V) \Bigr)
\]
where the displayed equivariant projection is unique up to a nonzero scalar; thus its kernel is independent of the choice.

\begin{example}
$\bS_{(2,1)}(V)\otimes V
\;=\;
\bS_{(2,2)}(V)
\;\oplus\;
\underbrace{\bS_{(3,1)}(V)\,\oplus\,\bS_{(2,1,1)}(V)}_{R_{(1,1),(0,0)}}.$
\end{example}

\begin{remark}
  The projection map used above is a hook analogue of the Cartan product $\bS_\lambda(V) \otimes \bS_\mu(V) \to \bS_{\lambda+\mu}(V)$ which comes geometrically from multiplying sections of line bundles on a flag variety. We suspect that it has a geometric interpretation related to super flag varieties.
\end{remark}

Let \(\bar q_\eps\) be the composite quotient map
\(\bV \otimes \bV \to \Sym_\eps^2(\bV)\to\mathbb H_\eps(V)_2\), and set
\[
I_\eps:=\big\langle \bar q_\eps\bigl(R_{(a,b),(c,d)}\bigr):a>c,\ b>d
\big\rangle\subseteq\mathbb H_\eps(V).
\]
The \defi{universal hook algebra} with parity \(\eps\) is
\[
\mathcal H_\eps(V):=\mathbb H_\eps(V)/I_\eps.
\]
We henceforth write \(I:=I_0\) and \(\mathcal H(V):=\mathcal H_0(V)\) in the purely even case.

\begin{definition}[Hook multigrading]\label{def:hook-multigrading}
Give the summand \(\bS_{(a+1,1^b)}(V)\) multidegree
\(\mathbf e_{a,b}\in\mathbb N^{(\mathbb N^2)}\), the standard basis vector
indexed by \((a,b)\). Both defining ideals are homogeneous for this grading,
so it descends to \(\mathcal H_\eps(V)\).
\end{definition}
Moreover, each
\(\mathcal H_\eps(V)\) is a graded Zhang twist of \(\mathcal H_0(V)\);
hence Koszulity with respect to the rank grading defined below is
independent of \(\eps\)~\cite{Zhang1996}. We
work with \(\eps=0\) for notational simplicity; in \S\ref{sec:applications} the parity is induced by
Tor homological degree.

Give \(\bS_{(a+1,1^b)}(V)\) bidegree \((1,a+b+1)\). Then the component
indexed by a Frobenius union \(\lambda\) has bidegree
\((\rank(\lambda),|\lambda|)\). The rank grading is connected and generated
in degree~\(1\); it is the internal grading below, while polynomial degree
is auxiliary.

\begin{theorem} \label{thm:main-hook}
  \begin{enumerate}[\rm (1)]
  \item As a bigraded \(\GL(V)\)-representation,
    \[
      \cH(V) \cong \bigoplus_\lambda \bS_\lambda(V),
      \qquad
      \deg\bS_\lambda(V)=\bigl(\rank(\lambda),|\lambda|\bigr),
    \]
    where \(\lambda\) ranges over all partitions.
  \item The algebra \(\cH(V)\) is Koszul with respect to the rank
  grading.
  \end{enumerate}
\end{theorem}

We first prove part~(1) in \S\ref{sec:term-orders}; its Hilbert function
consequence supplies the comparison in Theorem~\ref{thm:quad-gb}, whose
quadratic Gr\"obner basis implies part~(2). A complementary cohomological
description of the relation spaces will appear in
\cite{SamVandeBogertWeymanGiambelli}.

\begin{remark}
This construction also applies to symmetric group representations by a simple formal transformation. Namely, for $S_n$, we take $V=\bk^n$ and work exclusively with the weight spaces of weight $(1,1,\dots,1)$ with respect to the diagonal torus.
\end{remark}

\subsection{Term orders} \label{sec:term-orders}

We construct a Gr\"obner degeneration of \(\cH(V)\) to a quadratic monomial
algebra. Call a semistandard Young tableau of hook shape a
\defi{semistandard hook tableau}. Fix an ordered basis of \(V\), and for such
a tableau \(S\), let \(z_S\) denote the corresponding Schur functor basis
vector.

A monomial basis of \(\mathbb H(V)\) consists of products
\(z_T=z_{T_1}\cdots z_{T_r}\), where \(T_i\) is a semistandard tableau of
shape \((a_i+1,1^{b_i})\) and $(a_1,b_1)>\cdots>(a_r,b_r)$.
We call \(T=(T_1,\dots,T_r)\) a \defi{hook tabloid}. Equivalently, it is a
filling of the diagram with Frobenius coordinates
\((a_1,\dots,a_r\mid b_1,\dots,b_r)\) whose individual diagonal hooks are
semistandard, with no condition between distinct hooks.

For a semistandard hook tableau \(S\) of shape \((a+1,1^b)\), write
\(z_S=z^{(a,b)}_{I,\alpha}\), where \(I=\{i_0<\cdots<i_b\}\) is its first
column and \(\alpha\) is the weakly increasing word (possibly empty) of its remaining arm entries; all of them are \(\ge i_0\).

A hook tabloid is \defi{semistandard} if this filling is a semistandard
Young tableau, and \defi{nonsemistandard} otherwise.

\begin{theorem} \label{thm:quad-gb}
There is a term order on \(\mathbb H(V)\) for which \(I\) has a quadratic
Gr\"obner basis whose initial ideal is generated by the nonsemistandard
quadratic hook tabloids. Consequently, the semistandard hook tabloids form a
\(\bk\)-basis of \(\cH(V)\).
\end{theorem}

\begin{remark}
  The algebra $\mathbb{H}(V)$ behaves like an artinian analogue of an algebra with straightening law (ASL). The key point is that a basis for an ASL is given by products of generators $z_{i_1} \cdots z_{i_r}$ indexed by {\it weakly} decreasing chains: $z_{i_1} \ge \cdots \ge z_{i_r}$. We don't allow generators to be duplicated, so we only consider products over {\it strictly} decreasing chains.
\end{remark}

\begin{remark}\label{rmk:gorodentsov}
A related Koszulity theorem for finite nested hook collections is asserted in
the appendix to \cite{GorodentsovKhoroshkinRudakov2007}, but the proof has a significant mistake, which we illustrate with an example. For \(\dim V=3\), consider the Frobenius union projection
\[
\mu\colon \bS_{(3,1)}(V)\otimes V\twoheadrightarrow\bS_{(3,2)}(V).
\]
For
\[
\ytableausetup{boxsize=1.1em}
T=\ytableaushort{123,2},
\qquad
D=\ytableaushort{123,22},
\qquad
T'=\ytableaushort{122,23},
\]
the outer hook \(T\) and inner entry \(2\) assemble to \(D\). Its rows are
weakly increasing, but its second column is \((2,2)\), so classical column
straightening gives zero and the proposed lift in \cite[\S5.2.2, case~(B)]{GorodentsovKhoroshkinRudakov2007} predicts
\(\mu(z_T\otimes e_2)=0\). Writing \(v_{T'}\) for the corresponding
semistandard tableau basis vector and normalizing \(\mu\) on a semistandard
reference lift,
a \texttt{Macaulay2}~\cite{M2} computation with
\texttt{HookPieri}~\cite{HookPieri} instead gives
\[
\mu(z_T\otimes e_2)=-\tfrac12v_{T'}\ne0. \qedhere
\]
\end{remark}

To orient the straightening relations, we compare the cumulative contents of
successive initial segments of the diagonal hooks.

\begin{definition}[Running order]\label{def:running-order}
Let \(T=(T_1,\dots,T_r)\) be a hook tabloid of type
\(\mathbf h=\bigl((a_1,b_1),\dots,(a_r,b_r)\bigr)\), where \(T_t\)
has shape \((a_t{+}1,1^{b_t})\).

For a semistandard hook tableau~\(H\) of shape \((A{+}1,1^B)\) and \(0\le s\le A\),
let \(L_s(H)\) consist of its first column and first~\(s\) arm boxes, and set
\[
\pi_s(H):=
\bigl(\#\{u\in L_s(H):H(u)\le q\}\bigr)_{q=1}^{\dim V},
\qquad
\Pi(H):=\pi_A(H).
\]

For each \(1\le t\le r\) and \(0\le s\le a_t\), define the
\defi{running vector}
\[
V_{t,s}(T)
:=
\Pi(T_1)+\cdots+\Pi(T_{t-1})+\pi_s(T_t).
\]
Order the index set
\(\mathcal I_{\mathbf h}:=\{(t,s):1\le t\le r,\;0\le s\le a_t\}\)
lexicographically. Its elements are called \defi{running indices}, and
\(\{(t,s):0\le s\le a_t\}\) is the \(t\)th \defi{block}. For a box~\(u\)
in the \(t\)th hook, define
\[
\sigma(u):=
\begin{cases}
(t,0),&u\text{ lies in the first column},\\
(t,s),&u\text{ is the $s$th arm box from the left in the $t$th hook}.
\end{cases}
\]
For two tabloids \(T,U\) of type \(\mathbf h\), set
\(T\run U\) if all their running vectors agree, or if, at the largest pair
\((t_0,s_0)\) for which they differ,
\[
V_{t_0,s_0}(U)\trianglelefteq_{\mathrm{vec}}V_{t_0,s_0}(T),
\]
where \(\mathbf u\trianglelefteq_{\mathrm{vec}}\mathbf v\) means
\(u_q\le v_q\) for every~\(q\).
\end{definition}

For each fixed type, \(\run\) is a partial order. Indeed, it is the reverse
lexicographic extension of~\(\trianglelefteq_{\mathrm{vec}}\) on the running
vector sequences, and these sequences determine the tabloid:
\(V_{1,0}(T)\) and the differences
\(V_{t,0}(T)-V_{t-1,a_{t-1}}(T)\), for \(t>1\), determine the semistandard
first columns, while the differences \(V_{t,s}(T)-V_{t,s-1}(T)\), for \(s>0\),
determine the arm entries. Write \(T\strun U\) for its strict part:
\(T\run U\) and \(T\ne U\).

Thus the comparison inspects the running indices in decreasing
lexicographic order and regards the tabloid with the coordinatewise larger
vector as smaller. This orientation matches the triangularity below and is
stable under insertion
(Lemma~\ref{lem:single-hook-insertion}).

\begin{proposition} \label{prop:swap-small}
Suppose \(T,U\) are hook tabloids of the same type. Assume that \(T\) is
obtained from \(U\) by swapping entries \(x<y\) in boxes \(u,v\), where
\(U(u)=y\), \(U(v)=x\), and
\(\sigma(u)<_{\mathrm{lex}}\sigma(v)\). Then \(T\strun U\).
\end{proposition}

\begin{proof}
The running vectors differ only at indices between \(\sigma(u)\) and
\(\sigma(v)\). At the largest such index, the set of boxes counted by the
corresponding running vector contains \(x\) in \(T\) and \(y\) in \(U\), so
\[
V_{t,s}(T)_q-V_{t,s}(U)_q=
\begin{cases}
1,&x\le q<y,\\
0,&\text{otherwise}.
\end{cases}
\]
Thus \(V_{t,s}(U)\trianglelefteq_{\mathrm{vec}}V_{t,s}(T)\); since \(x<y\),
we have \(T\strun U\).
\end{proof}

\begin{example}
  Diagrammatically, the successive regions whose content is tracked by the
running vectors can be visualized as follows (shown for $\lambda=(3,2)$ with two diagonal hooks of types $(2,1),(0,0)$; the yellow region is the increment at each step, the blue region is the previously accumulated content):
\[
\makebox[\textwidth][c]{%
\begin{tikzpicture}[x=.50cm,y=.50cm]

\HookPanelThreeTwo{0}{0}
{$(1,0)$}
{}
{1/1,1/2}

\HookPanelThreeTwo{4.7}{0}
{$(1,1)$}
{1/1,1/2}
{2/1}

\HookPanelThreeTwo{9.4}{0}
{$(1,2)$}
{1/1,1/2,2/1}
{3/1}

\HookPanelThreeTwo{14.1}{0}
{$(2,0)$}
{1/1,1/2,2/1,3/1}
{2/2}

\draw[-{Latex[length=1.6mm]},thick] (3.55,-1) -- (4.15,-1);
\draw[-{Latex[length=1.6mm]},thick] (8.25,-1) -- (8.85,-1);
\draw[-{Latex[length=1.6mm]},thick] (12.95,-1) -- (13.55,-1);

\end{tikzpicture}%
}
\]

\begin{samepage}
For the semistandard hook tabloids
\[
T=\ytableaushort{112,23},
\qquad
U=\ytableaushort{113,22},
\]
of type \(((2,1),(0,0))\), the running vectors in coordinates
\(q=1,2,3\) are
\[
\begin{array}{r|c|c|c|c}
& V_{1,0} & V_{1,1} & V_{1,2} & V_{2,0} \\\hline
T & (1,2,2) & (2,3,3) & (2,4,4) & (2,4,5) \\
U & (1,2,2) & (2,3,3) & (2,3,4) & (2,4,5)
\end{array}
\]
Their last disagreement is at \((1,2)\), where
\(V_{1,2}(U)\trianglelefteq_{\mathrm{vec}}V_{1,2}(T)\); hence
\(T\run U\).
\end{samepage}
\end{example}

\subsection{From straightening to Koszulity}

Retain the presentation \(\mathcal H(V)=\mathbb H(V)/I\) from
\S\ref{ssec:universal-hook-algebra}.
The straightening argument has three ingredients.  Section~\ref{sec:hook-schur}
constructs explicit row and column Pieri maps in the two hook models;
Section~\ref{sec:inv-matrices} establishes invertibility of the required
elimination matrices;
and Section~\ref{sec:elim-systems} uses these results to straighten the first
failure of semistandardness.  The resulting relations in rank two are recorded in
Propositions~\ref{prop:rank2-local-row-elimination-ijk}
and~\ref{prop:rank2-local-column-elimination}.

\begin{proof}[Proof of Theorem~\ref{thm:main-hook}(1)]
Write
\(\lambda\leftrightarrow(a_1,\dots,a_r\mid b_1,\dots,b_r)\)
in Frobenius coordinates.  Put
\(h_t=(a_t,b_t)\), \(H_t=(a_t+1,1^{b_t})\), and
\(\mathbf h=(h_1>\cdots>h_r)\).  Consider the multiplication map
of~\(\cH(V)\),
\[
m_\lambda\colon \bS_{H_1}(V)\otimes\cdots\otimes
\bS_{H_r}(V)\longrightarrow \cH(V).
\]
The presentation in \S\ref{ssec:universal-hook-algebra} is homogeneous for the
multigrading by hook type, giving
\(\cH(V)=\bigoplus_\lambda\cH_\lambda(V)\) where
\(\cH_\lambda(V):=\im(m_\lambda)\) is spanned by the hook tabloids of
type~\(\mathbf h\).

\emph{Spanning.} The row and column elimination propositions
(Propositions~\ref{prop:rank2-local-row-elimination-ijk}
and~\ref{prop:rank2-local-column-elimination}), together with
Corollary~\ref{cor:mult-cancel}, show that, for every nonsemistandard hook
tabloid~\(z_T\) of type~\(\mathbf h\), there is a
relation
\[
z_T\equiv\sum_{T'\strun T}c_{T'}\,z_{T'}\pmod I,
\]
where the sum is over tabloids of the same type and content as~\(T\).
There are only finitely many tabloids of fixed type and content, so iteration
terminates at semistandard tabloids. These span~\(\cH_\lambda(V)\), so in particular, $\dim \cH_\lambda(V) \le \dim \bS_\lambda(V)$.

On the other hand, let \(c_\lambda\colon\bS_{H_1}(V)\otimes\cdots\otimes
\bS_{H_r}(V)\twoheadrightarrow\bS_\lambda(V)\) be the iterated
Frobenius union projection, where
\(\lambda=H_1\sqcup\cdots\sqcup H_r\).  The Littlewood--Richardson rule
gives multiplicity one for \(\bS_\lambda(V)\) both in the displayed tensor
product and after any pair \(H_i,H_j\) is replaced by
\(H_i\sqcup H_j\). For fixed \(i<j\), it follows that
\(\bS_\lambda(V)\) occurs only through the summand
\[
\bS_{H_i\sqcup H_j}(V)\subseteq
\bS_{H_i}(V)\otimes\bS_{H_j}(V).
\]
Thus \(c_\lambda\) annihilates
\[
R_{h_i,h_j}\otimes\!\bigotimes_{t\ne i,j}\bS_{H_t}(V).
\]
Hence the component of type~\(\mathbf h\),
\(I_{\mathbf h}:=I\cap\mathbb H(V)_{\mathbf h}\), lies in
\(\ker c_\lambda\), and \(c_\lambda\) induces a surjection
\(\cH_\lambda(V)\twoheadrightarrow\bS_\lambda(V)\). The previous bound implies this is an isomorphism.
\end{proof}

\begin{lemma}[Insertion invariance]\label{lem:single-hook-insertion}
Let \(T,U\) be hook tabloids of common type
\(\mathbf h=(h_1>\cdots>h_r)\), where \(h_t=(a_t,b_t)\), and let \(G\) be a
semistandard hook tableau insertable at position~\(p\), where
\(1\le p\le r+1\). If \(h_G\) is the hook type of~\(G\), this means
\[
\begin{cases}
h_G>h_1,&p=1,\\
h_{p-1}>h_G>h_p,&2\le p\le r,\\
h_r>h_G,&p=r+1.
\end{cases}
\]
Write \(G\star_pT\) for the tabloid with~\(G\) inserted.
Then
\[
T\run U\quad\Longleftrightarrow\quad G\star_pT\run G\star_pU.
\]
\end{lemma}

\begin{proof}
Set
$E_p(X):=V_{p-1,a_{p-1}}(X)$ if $p>1$ and 0 if $p=1$.
With \(\pi_s\) and \(\Pi\) as in Definition~\ref{def:running-order}, we have
\[
V_{t,s}(G\star_pT)=
\begin{cases}
V_{t,s}(T),&t<p,\\
E_p(T)+\pi_s(G),&t=p,\\
V_{t-1,s}(T)+\Pi(G),&t>p.
\end{cases}
\]
and so
\[
V_{t,s}(G\star_pT)-V_{t,s}(G\star_pU) = 
\begin{cases}
  V_{t,s}(T)-V_{t,s}(U) 
  ,&t<p,\\
  E_p(T)-E_p(U) 
  ,&t=p,\\
V_{t-1,s}(T)-V_{t-1,s}(U)
,&t>p.
\end{cases} 
\]
Thus insertion adds a block of copies of~\(E_p(T)-E_p(U)\) to the difference
sequence. The last nonzero difference
may shift, but neither its value nor its comparison under
\(\trianglelefteq_{\mathrm{vec}}\) changes, proving the equivalence.
\end{proof}

\begin{corollary}[Multiplicativity and cancellation]\label{cor:mult-cancel}
If \(TZ\) and \(UZ\) are both defined, then
\(T\run U\Longleftrightarrow TZ\run UZ\).
\end{corollary}

The row straightening argument in Proposition~\ref{prop:rank2-local-row-elimination-ijk} also requires stability under adjoining common arm entries.

\begin{lemma}\label{lem:rQ-monotone}
Let \(T\run U\), fix a hook position~\(p\), and let
\(\delta\in\mathbb N^{\dim V}\).  Suppose that adjoining the multiset of
content~\(\delta\) to the \(p\)th arm of each tabloid and reordering the arms
produces hook tabloids \(T^+\) and~\(U^+\).  Then \(T^+\run U^+\).
\end{lemma}

\begin{proof}
Differences in blocks before and after~\(p\) are unchanged.  Write
\(D_s:=V_{p,s}(T)-V_{p,s}(U)\), and let
\(x_1\le\cdots\le x_a\) and \(y_1\le\cdots\le y_a\) be the two arms.
If the largest block containing a nonzero difference lies after~\(p\), its
largest nonzero difference is unchanged. If it lies before~\(p\), every
difference in block~\(p\) vanishes. Thus
\(V_{p,0}(T)=V_{p,0}(U)\), and the two sorted arms coincide, so extension
introduces no difference in block~\(p\).

We may therefore assume that the
largest block containing a nonzero difference is~\(p\).
If \(D_a\ne0\), the final extended arm stage has the same difference, so
the comparison is unchanged.  Suppose \(D_a=0\).  If the arms coincide,
every difference in this block is zero, contrary to our assumption. Set \(k=\max\{i:x_i\ne y_i\}\). Cancellation along the common suffix gives
\[
D_k=0,
\qquad
(D_{k-1})_q=
\begin{cases}
1,&y_k\le q<x_k,\\
-1,&x_k\le q<y_k,\\
0,&\text{otherwise},
\end{cases}
\qquad 1\le q\le\dim V.
\]
Thus \(T\run U\) precisely when \(y_k<x_k\): the comparison is reverse
lexicographic on the sorted arms.  Adjoining a common multiset preserves
that comparison, since common multiplicities cancel at the largest label
where the two multisets differ.
\end{proof}

We now refine the running order to a term order compatible with
multiplication. For
each hook type introduce a symbol~\(X_{a,b}\), of weight \(a+b+1\), and
fix the weighted degree lexicographic monomial order
\(\prec_{\mathrm{type}}\) on the free commutative monoid they generate;
order variables first by weight and then by~\(a\).  For a tabloid~\(T\), set
\[
\tau(T):=\prod_{t=1}^rX_{a_t,b_t},
\]
where \(((a_1,b_1),\ldots,(a_r,b_r))\) is its type.
If \(T,U\) have the same type~\(\mathbf h\), declare
\(T\prec_{\mathbf h}U\) when, at the largest running index where they differ,
\[
V_{t,s}(T)>_{\mathrm{lex}}V_{t,s}(U).
\]
Finally, order the nonzero monomials of~\(\mathbb H(V)\) by
\begin{equation}\label{eq:global-tabloid-order}
T\prec U
\quad\Longleftrightarrow\quad
\tau(T)\prec_{\mathrm{type}}\tau(U)
\quad\text{or}\quad
\bigl(\tau(T)=\tau(U)\ \text{and}\ T\prec_{\mathbf h}U\bigr).
\end{equation}

\begin{proposition}\label{prop:global-tabloid-order}
The order~\(\prec\) is a term order on~\(\mathbb H(V)\).  For distinct
tabloids of fixed type, \(T\run U\) implies \(T\prec U\).
\end{proposition}

\begin{proof}
Since the running vectors distinguish tabloids,
\(\prec_{\mathbf h}\) is total; it refines~\(\run\) because lexicographic
order refines coordinatewise dominance.  The calculation in
Lemma~\ref{lem:single-hook-insertion}, with
lexicographic comparison in place of coordinatewise comparison, gives
\(T\prec_{\mathbf h}U\Longleftrightarrow
G\star_pT\prec_{\mathbf h'}G\star_pU\), where~\(\mathbf h'\) is the common
type after insertion.  Together
with multiplicativity of~\(\prec_{\mathrm{type}}\), this proves
multiplicativity and cancellation for~\(\prec\).  Finally, the positive
weights, together with the fact that only finitely many hook types have
bounded weight, make
\(\prec_{\mathrm{type}}\) a well-order, and there are only finitely many
tabloids of any fixed type; hence~\(\prec\) is well-founded.
\end{proof}

The three comparisons are summarized as follows:
\begin{center}
\renewcommand{\arraystretch}{1.3}
\begin{tabular}{l|l|l}
Symbol & Scope & Comparison \\\hline
$\run$ & fixed type & coordinatewise at the largest differing index \\
$\prec_{\mathbf h}$ & fixed type & lexicographic at the largest differing index \\
$\prec$ & all nonzero monomials & type monomial, then $\prec_{\mathbf h}$
\end{tabular}
\end{center}

\begin{proposition}[Initial ideal]
\label{prop:quadratic-grobner-hook}
Let \(\prec\) be the monomial order in~\eqref{eq:global-tabloid-order}.
Then:
\begin{enumerate}
\item For each nonsemistandard quadratic hook tabloid
\(m=z^{(a,b)}_{I,\alpha}\,z^{(c,d)}_{J,\beta}\) with \((a,b)>(c,d)\),
there is a relation
\(f_m=m-\sum_{m'\prec m}c_{m,m'}\,m'\in I\), with each~\(m'\)
semistandard.
\item
\(\operatorname{in}_\prec(I)=J_{\mathrm{nss}}\). This is the monomial ideal
generated by the nonsemistandard quadratic tabloids.
In particular, \(\{f_m\}\) is a quadratic Gr\"obner basis.
\end{enumerate}
\end{proposition}

\begin{proof}
For~(1), apply the rank two case of the spanning argument in the proof of
Theorem~\ref{thm:main-hook}(1); Proposition~\ref{prop:global-tabloid-order}
converts each strict running comparison into a strict comparison under~\(\prec\).
For~(2), each
\(f_m\) has leading term~\(m\), so
\(J_{\mathrm{nss}}\subseteq\operatorname{in}_\prec(I)\).  The standard
monomials for~\(J_{\mathrm{nss}}\) are precisely the semistandard hook
tabloids: a hook tabloid is semistandard if and only if its restriction to
every pair of diagonal hooks is semistandard.  By
Theorem~\ref{thm:main-hook}(1), their residue
classes form a basis of~\(\cH(V)\).  Thus \(J_{\mathrm{nss}}\) and
\(\operatorname{in}_\prec(I)\) have the same multigraded Hilbert function,
and the preceding inclusion is an equality.
\end{proof}

Together with the quadratic monomials involving incomparable hook types, which
already vanish in \(\mathbb H(V)\), the \(f_m\) form a quadratic Gr\"obner
basis for \(\mathcal H(V)\). Thus \(\mathcal H(V)\) is
Koszul~\cite[Theorem~4.3.12]{PolishchukPositselski2005}, proving
Theorem~\ref{thm:main-hook}(2).

\medskip
We conclude with the Koszul retracts used in the Tor applications.  For
\(\delta\in\mathbb Z\), set
\[
\cH^{(\delta)}(V):=
\bigoplus_{\lambda\in Q_\delta}\cH_\lambda(V),
\]
where \(Q_\delta\) is defined in Definition~\ref{def:Q-sets}, and let
\(r_\delta\) be the projection onto these type components.  All Segre
products below use the rank grading:
\[
A\circ B:=\bigoplus_{n\ge0}A_n\otimes_\bk B_n.
\]

\begin{proposition}\label{prop:diagonal-retracts}
Fix \(\delta,\delta'\in\mathbb Z\).  Let
\(\cK^{(\delta,\delta')}(V)
:=\cH^{(\delta)}(V)\circ\cH^{(\delta')}(V)\).
Then:
\begin{enumerate}
\item[\textup{(i)}]
\(r_\delta\colon\cH(V)\to\cH^{(\delta)}(V)\) is an algebra retraction,
so \(\cH^{(\delta)}(V)\) is Koszul.
\item[\textup{(ii)}]
Writing \(r_{\delta,n}:=r_\delta|_{\cH(V)_n}\), the map
\[
r_{\delta,\delta'}:=\bigoplus_{n\ge0}
r_{\delta,n}\otimes r_{\delta',n}
\]
is an algebra retraction from \(\cH(V)\circ\cH(V)\) onto
\(\cK^{(\delta,\delta')}(V)\), which is therefore Koszul.
\end{enumerate}
\end{proposition}

\begin{proof}
If \(\lambda\) and \(\mu\) are mergeable in the sense of
Definition~\ref{def:frobenius-union}, then
\(\lambda\sqcup\mu\in Q_\delta\) if and only if
\(\lambda,\mu\in Q_\delta\); nonmergeable products vanish.  Hence
\(r_\delta\) is multiplicative.  The same argument on homogeneous tensors
proves the assertion for~\(r_{\delta,\delta'}\).  Finally, algebra retracts and
ordinary Segre products of Koszul algebras graded by rank are Koszul
\cite{BackelinFroberg1985}.
\end{proof}

\section{Hook Schur Functors and Pieri Maps} \label{sec:hook-schur}

This section defines the hook Pieri maps \(\rP_m,\rQ_m,\rP'_m,\rQ'_m\) and the compositions \(\Phi,\Phi',\Psi,\Psi'\) used in the arguments of \S\S\ref{sec:inv-matrices}--\ref{sec:elim-systems} and derives explicit formulas for them.

\smallskip\noindent\textbf{Notation.}
Throughout, fix \(V=\bk^n\) with ordered basis
\(e_1<\cdots<e_n\) and write \([n]:=\{1,\dots,n\}\).
Recall from \S\ref{ssec:universal-hook-algebra} that the hook shape
\((a,1^b)\) has hook type \((a-1,b)\).
Every normalized map below is used only when its displayed denominator is a
unit in~\(\bk\); this is automatic under the characteristic hypotheses of
\S\S\ref{sec:inv-matrices}--\ref{sec:elim-systems}.

For \(\alpha\in\bN^n\), set
\(e^\alpha:=\prod_r e_r^{\alpha_r}\in S^{|\alpha|}V\); we use
\(\eps_i\) for the \(i\)th standard basis vector in exponents, so
\(e^{\eps_i}=e_i\). If \(\beta\le\alpha\) coordinatewise, we set
\[
\binom{\alpha}{\beta}:=\prod_{i=1}^n\binom{\alpha_i}{\beta_i},
\qquad
\supp(\alpha):=\{i:\alpha_i\ne0\}.
\]
For \(J\subseteq[n]\), write \(\mathbf1_J:=\sum_{j\in J}\eps_j\), and write \(\delta_{ij}\) for the Kronecker delta.

For \(I=\{i_1<\cdots<i_q\}\), set
\(e_I:=e_{i_1}\wedge\cdots\wedge e_{i_q}\in\bigwedge^qV\).
For \(J\subseteq I\), define \(\sgn(J,I)\) by
\[
e_I=\sgn(J,I)\,e_J\wedge e_{I\setminus J}.
\]
We suppress braces around singletons in such expressions; in particular,
\(I\setminus i:=I\setminus\{i\}\) and
\(\sgn(i,I):=\sgn(\{i\},I)\).

For the dual basis \(e_1^\vee,\dots,e_n^\vee\), define the operators (components of derivations)
\begin{align*}
\iota_{e_i^\vee}\colon\bigwedge^qV&\longrightarrow\bigwedge^{q-1}V,
&e_I&\longmapsto\sgn(i,I)e_{I\setminus i},\\
\partial_{e_i^\vee}\colon S^qV&\longrightarrow S^{q-1}V,
&e^\alpha&\longmapsto\alpha_i e^{\alpha-\eps_i},
\end{align*}
where the displayed values are understood as zero when \(i\notin I\) and
\(\alpha_i=0\), respectively.

\subsection{Hook Schur Functors}

We use the characteristic-free convention for Schur modules of
Akin--Buchsbaum--Weyman~\cite{AkinBuchsbaumWeyman1982}; see also
Weyman~\cite[\S2]{Weyman2003}.

For \(p,q\in\bZ\), set
\[
C_{p,q}:=
\begin{cases}
S^p(V)\otimes\bigwedge^qV,&p,q\ge0,\\
0,&\text{otherwise}.
\end{cases}
\]
For \(p+q>0\), the \defi{Koszul differential} and, when
\(p+q\in\bk^\times\), the \defi{normalized contraction}
\[
d_{p,q}\colon C_{p-2,q+2}\longrightarrow C_{p-1,q+1},
\qquad
h_{p,q}\colon C_{p+1,q-1}\longrightarrow C_{p,q}
\]
are defined by
\[
d_{p,q}(f\otimes\omega)
=\sum_i(e_if)\otimes\iota_{e_i^\vee}(\omega),
\qquad
h_{p,q}(g\otimes\eta)
=\frac{1}{p+q}\sum_i(\partial_{e_i^\vee}g)\otimes(e_i\wedge\eta).
\]
For \(a\ge1\) and \(b\ge0\), the standard hook presentation gives the
first isomorphism below; if \(a+b\in\bk^\times\), the normalized Cartan
identity gives the second:
\[
\bS_{(a,1^b)}(V)
\;\cong\;
\coker\,d_{a,b}
\;=\;
\frac{S^{a-1}V\otimes\bigwedge^{b+1}V}{\im(d_{a,b})}
\;\cong\;
\coker\,h_{a,b}
\;=\;
\frac{S^{a}V\otimes\bigwedge^{b}V}{\im(h_{a,b})}.
\]

We refer to these two cokernels as the \defi{\(d\)-model} and the \defi{\(h\)-model} for~\(\bS_{(a,1^b)}(V)\), and write \([\,\cdot\,]_d\) and
\([\,\cdot\,]_h\) for the corresponding equivalence classes.

The tableau notation below always denotes the \(d\)-model class. We pass to the \(h\)-model only to calculate some explicit formulas. Write
\[
\ytableausetup{boxsize=1.7em}
T^\alpha_I \;:=\;  [e^\alpha\otimes e_I]_d \;=\;
\vcenter{\hbox{$
\begin{ytableau}
\scriptstyle i_1 & \scriptstyle j_1 & \scriptstyle j_2 & \scriptstyle \cdots & \scriptstyle j_{a-1} \\
\scriptstyle i_2 \\
\scriptstyle \vdots \\
\scriptstyle i_{b+1}
\end{ytableau}
$}}
\]
for the class of a monomial, where
\(j_1\le\cdots\le j_{a-1}\) lists the multiset encoded by~\(\alpha\)
and \(I=\{i_1<\cdots<i_{b+1}\}\). Thus the shape \((a,1^b)\) is
recovered from \(|\alpha|=a-1\) and \(|I|=b+1\).
In examples, a word \(j_1\cdots j_s\) in a superscript abbreviates the
content vector \(\sum_{\nu=1}^s\eps_{j_\nu}\), while a strictly increasing
word \(i_1\cdots i_t\) in a subscript abbreviates the set
\(\{i_1<\cdots<i_t\}\) with its increasing wedge order.
\begin{example}
  $T^{(0,1,0,1,0)}_{\{1,3,5\}}\in\bS_{(3,1,1)}(\bk^5)$ is the semistandard hook tableau $\ytableausetup{boxsize=1.1em}
\ytableaushort{124,3,5}$.
\end{example}

For subsets \(I_1,\dots,I_r\), with braces around singleton factors
suppressed, write
\[
T^\alpha_{I_1\wedge\cdots\wedge I_r}
:=[e^\alpha\otimes(e_{I_1}\wedge\cdots\wedge e_{I_r})]_d.
\]
The factor order records the sign, and the expression vanishes if two factors
intersect.  For example,
\(T^\alpha_{\{1,3\}\wedge\{2\}}=-T^\alpha_{\{1,2,3\}}\), whereas
\(T^\alpha_{\{1,3\}\wedge\{1\}}=0\).

\begin{samepage}
\begin{proposition} \label{prop:hd-convert}
  Suppose \(a\ge1\), \(b\ge0\), and \(a+b\in\bk^\times\).  We have inverse
  isomorphisms
  \begin{align*}
    \bar{d}_{a,b} \colon \coker d_{a,b} &\xrightarrow{\cong} \coker h_{a,b} &
    \bar{h}_{a,b} \colon \coker h_{a,b} &\xrightarrow{\cong} \coker d_{a,b}\\
    [e^\alpha \otimes e_I]_d &\mapsto [d_{a+1,b-1}(e^\alpha \otimes e_I)]_h &
                                        [e^\alpha \otimes e_I]_h &\mapsto [h_{a-1,b+1}(e^\alpha \otimes e_I)]_d
\end{align*}
\end{proposition}
\end{samepage}

\begin{proof}
Consecutive \(d\)-maps and consecutive \(h\)-maps compose to zero, so the
displayed maps are well-defined.  On \(C_{a-1,b+1}\) and \(C_{a,b}\),
respectively, the normalized Cartan identity
\cite[Formula~(III)]{Cartan1951Notions} reads
\[
h_{a-1,b+1}d_{a+1,b-1}+d_{a,b}h_{a-2,b+2}=\mathrm{id},
\qquad
d_{a+1,b-1}h_{a-1,b+1}+h_{a,b}d_{a+2,b-2}=\mathrm{id}.
\]
Modulo \(\im d_{a,b}\) and \(\im h_{a,b}\), these say
\(\bar h_{a,b}\bar d_{a,b}=\mathrm{id}\) and
\(\bar d_{a,b}\bar h_{a,b}=\mathrm{id}\), respectively.
\end{proof}

\subsection{Explicit Pieri Formulas for Hooks}

In characteristic zero, Pieri's rule gives, up to scalar, the horizontal maps
\begin{align*}
\rP_m&\colon \bS_{(a,1^b)}(V)
  \to S^m(V)\otimes\bS_{(a-m,1^b)}(V),\qquad
\rQ_m\colon S^m(V)\otimes\bS_{(a-m,1^b)}(V)
  \to \bS_{(a,1^b)}(V),
\end{align*}
for \(0\le m\le a-1\), and the vertical maps
\begin{align*}
\rP'_m&\colon \bS_{(a,1^{b+m})}(V)
  \to \bigwedge^m V\otimes\bS_{(a,1^b)}(V),\qquad
\rQ'_m\colon \bS_{(a,1^b)}(V)\otimes\bigwedge^m V
  \to \bS_{(a,1^{b+m})}(V).
\end{align*}
We fix the following normalizations. Integral constructions appear in
\cite{Olver1982,MaliakasOlver1992}; see also
\cite{SamWeyman2011Pieri,Sam2009SchurModules}. We use the same notation
after canonical permutations of tensor factors. Pictorially, the yellow
boxes are those removed or appended:
\[
\ytableausetup{boxsize=0.82em}
\rP_m,\rQ_m:\;\;
\begin{ytableau}
{} & {} & *(yellow!60) & *(yellow!60) \\ {} \\ {}
\end{ytableau}
\;\rightleftharpoons\;
\underbrace{\begin{ytableau}*(yellow!60) & *(yellow!60)\end{ytableau}}_{S^m V}
\,\otimes\,
\begin{ytableau}{} & {} \\ {} \\ {}\end{ytableau}
\qquad
\rP'_m,\rQ'_m:\;\;
\begin{ytableau}
{} & {} \\ {} \\ {} \\ *(yellow!60) \\ *(yellow!60)
\end{ytableau}
\;\rightleftharpoons\;
\begin{ytableau}{} & {} \\ {} \\ {}\end{ytableau}
\,\otimes\,
\underbrace{\begin{ytableau}*(yellow!60) \\ *(yellow!60)\end{ytableau}}_{\bigwedge^m V}
\]

The \(d\)-model gives \(\rQ_m\) and \(\rP'_m\) directly.

\begin{proposition}\label{prop:direct-hook-pieri}
  In the $d$-model, we have
  \begin{align*}
    \rQ_m(e^\beta \otimes T^\alpha_I) &=  T^{\alpha+\beta}_I, \qquad
    \rP'_m(T^\alpha_I) = \sum_{J \subseteq I,\ |J|=m} \sgn(J,I) e_J\otimes T^\alpha_{I \setminus J}.
  \end{align*}
\end{proposition}

\begin{proof}
  By associativity, we have a commutative square (with \(S^r(V)=0\) for
  \(r<0\))
  \[
    \xymatrix{ S^m(V) \otimes S^{a-m-2}(V) \otimes V \ar[r] \ar[d] & S^{a-2}(V) \otimes V \ar[d] \\
      S^m(V) \otimes S^{a-m-1}(V) \ar[r] & S^{a-1}(V) }
  \]
  where all maps are given by multiplication (with the left vertical map acting as the identity on $S^m(V)$). Hence, we can tensor this with $\bigwedge^{b+1}(V)$ and precompose with the comultiplication $\bigwedge^{b+2}(V) \to V \otimes \bigwedge^{b+1}(V)$ to obtain a commutative square
  \[
    \xymatrix{ S^m(V) \otimes S^{a-m-2}(V) \otimes \bigwedge^{b+2}(V) \ar[r] \ar[d] & S^{a-2}(V) \otimes \bigwedge^{b+2}(V) \ar[d] \\
      S^m(V) \otimes S^{a-m-1}(V) \otimes \bigwedge^{b+1}(V) \ar[r] & S^{a-1}(V) \otimes \bigwedge^{b+1}(V)}
  \]  
  Thus the formula for \(\rQ_m\) descends to the indicated cokernels; it is
  equivariant and nonzero. The formula for \(\rP'_m\) follows similarly from
  coassociativity of exterior comultiplication.
\end{proof}

Symmetric comultiplication and exterior multiplication give the other two
maps directly in the \(h\)-model.

\begin{lemma}\label{lem:hook-pieri-h-model}
  In the $h$-model, we have
  \begin{align*}
    \rP_m([e^\alpha \otimes e_I]_h) &=
\sum_{\substack{\beta \le \alpha\\ |\beta|=m}} \tbinom{\alpha}{\beta}
e^\beta\otimes[e^{\alpha-\beta}\otimes e_I]_h,\qquad
    \rQ'_m([e^\alpha \otimes e_I]_h \otimes e_K) = [e^\alpha \otimes (e_I \wedge e_K)]_h.
  \end{align*}
\end{lemma}

\begin{proof}
Both operations preserve the images of the corresponding \(h\)-maps, by
coassociativity of symmetric comultiplication and associativity of exterior
multiplication.
\end{proof}

\begin{proposition}\label{prop:basic-hook-pieri-closed}
  In the $d$-model, we have the following explicit formulas.
\begin{enumerate}
\item[\textup{(i)}]
\textbf{Horizontal (row removal).}
For \(0\le m\le a{-}1\), the map
\[
\rP_m \colon \bS_{(a,1^b)}(V)\longrightarrow S^m(V)\otimes\bS_{(a-m,1^b)}(V)
\]
is given by the formula
\begin{align*}
\rP_m(T^\alpha_I) &=
\sum_{\substack{\beta\le\alpha\\|\beta|=m}} \tbinom{\alpha}{\beta}
e^\beta\otimes T^{\alpha-\beta}_I
+\frac{1}{a+b-m}
\sum_{i\in I}\sgn(i,I)
\sum_{\substack{\beta\le\alpha\\|\beta|=m}} \sum_{s\in\supp(\beta)}
\beta_s \tbinom{\alpha}{\beta} e^{\beta-\eps_s+\eps_i}\otimes T^{\alpha-\beta}_{ s\wedge (I\setminus i)}.
\end{align*}

\item[\textup{(ii)}]
\textbf{Vertical (column appending).}
For \(m\ge 0\), the map
\[
\rQ'_m \colon \bS_{(a,1^b)}(V) \otimes \bigwedge^m V \longrightarrow\bS_{(a,1^{b+m})}(V)
\]
is given by the formula
  \begin{align*}
    \rQ'_m(T^\alpha_I \otimes e_K) &= T^\alpha_{I \wedge K}
     - \frac{1}{a+b+m} \sum_{k \in K} \sgn(k,K)
     \sum_{\substack{s\in[n]\\ \alpha_s+\delta_{k,s}>0}}
     (\alpha_s+\delta_{k,s})
     T^{\alpha+\eps_k-\eps_s}_{I\wedge s\wedge(K\setminus k)}.
  \end{align*}
\end{enumerate}
\end{proposition}

\begin{remark} \label{rmk:parts}
In each formula, the first summand is the \defi{principal part}; the remaining
summands are the \defi{off-diagonal terms}, obtained by exchanging a row
label with a column or exterior label.
\end{remark}

\begin{proof}
\textbf{(i)}
The factorization $\rP_m=(1\otimes\bar h_{a-m,b})\circ\Delta_m\circ\bar d_{a,b}$
uses the component
\(\Delta_m\colon S^aV\to S^mV\otimes S^{a-m}V\) of symmetric
comultiplication on the symmetric factor. Multiplicativity of symmetric
comultiplication gives
\[
\Delta_m(e^{\alpha+\eps_i})
=
\sum_{\substack{\beta\le\alpha\\|\beta|=m}}
\binom{\alpha}{\beta}
e^\beta\otimes e^{\alpha-\beta+\eps_i}
+
\sum_{\substack{\delta\le\alpha\\|\delta|=m-1}}
\binom{\alpha}{\delta}
e^{\delta+\eps_i}\otimes e^{\alpha-\delta}.
\]
After inserting this decomposition and summing over~\(i\) in
\(\bar d_{a,b}(T_I^\alpha)\), the first sum factors through
\(\bar d_{a-m,b}\), and hence gives the principal part because
\(\bar h_{a-m,b}\bar d_{a-m,b}=\mathrm{id}\). Applying
\(\bar h_{a-m,b}\) to the second sum and setting
\(\beta=\delta+\eps_s\), the identity
\[
 \binom{\alpha}{\beta-\eps_s} (\alpha - \beta+ \eps_s)_s = 
 \beta_s \binom{\alpha}{\beta}. 
\]
then gives exactly the off-diagonal term in~\textup{(i)}.

\textbf{(ii)}
Let \(\mu_m([g\otimes\eta]_h\otimes e_K)
:=[g\otimes(\eta\wedge e_K)]_h\). Then
$\rQ'_m=\bar h_{a,b+m}\circ\mu_m\circ(\bar d_{a,b}\otimes1)$.
Moreover,
\begin{align*}
\bar d_{a,b+m}(T^\alpha_{I\wedge K})
&=\mu_m(\bar d_{a,b}(T^\alpha_I)\otimes e_K)
-(-1)^b\left[\sum_{k\in K}\sgn(k,K)e^{\alpha+\eps_k}
\otimes e_I\wedge e_{K\setminus k}\right]_h.
\end{align*}
Apply \(\bar h_{a,b+m}\) and use
\(e_s\wedge e_I=(-1)^{b+1}e_I\wedge e_s\). The two signs combine to
give the correction in~\textup{(ii)}. \qedhere
\end{proof}

Finally, we assemble these into four compositions of Pieri maps, which provide the hook analogues of the classical Garnir relations~\cite[\S7.4]{Fulton1997}.

\begin{definition}[Hook Pieri compositions]\label{def:hook-pieri-compositions}
For \(a,c\ge1\), \(b,d,r\ge0\), define the following maps whenever the
displayed hooks exist:
\begin{align*}
\Phi_m&\colon \bS_{(a,1^b)}(V)\otimes\bS_{(c,1^d)}(V)
\longrightarrow \bS_{(a+m,1^b)}(V)\otimes\bS_{(c-m,1^d)}(V),\\
\Psi_{m,r}&\colon \bS_{(a,1^b)}(V)\otimes\bS_{(c,1^d)}(V)\otimes S^r(V)
\longrightarrow \bS_{(a+m,1^b)}(V)\otimes\bS_{(c-m+r,1^d)}(V),\\
\Phi'_m&\colon \bS_{(a,1^b)}(V)\otimes\bS_{(c,1^d)}(V)
\longrightarrow \bS_{(a,1^{b+m})}(V)\otimes\bS_{(c,1^{d-m})}(V),\\
\Psi'_{m,r}&\colon \bS_{(a,1^b)}(V)\otimes\bS_{(c,1^d)}(V)
\otimes\textstyle\bigwedge^r(V)
\longrightarrow \bS_{(a,1^{b+m})}(V)\otimes\bS_{(c,1^{d-m+r})}(V).
\end{align*}
Thus \(0\le m\le c-1\) in the unprimed lines and \(0\le m\le d\) in
the primed lines. If \(\tau\) denotes the canonical symmetry of adjacent
tensor factors, the composites are
\begin{align*}
\Phi_m&=(\rQ_m\otimes1)\circ(\tau\otimes1)\circ(1\otimes\rP_m),&
\Psi_{m,r}&=(1\otimes\rQ_r)\circ(1\otimes\tau)\circ(\Phi_m\otimes1),\\
\Phi'_m&=(\rQ'_m\otimes1)\circ(1\otimes\rP'_m),&
\Psi'_{m,r}&=(1\otimes\rQ'_r)\circ(\Phi'_m\otimes1).
\end{align*}
The unprimed maps are the \defi{row Pieri compositions}; the primed maps are
the \defi{column Pieri compositions}.
\end{definition}

\begin{proposition} \label{prop:relations}
  Let \(a>c\), \(b>d\), and \(m,r\ge0\). With
  \(R_{(a-1,b),(c-1,d)}\) as defined in
  \S\ref{ssec:universal-hook-algebra}, set
  \[
  R:=R_{(a-1,b),(c-1,d)}
  \subseteq\bS_{(a,1^b)}(V)\otimes\bS_{(c,1^d)}(V).
  \]
  For this fixed target, the relevant instances are
  \begin{align*}
  \Psi_{m,r}&\colon
  \bS_{(a-m,1^b)}(V)\otimes\bS_{(c+m-r,1^d)}(V)\otimes S^r(V)
  \longrightarrow\bS_{(a,1^b)}(V)\otimes\bS_{(c,1^d)}(V),\\
  \Psi'_{m,r}&\colon
  \bS_{(a,1^{b-m})}(V)\otimes\bS_{(c,1^{d+m-r})}(V)
  \otimes\textstyle\bigwedge^r(V)
  \longrightarrow\bS_{(a,1^b)}(V)\otimes\bS_{(c,1^d)}(V).
  \end{align*}
  Then:
  \begin{enumerate}
  \item If \(c-r\ge a-m\ge1\), then \(\im\Psi_{m,r}\subseteq R\).
  \item If \(d-r\ge b-m\ge0\), then \(\im\Psi'_{m,r}\subseteq R\).
  \end{enumerate}
\end{proposition}

\begin{proof}
  Since the projection defining \(R\) and the Pieri compositions are defined between free modules over the localization of \(\bZ\) at the displayed denominators, it suffices to check the vanishing over \(\bQ\).

The primed assertion follows from the unprimed one by transpose duality for
polynomial functors~\cite[\S3.3.8]{SamSnowden2015}: its inequality becomes
\(d+1-r\ge b+1-m\ge1\). We prove the unprimed assertion.

Let \(\lambda\) have Frobenius coordinates \((a-1,c-1\mid b,d)\). Since
\(R\) is the kernel of the projection onto \(\bS_\lambda(V)\), it suffices
to show that \(\bS_\lambda\) does not occur in the source of \(\Psi_{m,r}\).
Otherwise, the Littlewood--Richardson rule gives a summand
\(\bS_\mu\) of \(\bS_{(a-m,1^b)}\otimes S^r\) such that
\(\bS_\lambda\) occurs in
\(\bS_\mu\otimes\bS_{(c+m-r,1^d)}\). Pieri's rule and the
Littlewood--Richardson rule give
\[
(a-m,1^b)\subseteq\mu\subseteq\lambda,
\]
so \(\ell(\mu)=\ell(\lambda)=b+1\), and \(\lambda/\mu\) contains no cell
in its first column. A Littlewood--Richardson tableau of this skew shape and content
\((c+m-r,1^d)\) would therefore place \(c+m-r\) entries equal to \(1\) in
the remaining \(a-1\) columns. This is impossible because
\(c-r\ge a-m\) gives \(c+m-r\ge a\).
\end{proof}

\section{Elimination Systems for Hook Pieri Maps}\label{sec:inv-matrices}

This section proves the invertibility of the \emph{elimination matrices},
which record coefficients of the hook Pieri maps on the finite swap sets of
Definitions~\ref{def:closure-set-row} and~\ref{def:Theta-closure-matrix}.
After normalization, each elimination matrix is the identity plus a scalar
multiple of a matrix identified in
\S\ref{ssec:orbit-quotient} with a shifted principal submatrix of the
exterior-power action induced by a rank-one perturbation of $-I$. The
resulting characteristic polynomials give the row and column criteria of
Theorems~\ref{thm:row-elimination-invertible} and
\ref{thm:Theta-elim-invertible}, including the exact row singularity locus
and explicit positive-characteristic bounds.

\subsection{Exchange matrices}\label{ssec:orbit-quotient}

Fix an ordered basis $e_1,\dots,e_n$ of a vector space $E$, positive integers $t_1,\dots,t_n$, and set $N=\sum_i t_i$. Define the \defi{weighted label operator} $\Gamma_t \colon E \to E$ by
\[
  \Gamma_t(e_i)=u_t-e_i=(t_i-1)e_i+\sum_{j\ne i}t_j e_j.
\]
We can also write this as $\Gamma_t=u_t\otimes\mathbf 1-\operatorname{id}_E$
where $u_t=\sum_i t_i e_i$, and  $\mathbf 1\in E^\vee$ is defined by $\mathbf 1(e_i)=1$ for all $1\le i\le n$.

For $0\le k\le n$, the
representation of $\GL(E)$ on $\bigwedge^kE$ differentiates to
$\mathfrak{gl}(E)\longrightarrow\mathfrak{gl}(\textstyle \bigwedge^kE)$ given by $T\mapsto T^{\wedge k}$ where
\[
  T^{\wedge k}(v_1\wedge\cdots\wedge v_k)=\displaystyle \sum_{\ell=1}^k v_1\wedge\cdots \wedge T(v_\ell) \wedge \cdots\wedge v_k.
\]
This is not the multiplicative exterior power $\bigwedge^kT$.  After extending scalars and triangularizing $T$, the induced wedge basis triangularizes $T^{\wedge k}$; its diagonal entries are the sums of $k$ diagonal entries of $T$. Hence, we have
\[
  \Gamma_t^{\wedge k}(e_I)=\Bigl(\sum_{i\in I}(t_i-1)\Bigr)e_I+\sum_{i\in I,\ j\notin I}\epsilon(I;i,j)\,t_j\,e_{(I\setminus i)\cup j},
\]
where $\epsilon(I;i,j)$ is the sign obtained by replacing $i$ by $j$ in the increasing sequence for $I$ and sorting.  This is the exchange sign used below. 

\begin{proposition}\label{cor:Gamma-spectrum}
  The characteristic polynomial of $\Gamma_t^{\wedge k}$ is
\[
  \det\bigl(xI-\Gamma_t^{\wedge k}\bigr)=\bigl(x-(N-k)\bigr)^{\binom{n-1}{k-1}}\,(x+k)^{\binom{n-1}{k}},
\]
with the convention $\tbinom qr=0$ for $r<0$ or $r>q$.  In particular,
\[
  \det\bigl(\Gamma_t^{\wedge k}\bigr)=(N-k)^{\binom{n-1}{k-1}}\,(-k)^{\binom{n-1}{k}}.
\]
\end{proposition}

The proof is straightforward from the discussion above and omitted.

\begin{example}\label{ex:orbit-quotient-prototype}
Take labels $\{1,2,3\}$ with multiplicities $t=(2,1,1)$, so $N=4$, and $k=2$.  Then $\Gamma_t(e_i)=u_t-e_i$ gives $\Gamma_t(e_1)=e_1+e_2+e_3$, $\Gamma_t(e_2)=2e_1+e_3$, $\Gamma_t(e_3)=2e_1+e_2$, and in the wedge basis $e_{12},e_{13},e_{23}$,
\[
  \Gamma_t^{\wedge 2}=\begin{pmatrix} 1 & 1 & -2\\ 1 & 1 & 2\\ -1 & 1 & 0\end{pmatrix},
  \qquad
  \det\bigl(xI-\Gamma_t^{\wedge 2}\bigr)=(x-2)^2(x+2),
\]
matching $\{N-k,-k\}=\{2,-2\}$ with multiplicities $\binom{2}{1}=2$ and $\binom{2}{2}=1$.
\end{example}

For $C\subseteq\{1,\dots,n\}$, let $E_C$ be the span of $\{ e_i:i\in C\}$ and let $\Gamma_{t_C}$ be the weighted label operator for the restricted multiplicities.  Set
\[
  F=\{1,\dots,n\}\setminus C,\qquad
  N_C=\sum_{i\in C}t_i,\qquad N_F=\sum_{i\in F}t_i.
\]
Assume $|F|\le k\le n$, put $p=k-|F|$, and let $\Omega$ be the set of $k$-subsets containing $F$.  Write $A_\Omega$ for the corresponding principal submatrix of $\Gamma_t^{\wedge k}$.

\begin{proposition} \label{prop:principal-sub}
  The isomorphism
    $\Phi_F\colon \bigwedge^pE_C\to
    {\rm span}\{ e_J:J\in\Omega\}$ defined by $\Phi_F(e_I) = e_I\wedge e_F$  satisfies
  \[
    \Phi_F^{-1}A_\Omega\Phi_F=(N_F-|F|)\,\operatorname{id}+\Gamma_{t_C}^{\wedge p}.
  \]
  If $C\ne\varnothing$, it follows that
  \begin{align*}
    \det(xI-A_\Omega)
    &=(x-(N-k))^{\binom{|C|-1}{p-1}}
     (x-(N_F-k))^{\binom{|C|-1}{p}}\\
    \det(A_\Omega)
    &=(N-k)^{\binom{|C|-1}{p-1}}
     (N_F-k)^{\binom{|C|-1}{p}}.
  \end{align*}
  If $C=\varnothing$, then necessarily $k=n$ and $A_\Omega$ is the $1\times1$ matrix $(N-k)$.
\end{proposition}

\begin{proof}
Apply the derivation formula to $e_I\wedge e_F$ and project onto the span of $\{e_J:J\in\Omega\}$.  The terms acting on $e_I$ give $\Gamma_{t_C}^{\wedge p}$: components in $E_F$ vanish upon wedging with $e_F$.  For each $i\in F$, only the diagonal term $(t_i-1)e_i$ survives the projection; every nonzero exchange removing $i$ leaves $\Omega$.  This proves the conjugacy formula.

For $C\ne\varnothing$, Proposition~\ref{cor:Gamma-spectrum} gives the two roots $N_C-p$ and $-p$, with multiplicities $\binom{|C|-1}{p-1}$ and $\binom{|C|-1}{p}$.  Adding $N_F-|F|$ gives $N-k$ and $N_F-k$, respectively.
\end{proof}

\subsection{Rows}\label{ssec:rows-inv}

Fix $a\ge1$, $b\ge0$, and $0\le m\le a-1$, and consider
\[
  \rP_m \colon \bS_{(a,1^b)}(V) \to S^m(V) \otimes \bS_{(a-m,1^b)}(V).
\]
The matrix below isolates the coefficients of $\rP_m$ among tabloids with fixed total content.

\begin{definition}\label{def:closure-set-row}
Fix a basis vector \(T_{J_0}^{\rho+\tau_0}\) and a subset
\(K_0\subseteq J_0\). The parameters have the following meanings.
\begin{center}
\renewcommand{\arraystretch}{1.25}
\begin{tabular}{l|l|l}
Symbol & Meaning & Size \\\hline
$J_0\subseteq[n]$ & column entries of $T_{J_0}^{\rho+\tau_0}$ & $|J_0|=b+1$ \\
$\tau_0\in\bN^n$ & arm content selected for extraction & $|\tau_0|=m$ \\
$\rho\in\bN^n$ & arm content retained in the hook factor & $|\rho|=a-m-1$ \\
$K_0\subseteq J_0$ & column entries allowed to exchange with labels counted by $\tau_0$ &
\end{tabular}
\end{center}

We keep \(\rho\) and \(J_0\setminus K_0\) fixed, while labels counted by
\(\tau_0\) may exchange with entries of \(K_0\) between the \(S^m(V)\)-factor
and the column. The \defi{swap set} \(\Omega=\Omega(J_0;\tau_0,K_0,\rho)\) consists of the pairs \((\tau,J)\) with \(|\tau|=m\) and \(|J|=b+1\) satisfying
\[
\tau+\mathbf 1_J\;=\;\tau_0+\mathbf 1_{J_0}, \qquad J_0 \setminus K_0 \subseteq J.
\]
Define the \defi{elimination matrix} \(M=M(J_0;\tau_0,K_0,\rho)\), with rows and columns indexed by $\Omega$:
\[
M_{(\tau,J),(\tau',J')}
:=\operatorname{coeff}_{e^\tau\otimes T^\rho_J}
\bigl(\rP_m(T^{\rho+\tau'}_{J'})\bigr),
\]
where $\operatorname{coeff}$ refers to taking the coefficient of the corresponding basis vector in the explicit $d$-model expansion of Proposition~\ref{prop:basic-hook-pieri-closed}(i). \qedhere
\end{definition}

\begin{example}\label{ex:row-elim-intro}
Assume $\operatorname{char}\bk=0$.  Let $V=\bk^4$ and $a=4$, $b=2$, $m=2$.  With $\rho=\eps_1$, $\tau_0=2\eps_2$, $J_0=\{1,3,4\}$, $K_0=\{3,4\}$, the conditions $\tau+\mathbf 1_J=\eps_1+2\eps_2+\eps_3+\eps_4$ and $J_0\setminus K_0\subseteq J\subseteq\{1,2,3,4\}$ give three pairs $(\tau_i,J_i)\in\Omega$ with source tabloids $T_i := T^{\rho+\tau_i}_{J_i}$:
\[
\ytableausetup{boxsize=1.1em}
\begin{aligned}
T_1&=
\begin{ytableau}
1 & 1 & *(orange!35)2 & *(orange!35)2\\
*(blue!30)3\\
*(blue!30)4
\end{ytableau},
&\qquad
T_2&=
\begin{ytableau}
1 & 1 & *(orange!35)2 & *(blue!30)3\\
*(orange!35)2\\
*(blue!30)4
\end{ytableau},
&\qquad
T_3&=
\begin{ytableau}
1 & 1 & *(orange!35)2 & *(blue!30)4\\
*(orange!35)2\\
*(blue!30)3
\end{ytableau}.
\end{aligned}
\]
Blue marks the entries originating in \(K_0\), orange those originating
in \(\tau_0\), and the fixed entries are uncolored. All three tabloids share the label
multiset $\{1,1,2,2,3,4\}$; only its distribution between the arm and
column varies. Picking the coefficient of $e^\tau\otimes T^{\eps_1}_J$ in $\rP_2(T_i)$ for each
$(\tau,J)\in\Omega$ yields
\[
\begingroup
\setlength{\arraycolsep}{6pt}
M\;=\;
\bordermatrix{
& T^{122}_{134} & T^{123}_{124} & T^{124}_{123} \cr
\rule{0pt}{2.6ex} e_2^2\otimes T^1_{134}   & 1        & \tfrac14 & -\tfrac14 \cr
\rule{0pt}{2.6ex} e_2e_3\otimes T^1_{124}  & \tfrac12  & \tfrac54 & \tfrac14 \cr
\rule{0pt}{2.6ex} e_2e_4\otimes T^1_{123}  & -\tfrac12 & \tfrac14 & \tfrac54
},
\qquad
\det M=\tfrac98.
\endgroup
\]
Using notation as in the next proof, we have \(F=\{1\}\), \(C=\{2,3,4\}\), \(p=2\), \(N_F=1\), and
\(d=a+b-m=4\).
Thus the two eigenvalues in the proof below are \(3/2\) and \(1/2\), with
multiplicities two and one; their product is \(9/8\).
\end{example}

\begin{theorem}\label{thm:row-elimination-invertible}
Suppose $\bk$ has characteristic $0$ or $>a+b$.  Then $M=M(J_0;\tau_0,K_0,\rho)$ is singular precisely when
\[
  m=a-1,\qquad J_0=K_0,\qquad \text{and}\qquad  \supp(\tau_0)\nsubseteq J_0.
\]
\end{theorem}

\begin{proof}
With \(\epsilon(J;i,s)\) and \(\Gamma_t^{\wedge k}\) as in
\S\ref{ssec:orbit-quotient}, set $d=a+b-m$. By
Proposition~\ref{prop:basic-hook-pieri-closed}(i), the part of $\rP_m(T^{\rho+\tau'}_{J'})$ whose second factor has arm content $\rho$ is
\[
c_{\tau'}\left(
e^{\tau'}\otimes T^\rho_{J'}
+\frac1d\sum_{i\in J'}\sum_{s\in\supp(\tau')}
\tau'_s\sgn(i,J')\,
e^{\tau'-\eps_s+\eps_i}\otimes T^\rho_{s\wedge(J'\setminus i)}
\right),
\qquad
c_{\tau'}=\tbinom{\rho+\tau'}{\tau'}.
\]
Since $|\rho+\tau'|=a-1$, each $c_{\tau'}$ is a unit under the characteristic hypothesis.  Thus
\[
  M=M^{\mathrm{norm}}D,\qquad
  M^{\mathrm{norm}}=I+\frac1dL,\qquad
  D=\operatorname{diag}(c_{\tau'})_{(\tau',J')\in\Omega},
\]
where $L$ has diagonal entries
\[
L_{(\tau,J),(\tau,J)}=\sum_{i\in J}\tau_i
\]
and possible off-diagonal entries
\[
L_{(\tau'-\eps_s+\eps_i,(J'\setminus i)\cup s),(\tau',J')}
=\tau'_s\epsilon(J';i,s)
\qquad(i\in J',\ s\notin J'),
\]
whenever the indicated row belongs to $\Omega$; all other entries vanish. Put
\[
  t=\tau_0+\mathbf1_{J_0},\qquad F=J_0\setminus K_0,
  \qquad N=|t|=m+b+1.
\]
The map $(\tau,J)\mapsto J$ identifies $\Omega$ with the $(b+1)$-subsets of $\supp(t)$ containing $F$; its inverse is $J\mapsto(t-\mathbf1_J,J)$.  On $E_t={\rm span}\{ e_i:i\in\supp(t)\}$, the diagonal and off-diagonal entries of $L$ are respectively
\[
  \sum_{i\in J}\tau_i=\sum_{i\in J}(t_i-1),
  \qquad
  \tau'_s\epsilon(J';i,s)=t_s\epsilon(J';i,s),
\]
where the second equality uses $s\notin J'$.  Hence $L$ is the principal submatrix of $\Gamma_t^{\wedge(b+1)}$ indexed by these subsets.

Set $C=\supp(t)\setminus F$, $p=b+1-|F|$, and $N_F=\sum_{i\in F}t_i$.  If $C\ne\varnothing$, Proposition~\ref{prop:principal-sub} gives
\[
\det(xI-M^{\mathrm{norm}})
=\left(x-\frac{a+b}{d}\right)^{\binom{|C|-1}{p-1}}
 \left(x-\frac{|\rho|+N_F}{d}\right)^{\binom{|C|-1}{p}}.
\]
If $C=\varnothing$, then $M^{\mathrm{norm}}$ is the $1\times1$ matrix $((a+b)/d)$.

Since $N_F\le N$, the integers $d$, $a+b$, and every nonzero value of $|\rho|+N_F$ lie between $1$ and $a+b$, and hence are units in $\bk$.  Moreover, $|\rho|+N_F=0$ precisely when $m=a-1$ and $F=\varnothing$, or equivalently $m=a-1$ and $J_0=K_0$.  In that case $C=\supp(t)$, $p=b+1$, and the zero root occurs with multiplicity
\[
  \binom{|\supp(t)|-1}{b+1}.
\]
Since $\supp(t)=J_0\cup\supp(\tau_0)$ and $|J_0|=b+1$, this multiplicity is positive precisely when $\supp(\tau_0)\nsubseteq J_0$.  The diagonal entries of $D$ are also units, so the result follows.
\end{proof}


\subsection{Columns}\label{ssec:cols-inv}

Fix $a\ge1$, $b\ge0$, and $m\ge0$, and consider
\[
\rQ'_m \colon \bS_{(a,1^b)}(V)\otimes\textstyle\bigwedge^m V
\longrightarrow\bS_{(a,1^{b+m})}(V).
\]

\begin{definition} \label{def:Theta-closure-matrix}
Fix a basis tensor \(T_{I^-}^{\rho+\tau_0}\otimes e_{J_0}\), where
\(I^-\cap J_0=\emptyset\). The parameters have the following meanings.
\begin{center}
\renewcommand{\arraystretch}{1.25}
\begin{tabular}{l|l|l}
Symbol & Meaning & Size \\\hline
$I^-\subseteq[n]$ & column index of the hook factor & $|I^-|=b+1$ \\
$J_0\subseteq[n]$ & index set of the exterior factor & $|J_0|=m$ \\
$\rho\in\bN^n$ & arm content held fixed & \\
$\tau_0\in\bN^n$ & arm content redistributed with the labels in $J_0$ & $|\rho+\tau_0|=a-1$
\end{tabular}
\end{center}

We keep \(I^-\) and \(\rho\) fixed and redistribute the content
\(\tau_0+\mathbf1_{J_0}\) between the arm and the exterior factor.
The \defi{swap set} \(\Omega=\Omega(I^-;\tau_0,J_0,\rho)\) consists of the pairs \((\tau,J)\) with \(|\tau|=|\tau_0|\) and \(|J|=m\) satisfying
\[
\tau+\mathbf 1_J\;=\;\tau_0+\mathbf 1_{J_0}, \qquad J\cap I^-=\varnothing.
\]
The \defi{elimination matrix} \(M=M(I^-;\tau_0,J_0,\rho)\), indexed by~\(\Omega\), is defined by
\[
M_{(\tau,J),(\tau',J')}
:=\operatorname{coeff}_{T_{I^-\wedge J}^{\rho+\tau}}
\bigl(\rQ'_m(T_{I^-}^{\rho+\tau'}\otimes e_{J'})\bigr),
\]
where $\operatorname{coeff}$ refers to the explicit $d$-model expansion of Proposition~\ref{prop:basic-hook-pieri-closed}(ii). \qedhere
\end{definition}

\begin{example}\label{ex:col-elim-intro}
Assume $\operatorname{char}\bk=0$.  Take $V=\bk^5$ with $a=3$, $b=1$, $m=2$, $I^-=\{1,2\}$, $J_0=\{3,4\}$, $\tau_0=2\eps_5$, and $\rho=0$.  The conditions $J\cap I^-=\varnothing$ and $\tau+\mathbf1_J=\tau_0+\mathbf1_{J_0}$ give three pairs $(\tau_i,J_i)\in\Omega$.
Their target tabloids $U_i:=T^{\tau_i}_{I^-\wedge J_i}$ are
\[
\ytableausetup{boxsize=1.1em}
\begin{aligned}
U_1&=
\begin{ytableau}
1 & *(orange!35)5 & *(orange!35)5\\
2\\
*(blue!30)3\\
*(blue!30)4
\end{ytableau},
&\qquad
U_2&=
\begin{ytableau}
1 & *(blue!30)4 & *(orange!35)5\\
2\\
*(blue!30)3\\
*(orange!35)5
\end{ytableau},
&\qquad
U_3&=
\begin{ytableau}
1 & *(blue!30)3 & *(orange!35)5\\
2\\
*(blue!30)4\\
*(orange!35)5
\end{ytableau}.
\end{aligned}
\]
Blue marks entries originating in $J_0$, orange marks those originating in
$\tau_0$, and the fixed entries of $I^-$ are uncolored.
Reading off the coefficient of $U_i$ in $\rQ'_2(T^{\tau_j}_{I^-}\otimes e_{J_j})$ yields
\[
\begingroup
\setlength{\arraycolsep}{6pt}
M\;=\;\bordermatrix{
& T^{55}_{12}\otimes e_{34} & T^{45}_{12}\otimes e_{35} & T^{35}_{12}\otimes e_{45} \cr
\rule{0pt}{2.6ex} T^{55}_{1234} &  \tfrac{2}{3} & -\tfrac16 &  \tfrac16 \cr
\rule{0pt}{2.6ex} T^{45}_{1235} & -\tfrac{1}{3} &  \tfrac12 & -\tfrac16 \cr
\rule{0pt}{2.6ex} T^{35}_{1245} &  \tfrac{1}{3} & -\tfrac16 &  \tfrac12
},
\qquad
\det M=\tfrac{1}{9}.
\endgroup
\]
With notation as in the next proof, we have \(C=\{3,4,5\}\), \(N_C=4\), \(m=2\), and
\(d=a+b+m=6\). The matrix
\(L\) in the proof below has eigenvalues \(4,4,0\), so
\(M=\operatorname{id}-L/6\) has eigenvalues \(1/3,1/3,1\), whose product
is \(1/9\).
Note the structural contrast with the row case (Example~\ref{ex:row-elim-intro}): the fixed labels of the target column are $I^-$ (not the row content).
\end{example}

\begin{theorem}\label{thm:Theta-elim-invertible}
  Suppose $\bk$ has characteristic $0$ or $>a+b+m$.  Then the elimination matrix $M=M(I^-;\tau_0,J_0,\rho)$ is invertible.
\end{theorem}

\begin{proof}
If $m=0$, then $\Omega$ is a singleton and
$M=(1)$, so assume $m>0$ and set $d=a+b+m$.
Proposition~\ref{prop:basic-hook-pieri-closed}(ii) gives
\begin{align*}
\rQ'_m(T^{\rho+\tau'}_{I^-}\otimes e_{J'})
&=T^{\rho+\tau'}_{I^-\wedge J'}- \frac1d\sum_{j\in J'}\sgn(j,J')
\sum_{\substack{s\in[n]\\ \rho_s+\tau'_s+\delta_{j,s}>0}}
(\rho_s+\tau'_s+\delta_{j,s})
T^{\rho+\tau'+\eps_j-\eps_s}_{I^-\wedge s\wedge(J'\setminus j)}.
\end{align*}
For $s=j$, the two wedge signs cancel and the term contributes $\rho_j+\tau'_j+1$ on the diagonal.  Hence
\[
 M=\operatorname{id}-\frac1dL,
\]
where $L$ has diagonal entries
\[
L_{(\tau,J),(\tau,J)}=\sum_{j\in J}(\rho_j+\tau_j+1)
\]
and possible off-diagonal entries
\[
L_{(\tau'-\eps_s+\eps_j,(J'\setminus j)\cup s),(\tau',J')}
=(\rho_s+\tau'_s)\epsilon(J';j,s)
\qquad(j\in J',\ s\notin J'),
\]
whenever the indicated row belongs to $\Omega$; all other entries vanish.
Put
\[
u=\tau_0+\mathbf1_{J_0},\qquad
C=\supp(u)\setminus I^-,\qquad t=\rho+u,\qquad
t_C=(t_i)_{i\in C}.
\]
The map $(\tau,J)\mapsto J$ identifies $\Omega$ with all $m$-subsets of $C$; its inverse is $J\mapsto(u-\mathbf1_J,J)$.  Thus, for $j\in J$ and for an off-diagonal exchange with $s\notin J'$,
\[
\rho_j+\tau_j+1=t_j,
\qquad
\rho_s+\tau'_s=t_s.
\]
On $E_C={\rm span}\{ e_i:i\in C\}$, the displayed entries therefore identify
\[
L=m\,\operatorname{id}+\Gamma_{t_C}^{\wedge m}.
\]

Set $c=|C|$ and $N_C=\sum_{i\in C}t_i$.  Proposition~\ref{cor:Gamma-spectrum} gives
\[
\det(x\operatorname{id}-M)
=\left(x-\frac{d-N_C}{d}\right)^{\binom{c-1}{m-1}}
 (x-1)^{\binom{c-1}{m}}.
\]
Finally, $N_C\le|t|=a-1+m$ gives
\[
1\le b+1\le d-N_C\le d.
\]
Thus $d$ and $d-N_C$ are units under the characteristic hypothesis, so $M$ is invertible.
\end{proof}


\section{Straightening for Hook Tabloids}\label{sec:elim-systems}

This section combines the elimination matrices of \S\ref{sec:inv-matrices}
with the Pieri compositions of \S\ref{sec:hook-schur} to straighten every
nonsemistandard rank two hook tabloid. This is the hook analogue of
characteristic-free straightening for Schur modules developed by
Akin--Buchsbaum--Weyman~\cite{AkinBuchsbaumWeyman1982}; see also the
determinantal straightening law of De Concini--Eisenbud--Procesi~
\cite{DeConciniEisenbudProcesi1980}.

For a tabloid of shape $\lambda$, assume throughout that $\operatorname{char}\bk=0$ or $\operatorname{char}\bk>|\lambda|$.  Together with the combinatorial hypotheses below, this makes the required elimination matrices invertible by Theorems~\ref{thm:row-elimination-invertible} and~\ref{thm:Theta-elim-invertible}; see Remark~\ref{rmk:general-rank-straightening}.
We retain the notation of \S\ref{sec:hook-schur}: \(\strun\) is the strict
part of the running order (Definition~\ref{def:running-order}),
\(R_{(a,b),(c,d)}\) the quadratic relation spaces of
\S\ref{ssec:universal-hook-algebra}, and
\(T_I^\alpha,\eps_i,\mathbf1_J,\supp,\sgn\) the \(d\)-model symbols.

\subsection{Columns}\label{ssec:cols-strgt}

We straighten failures of semistandardness in the column direction using the row Pieri compositions
$\Phi_m,\Psi_{m,r}$ of Definition~\ref{def:hook-pieri-compositions}, whose
intermediate factor is $S^mV$.

\smallskip
\noindent\emph{Straightening within a hook factor.}
If a \(d\)-model symbol \(T_D^\gamma\) is not a semistandard hook tableau,
then \(x:=\min\supp(\gamma)<\min D\), and the defining relation from
\S\ref{sec:hook-schur} gives
\begin{subequations}
  \begin{equation}\label{eq:d-model-straightening}
T_D^\gamma
=-\sum_{i\in D}\sgn(i,D\cup x)\,
T_{(D\cup x)\setminus i}^{\gamma-\eps_x+\eps_i}.
\end{equation}
Each summand is semistandard and replaces the smallest arm label \(x\) by a
larger label \(i\). We call this operation \emph{straightening}.

Suppose \(H,H'\) differ in the \(t\)th hook by such replacements. Denote their sorted \(t\)th arms 
\(r_1\le\cdots\le r_A\) and \(g_1\le\cdots\le g_A\), respectively. Set
\(h:=\max\{\nu:r_\nu\ne g_\nu\}\). If \(r_h<g_h\), then
\begin{equation}\label{eq:d-model-straightening-running-vector}
\bigl(V_{t,h-1}(H')-V_{t,h-1}(H)\bigr)_q
=
\begin{cases}
1,&r_h\le q<g_h,\\
0,&\text{otherwise}.
\end{cases}
\end{equation}
\end{subequations}
All comparisons in the running order below are made after straightening
each nonsemistandard factor.
The key triangularity statement concerns
\(\Phi_m\colon\bS_{(c,1^d)}(V)\otimes\bS_{(a,1^b)}(V)\to
\bS_{(c+m,1^d)}(V)\otimes\bS_{(a-m,1^b)}(V)\).

\begin{lemma}[Row Pieri triangularity] \label{lem:Omega-elimination-lower}
Fix \(\beta\in\bN^n\) with \(|\beta|=a-1\), a subset \(J_0\subseteq[n]\) with \(|J_0|=b+1\), and $1\le m\le a-1$.  Let \(\tau_0\in\bN^n\) be the content of the $m$ largest labels in the multiset of content~$\beta$, and set \(\rho:=\beta-\tau_0\).  Assume that $m<a-1$, or that $m=a-1$ and
$\min(J_0)\le\min\supp(\tau_0)$.  Given $T^\alpha_I\in\bS_{(c,1^d)}(V)$, set
 \[
   y = T_{I}^{\alpha+\tau_0}\otimes T_{J_0}^{\rho}.
 \]
Suppose that $y$ is a hook tabloid.
Then there exist $L\in\bS_{(a,1^b)}(V)$ and scalars $\gamma_x$ such that
\[
\Phi_m(T_{I}^\alpha \otimes L ) = y + \sum_{x\strun y} \gamma_x\,x.
\]
\end{lemma}

\begin{proof}
Set \(s:=\min\supp(\tau_0)\).  Then $\rho$ is supported on labels at most~$s$, while $\tau_0$ is supported on labels at least~$s$.  Put
\[
  t:=\tau_0+\mathbf1_{J_0},\qquad
  K_0:=\{i\in J_0:i>s\},\qquad
  F:=J_0\setminus K_0,\qquad
  \Omega:=\Omega(J_0;\tau_0,K_0,\rho).
\]
Thus \(\Omega\) is the row swap set of
Definition~\ref{def:closure-set-row}.
For every \((\sigma,A)\in\Omega\), the identity
\(\sigma+\mathbf1_A=\tau_0+\mathbf1_{J_0}\), together with \(F\subseteq A\),
gives \(\supp(\sigma)\subseteq\{s,s+1,\dots\}\).

For \(w:=\rho+t\) and \(\rho_{\tau,J}:=w-\tau-\mathbf1_J\), the hypotheses
exclude the singular case of
Theorem~\ref{thm:row-elimination-invertible}, so we may choose coefficients
\(c_{\sigma,A}\) such that
\[
L:=\sum_{(\sigma,A)\in\Omega}c_{\sigma,A}T_A^{\rho+\sigma}
\]
satisfies
\[
  \rP_m(L) = e^{\tau_0} \otimes T^\rho_{J_0} + \sum_{(\tau, J) \notin \Omega} \gamma_{(\tau,J)}\, e^\tau \otimes T^{\rho_{\tau,J}}_J .
\]
Applying \(\rQ_m\) gives
\[
  \Phi_m(T_{I}^\alpha \otimes L) = \underbrace{T_{I}^{\alpha+\tau_0}\otimes T_{J_0}^{\rho}}_y + \sum_{(\tau, J) \notin \Omega} \gamma_{(\tau,J)}\, T_{I}^{\alpha+\tau} \otimes T^{\rho_{\tau,J}}_J .
\]
Straighten every nonsemistandard factor in each summand by \eqref{eq:d-model-straightening} before comparing it with \(y\), and fix a resulting hook tabloid \(x\).

\smallskip
\noindent\textbf{Case 1: \(\rho_{\tau,J}\ne\rho\).}
For some $(\sigma,A)\in\Omega$, the arm $\rho_{\tau,J}$ is a submultiset
of $\rho+\sigma$ of cardinality $|\rho|$. Since $\rho$ is supported on
labels at most~$s$ and $\sigma$ on labels at least~$s$,
\(\rho_{\tau,J}\) is obtained from \(\rho\) by at least one strict
replacement by a larger label. Straightening by
\eqref{eq:d-model-straightening} makes only further such replacements, so
\eqref{eq:d-model-straightening-running-vector}, applied in the second
block, gives \(x\strun y\).

\smallskip
\noindent\textbf{Case 2: \(\rho_{\tau,J}=\rho\).}
Set
\[
C:=\supp(t)\setminus F,\quad p:=b+1-|F|,\quad
N_F:=\sum_{i\in F}t_i,\quad N_C:=\sum_{i\in C}t_i,\quad
\kappa:=|\rho|+N_F,
\]
put \(E_C:={\rm span}\{ e_i:i\in C\}\) and
\(u_C:=\sum_{i\in C}t_ie_i\), and define
\(U_C\in\End(\bigwedge^pE_C)\) by
\[
U_C(e_{i_1}\wedge\cdots\wedge e_{i_p})
=\sum_{q=1}^p(-1)^{q-1}u_C\wedge
e_{i_1}\wedge\cdots\wedge\widehat{e_{i_q}}\wedge\cdots\wedge e_{i_p}, 
\]
with \(U_C(1)=0\) when \(p=0\).
With \(d:=a+b-m\), the proof of
Theorem~\ref{thm:row-elimination-invertible} gives
\[
M^{\mathrm{norm}}=\frac1d(\kappa\operatorname{id}+U_C),
\qquad
(M^{\mathrm{norm}})^{-1}
=\frac d\kappa\left(\operatorname{id}-\frac1{a+b}U_C\right).
\]
Indeed, \(U_C^2=N_CU_C\) and \(\kappa+N_C=a+b\); the characteristic
hypothesis makes \(\kappa\) a unit, since \(|\rho|>0\) when \(m<a-1\),
whereas \(F\ne\varnothing\) when \(m=a-1\). Applying the inverse, the
additional terms cancel by the identity \(u_C\wedge U_C(\omega)=0\) for
every \(\omega\in\bigwedge^pE_C\). Thus the component of \(\rP_m(L)\)
whose second factor has arm content \(\rho\) is
\[
e^{\tau_0}\otimes T^\rho_{J_0}
+\frac1\kappa\sum_{\substack{f\in F\\v\in\supp(t)\setminus J_0}}
\eta_{f,v}t_v\,
e^{\tau_0+\eps_f-\eps_v}\otimes T^\rho_{(J_0\setminus f)\cup v},
\]
where \(\eta_{f,v}\in\{\pm1\}\) is the sign obtained by replacing \(f\) by
\(v\) in the increasing wedge \(e_{J_0}\) and sorting. Moreover,
\(f\le s\le v\) and \(f\ne v\), hence \(f<v\). If the second factor is
nonsemistandard, straightening it by \eqref{eq:d-model-straightening} gives
\(x\strun y\) in the second block, as in Case~1. Otherwise the second block
agrees with that of \(y\), while the first arm has gained \(f\) and lost
\(v\). Straightening a nonsemistandard first factor does not change its
total content, and hence \(x\strun y\) since
\[
\bigl(V_{1,c+m-1}(x)-V_{1,c+m-1}(y)\bigr)_q
=
\begin{cases}
1,&f\le q<v,\\
0,&\text{otherwise}.
\end{cases} \qedhere
\]
\end{proof}

\begin{example}[The cancellation in Case~2]\label{ex:row-case2-cancellation}
Take \(a=4\), \(b=1\), \(m=2\), \(\rho=\eps_2\),
\(\tau_0=\eps_3+\eps_4\), and \(J_0=\{1,5\}\). Thus the distinguished
source is \(T^{234}_{15}\), and the swap set is
\[
T^{245}_{13},\qquad T^{235}_{14},\qquad T^{234}_{15}.
\]
Solving the elimination system gives
\[
L=-\frac3{10}T^{245}_{13}-\frac3{10}T^{235}_{14}
  +\frac65T^{234}_{15}.
\]
The Pieri formula produces \(e_1e_5\otimes T^2_{34}\) with coefficient
\(-1/3\) from \(T^{245}_{13}\), coefficient \(1/3\) from
\(T^{235}_{14}\), and coefficient \(0\) from \(T^{234}_{15}\). Its
coefficient in \(\rP_2(L)\) is therefore
\[
\left(-\frac3{10}\right)\left(-\frac13\right)
+\left(-\frac3{10}\right)\left(\frac13\right)=0.
\]
Thus \(e_1e_5\otimes T^2_{34}\) cancels. Here \(F=\{1\}\), so the column
\(\{3,4\}\) lies outside the swap set. The part of \(\rP_2(L)\) with
second arm content \(\rho\) is
\[
e_3e_4\otimes T^2_{15}
+\frac12e_1e_3\otimes T^2_{45}
+\frac12e_1e_4\otimes T^2_{35},
\]
as asserted in Case~2.
\end{example}


Let $a>c\ge1$ and $b>d\ge0$, and consider a rank two hook tabloid
\[
T=\bigl(T^{\,i_1\cdots i_{a-1}}_{\,t_1,\dots,t_{b+1}},\;
T^{\,j_2\cdots j_c}_{\,j_1,q_1,\dots,q_d}\bigr)
\]
whose shape has Frobenius coordinates
\((a-1,c-1\mid b,d)\). 
Assume that \(T\) fails semistandardness in the column direction, and let
\(k\) be the smallest index with
\(i_k\ge j_k\), so \(i_r<j_r\) for \(r<k\).
The column conditions \(i_r<j_r\) for \(r<k\), combined with the
weak increase along each hook row, force the entries along the
displayed chain \(j_1\le\cdots\le j_k\le i_k\le\cdots\le i_{a-1}\)
to be weakly increasing as in the diagram:
\usetikzlibrary{matrix,backgrounds}
\[
  \begin{adjustbox}{max width=.92\textwidth,center}
  \begin{tikzpicture}[baseline=(m-2-1.base)]
\matrix (m) [
  matrix of math nodes,
  row sep=0em,
  column sep=0.65em,
  nodes={
    minimum width=1.9em,
    minimum height=1.7ex,
    anchor=center
  },
  ampersand replacement=\&
]{
i_1 \& \leq \& i_2 \& \leq \& \cdots \& \leq \& i_{k-1} \& \leq \& i_k \& \leq \& i_{k+1} \& \leq \& \cdots \\
\rotatebox{-90}{$<$} \& \& \rotatebox{-90}{$<$} \& \& \& \&
\rotatebox{-90}{$<$} \& \& \rotatebox{-90}{$\geq$} \& \&  \& \& \\
j_1 \& \leq \& j_2 \& \leq \& \cdots \& \leq \& j_{k-1} \& \leq \& j_k \& \leq \& j_{k+1} \& \leq \& \cdots \\
};
\begin{scope}[on background layer]
  \draw[
    blue!20,
    line width=18pt,
    line cap=rect,
    line join=miter
  ]
  (m-3-1.center) -- (m-3-9.center) -- (m-1-9.center) -- (m-1-13.center);
\end{scope}
\end{tikzpicture}
  \end{adjustbox}
\]

\begin{definition}\label{def:rank2-col-obstruction-data}
For this index~\(k\), define the content vectors
\[
\tau_0:=\sum_{\nu=k}^{a-1}\eps_{i_\nu},
\qquad
\rho:=\sum_{\nu=2}^{k}\eps_{j_\nu},
\qquad
\sigma:=\sum_{\nu=k+1}^{c}\eps_{j_\nu},
\]
where an empty sum is zero, and the \defi{induced triple}
\[
\begin{gathered}
T_1=T^{\,i_1\cdots i_{k-1}}_{\,t_1,\dots,t_{b+1}} \in \bS_{(k,1^b)}(V), \qquad
T_2=T^{\,\rho+\tau_0}_{\,j_1,q_1,\dots,q_d}\in \bS_{(a,1^d)}(V),  \qquad
T_3=e^\sigma\in S^{c-k}(V),
\end{gathered}
\]
where $T_3$ is the symmetric monomial of content~$\sigma$.  Diagrammatically, the triple is
\[
\ytableausetup{boxsize=1.7em,centertableaux}
\begin{ytableau}
\scriptstyle t_1 & \scriptstyle i_1 & \scriptstyle \cdots & \scriptstyle i_{k-1} \\
\scriptstyle t_2 \\
\scriptstyle \vdots \\
\scriptstyle t_{b+1}
\end{ytableau}
\;\otimes\;
\begin{ytableau}
*(blue!20) \scriptstyle j_1 & *(blue!20) \scriptstyle j_2 & *(blue!20) \scriptstyle \cdots & *(blue!20) \scriptstyle j_k & *(blue!20) \scriptstyle i_k & *(blue!20) \scriptstyle \cdots & *(blue!20) \scriptstyle i_{a-1} \\
\scriptstyle q_1 \\
\scriptstyle \vdots \\
\scriptstyle q_d
\end{ytableau}
\;\otimes\;
\begin{ytableau}
\scriptstyle j_{k+1} & \scriptstyle \cdots & \scriptstyle j_c
\end{ytableau} \qedhere
\]
\ytableausetup{boxsize=1em,nocentertableaux} 
\end{definition}

Consider the map
\[
\begin{aligned}
\Psi_{a-k,c-k}\colon
\bS_{(k,1^b)}(V)\otimes\bS_{(a,1^d)}(V)\otimes S^{c-k}(V)
\longrightarrow \bS_{(a,1^b)}(V)\otimes\bS_{(c,1^d)}(V).
\end{aligned}
\]

\begin{proposition}\label{prop:rank2-local-row-elimination-ijk}
There exists 
\(\widetilde T_2 \in \bS_{(a,1^d)}(V)\) and scalars $\gamma_{T'}$ such that
\[
\Psi_{a-k,c-k} (T_1\otimes\widetilde T_2\otimes T_3 )
= T\;+\;\sum_{T'\strun T}\gamma _{T'}\,T'.
\]
Hence, modulo $R_{(a-1,b),(c-1,d)}$, one obtains the straightening relation
\(T=-\sum_{T'\strun T}\gamma_{T'}\,T'\).
\end{proposition}

\begin{proof}
The inequalities in the displayed chain show that \(\tau_0\) is the
content of the $a-k$ largest labels in the arm of~$T_2$. If $k>1$, then
$a-k<a-1$; if $k=1$, then
\[
\min\{j_1,q_1,\dots,q_d\}=j_1\le i_1=\min\supp(\tau_0).
\]
Thus Lemma~\ref{lem:Omega-elimination-lower} gives $\widetilde T_2$ such that
\[
  \Phi_{a-k}(T_1\otimes\widetilde T_2)
  =T^{i_1\cdots i_{a-1}}_{t_1,\dots,t_{b+1}}
   \otimes T^{j_2\cdots j_k}_{j_1,q_1,\dots,q_d}
   +\text{terms strictly smaller in the running order}.
\]
The remaining map $1\otimes\rQ_{c-k}$ adjoins the common content~$\sigma$ to every second arm, so Lemma~\ref{lem:rQ-monotone} preserves each strict comparison and gives the asserted expansion.

Finally,
$c-(c-k)=a-(a-k)=k\ge1$, so Proposition~\ref{prop:relations}(i) places the image of $\Psi_{a-k,c-k}$ in $R_{(a-1,b),(c-1,d)}$.
\end{proof}

\begin{example}\label{ex:col-obstruction-straighten}
  \ytableausetup{boxsize=1.5em, mathmode}
  Consider the hook tabloid of rank two
\[
\ytableausetup{boxsize=1.1em}
T=
\begin{ytableau}
1 & *(blue!25) 1 & *(blue!25) 2 & *(blue!25) 3\\
2 & *(blue!25) 1 & 1 & 1\\
3 & 4
\end{ytableau}
\;=\;
\bigl(T^{123}_{123},\,T^{11}_{14}\bigr).
\]
Semistandardness in the column direction first fails at \(k=1\) (the outer
row entry~\(1\) is at least the inner row entry~\(1\)); the shaded cells highlight the entries in
the displayed chain of Definition~\ref{def:rank2-col-obstruction-data}. The
associated data are
\(\tau_0=\eps_1+\eps_2+\eps_3\), \(\rho=0\), \(\sigma=2\eps_1\), and the
associated triple is
$T_1=T^0_{123}$, $T_2=T^{123}_{14}$, $T_3=e_1^2$.
In the listed order, the swap set of Definition~\ref{def:closure-set-row} is
\[
T^{123}_{14},\qquad T^{134}_{12},\qquad T^{124}_{13},
\]
with elimination matrix
\[
M=\begin{pmatrix}
\frac32 & \frac12 & \frac12\\[4pt]
\frac12 & \frac32 & \frac12\\[4pt]
\frac12 & \frac12 & \frac32
\end{pmatrix},
\qquad \det M=\frac52.
\]
With this ordering, \(M\mathbf c=(1,0,0)^{\mathsf T}\) has solution
$\mathbf c=\frac15(4,-1,-1)^{\mathsf T}$, which determines~$\widetilde T_2$.
Applying \(\Psi_{a-k,c-k}\) and passing modulo $R_{(3,2),(2,1)}$ gives
\[
\begin{aligned}
\underbrace{\begin{ytableau}
1&1&2&3\\
2&1&1&1\\
3&4
\end{ytableau}}_{T}
&\;\equiv\;
-\frac12
\underbrace{\begin{ytableau}
1&1&1&2\\
2&1&1&3\\
3&4
\end{ytableau}}_{T'_1}
\;+\;\frac12
\underbrace{\begin{ytableau}
1&1&1&3\\
2&1&1&4\\
3&2
\end{ytableau}}_{T'_2}
\;+\;\frac12
\underbrace{\begin{ytableau}
1&1&1&2\\
2&1&1&4\\
3&3
\end{ytableau}}_{T'_3}
\;-\;\frac12
\underbrace{\begin{ytableau}
1&1&1&3\\
2&1&1&2\\
3&4
\end{ytableau}}_{T'_4}.
\end{aligned}
\]

Each tabloid on the right is strictly smaller than~$T$: all five share content $\{1^5,2^2,3^2,4\}$ and agree through the final step $(2,2)$, while at the rightmost differing step $(2,1)$,
\[
\begin{aligned}
V_{2,1}(T) &= (4,6,8,9), 
& V_{2,1}(T'_1) &= (5,7,8,9), 
& V_{2,1}(T'_2) &= (5,7,9,9), \\
V_{2,1}(T'_3) &= (5,7,9,9), 
& V_{2,1}(T'_4) &= (5,6,8,9).
\end{aligned}
\]
Thus $V_{2,1}(T)$ is coordinatewise at most $V_{2,1}(T'_i)$ for every~$i$,
and hence $T'_i\strun T$ by Definition~\ref{def:running-order}.
\end{example}

\subsection{Rows}\label{ssec:rows-strgt}

We straighten failures of semistandardness in the row direction using the column Pieri compositions
$\Phi'_m,\Psi'_{m,r}$ of Definition~\ref{def:hook-pieri-compositions}, whose
intermediate factors are exterior powers. Since the running order
treats rows and columns asymmetrically, this case requires two elimination
systems, one for each hook, coupled by the antisymmetrization in
$\Psi'_{m,r}$.

\begin{notation}\label{not:col-section-d-convention}
In this subsection, unlike in \S\ref{ssec:cols-strgt}, \(d\) denotes the
cardinality of the second input column \(\{j_1,\dots,j_d\}\), rather than
its leg length. Thus that hook is $\bS_{(c,1^{d-1})}(V)$, and the resulting
relation is modulo $R_{(a-1,b),(c-1,d-1)}$.
\end{notation}

For $q\ge m$ and $r\ge0$, we first consider the map
\[
  \Psi'_{m,r} \colon \bS_{(a,1^b)}(V)\otimes\bS_{(c,1^q)}(V)\otimes \bigwedge^r V \to  \bS_{(a,1^{b+m})}(V) \otimes \bS_{(c,1^{q-m+r})} (V).
\]
It extracts $m$ column entries from the second hook, inserts them into the
first, and then adjoins the exterior factor to the second.

\begin{lemma}[Column triangularity] \label{lem:triple-two-stage-elim}
  Fix \(T_J^\alpha\in\bS_{(a,1^b)}(V)\), \(T_K^\beta\in\bS_{(c,1^q)}(V)\), and \(e_I\in\bigwedge^r V\). Fix \(1\le m<|K|\), let \(K_+\) be the \(m\) largest elements of~\(K\), and set \(K_-:=K\setminus K_+\). Assume one of the following:
    {\rm (1)} $\max J<\min K_+$, {\rm (2)} $I=\varnothing$, or {\rm (3)} $\max K_-<\min I$.
Set
\[
  y=  T_{J\cup K_+}^\alpha\otimes T_{K_-\cup I}^\beta \in \bS_{(a,1^{b+m})}(V) \otimes\bS_{(c,1^{q-m+r})}(V) .
\]
If $y$ is a hook tabloid, then there exists
\(L\in\bS_{(a,1^b)}(V)\otimes\bS_{(c,1^q)}(V)\otimes \bigwedge^r V\) and scalars $\gamma_x$ such that
\[
\Psi'_{m,r}(L) = y + \sum_{x\strun y} \gamma_x \cdot x.
\]
\end{lemma}

\begin{proof}
Set \(k_{\min}:=\min K_+\) and \(i_{\min}:=\min I\), with
\(i_{\min}:=+\infty\) when \(I=\varnothing\), and put
\[
\alpha_0:=\sum_{\substack{\ell>k_{\min}\\ \ell\notin J\cup K_+}}\alpha_\ell\,\eps_\ell,\qquad \rho_\alpha:=\alpha-\alpha_0,
\qquad
\beta_0:=\sum_{\substack{\ell>i_{\min}\\ \ell\notin K_-\cup I}}\beta_\ell\,\eps_\ell,\qquad \rho_\beta:=\beta-\beta_0.
\]
Using the column swap sets of Definition~\ref{def:Theta-closure-matrix}, set
\[
\Omega_\alpha :=\Omega(J;\alpha_0,K_+,\rho_\alpha),\qquad
\Omega_\beta := \Omega(K_-;\beta_0,I,\rho_\beta).
\]
If \(I=\varnothing\), then \(\beta_0=0\), \(\rho_\beta=\beta\), and
\(\Omega_\beta=\{(0,\varnothing)\}\), so the second elimination is trivial.

By Theorem~\ref{thm:Theta-elim-invertible}, there are coefficients
$c_{\sigma,A}$ and $d_{\tau,B}$ for which
\[
\begin{aligned}
L_\alpha&=\sum_{(\sigma,A)\in\Omega_\alpha} c_{\sigma,A}\,T_J^{\rho_\alpha+\sigma}\otimes e_A
\ \in\ \bS_{(a,1^b)}(V)\otimes\textstyle\bigwedge^m V,\\
L_\beta&=\sum_{(\tau,B)\in\Omega_\beta} d_{\tau,B}\,T_{K_-}^{\rho_\beta+\tau}\otimes e_B
\ \in\ \bS_{(c,1^{q-m})}(V)\otimes\textstyle\bigwedge^r V,
\end{aligned}
\]
have coefficient one at the indices $(\alpha_0,K_+)$ and $(\beta_0,I)$,
respectively, after applying $\rQ'_m$ and $\rQ'_r$, and coefficient zero at
every other index of the corresponding swap set.

Put \(\eta:=(-1)^{m|K_-|}\). The inequalities in the statement imply
\(T_{J\wedge K_+}^\alpha\otimes T_{K_-\wedge I}^\beta=y\).
By definition, for $(\sigma,A)\in\Omega_\alpha$, we have \(\sigma+\mathbf1_A=\alpha_0+\mathbf1_{K_+}\). Hence
$\min A\ge k_{\min}>\max K_-$, and we may set
\[
L := \eta\sum_{\substack{(\sigma,A)\in\Omega_\alpha\\ (\tau,B)\in\Omega_\beta}}
c_{\sigma,A}\,d_{\tau,B}\; T_J^{\rho_\alpha+\sigma} \otimes T_{K_-\wedge A}^{\rho_\beta+\tau} \otimes e_{B}.
\]
Applying $1\otimes\rP'_m\otimes 1$ extracts an $m$-subset $A'$ from the
column $K_-\cup A$ of the middle hook. The term $A'=A$ carries the sign
$\sgn(A,K_-\cup A)=\eta$, so
\begin{align*}
&(1\otimes\rP'_m\otimes 1)(L) = L_\alpha\otimes L_\beta + \\
&\qquad \qquad \eta\sum_{\substack{(\sigma,A) \in \Omega_\alpha\\
  (\tau,B) \in \Omega_\beta}}
  \mkern-9mu
  c_{\sigma,A} d_{\tau,B}
  \mkern-27mu
  \sum_{\substack{A'\subseteq K_-\cup A\\|A'|=m,\ A'\ne A}}
  \mkern-27mu  
\sgn(A',K_-\cup A)
\cdot
T_J^{\rho_\alpha+\sigma}\otimes e_{A'}\otimes
T_{(K_-\cup A)\setminus A'}^{\rho_\beta+\tau}\otimes e_{B}.
\end{align*}

We compute the Pieri images for the two swap sets \(\Omega_\alpha\) and \(\Omega_\beta\).
For a column swap set
\(\Omega:=\Omega(I^-;\tau_0,J_0,\rho)\) of
Definition~\ref{def:Theta-closure-matrix}, write
\[
t:=\rho+\tau_0+\mathbf1_{J_0},\qquad
C:=\supp(\tau_0+\mathbf1_{J_0})\setminus I^-,
\qquad N_C:=\sum_{\ell\in C}t_\ell.
\]
Set \(\Delta:=|t|+|I^-|\), the size of the resulting hook. Let
\((c_{\tau,A})\) be the solution of the elimination system for the tableau
indexed by \((\tau_0,J_0)\). The inverse computed in the proof of Theorem~\ref{thm:Theta-elim-invertible}, followed
by the Pieri formula of Proposition~\ref{prop:basic-hook-pieri-closed}(ii),
gives, for every \(D'\subseteq[n]\) with \(|D'|=|I^-|\),
\begin{equation}\label{eq:column-system-output}
\begin{aligned}
\sum_{(\tau,A)\in\Omega}c_{\tau,A}\,
 \rQ'_{|J_0|}(T_{D'}^{\rho+\tau}\otimes e_A)=
T_{D'\wedge J_0}^{\rho+\tau_0}
-\frac1{\Delta-N_C}
 \sum_{\substack{f\in\supp(t)\setminus C\\v\in J_0}}
 \sgn(v,J_0)t_f\,
 T_{D'\wedge f\wedge(J_0\setminus v)}^
 {\rho+\tau_0+\eps_v-\eps_f}.
\end{aligned}
\end{equation}

For the data defining \(\Omega_\alpha\), every nonzero term in the sum on
the right of~\eqref{eq:column-system-output} satisfies
\(f<k_{\min}\le v\); for \(\Omega_\beta\), one has
\(f<i_{\min}\le v\). The latter assertion is vacuous when
\(I=\varnothing\).
For either system, each nonzero term in the sum on the right of
\eqref{eq:column-system-output} replaces an arm label~\(f\) by \(v>f\)
and the column label~\(v\) by~\(f\). After using \eqref{eq:d-model-straightening} to straighten any
nonsemistandard factor, \eqref{eq:d-model-straightening-running-vector} gives the strict comparison
in the corresponding hook.

\emph{The term \(A'=A\).} The corresponding summand is \(L_\alpha\otimes L_\beta\). Applying
\eqref{eq:column-system-output} with fixed columns \(J\) and \(K_-\)
produces \(y\) with coefficient 1. Every other resulting hook tabloid \(x\)
satisfies \(x\strun y\): the last differing running index lies in the second
block unless the second hooks agree, in which case it lies in the first.

\emph{The terms \(A'\ne A\).} Fix \(A'\ne A\), and put
\[
S:=A'\cap K_-,\qquad R:=A\setminus A',\qquad
D':=(K_-\setminus S)\cup R.
\]
Since \(A'=(A\setminus R)\sqcup S\), one has \(|S|=|R|>0\). Moreover,
every element of~\(S\) is smaller than every element of~\(R\).

Apply \eqref{eq:column-system-output} to the coefficients \(d_{\tau,B}\),
with fixed column~\(D'\). Its target term, when nonzero, has arm content
\(\beta\) and column \(D'\cup I\). Meanwhile, by definition of \(\Omega_\alpha\), every output of the first \(\rQ'_m\) has total
first hook content
\[
\alpha+\mathbf1_J+\mathbf1_{K_+}+\mathbf1_S-\mathbf1_R.
\]
If \(T_{D'\wedge I}^{\beta}\) is nonzero but nonsemistandard, then
straightening the second factor yields only tabloids \(x\strun y\), with the
comparison decided in the second block by
\eqref{eq:d-model-straightening-running-vector}. If it is nonzero and
semistandard, all running vectors in the second block agree with those of
\(y\): its column
contributes the opposite difference \(\mathbf1_R-\mathbf1_S\), while
\[
\bigl(V_{1,a-1}(x)-V_{1,a-1}(y)\bigr)_q
=|S\cap\{1,\ldots,q\}|-|R\cap\{1,\ldots,q\}|\ge0
\]
for every~\(q\), with strict inequality for some~\(q\). Thus
\((1,a-1)\) gives \(x\strun y\).

Fix an output of the first \(\rQ'_m\) and a nonzero term
\(T_{D'\wedge f\wedge(I\setminus\{v\})}^{\beta+\eps_v-\eps_f}\), where
\(f<v\), in the sum in \eqref{eq:column-system-output}; let \(x\) be any hook
tabloid obtained by straightening their product. Nonvanishing of the wedge implies
\(D'\cap(I\setminus\{v\})=\varnothing\). Since
\(D'=(K_-\setminus S)\cup R\), the content difference contributed by the
first hook and the second column is
\[
(\mathbf1_S-\mathbf1_R)
+(\mathbf1_R-\mathbf1_S+\eps_f-\eps_v)
=\eps_f-\eps_v.
\]
The running-order comparison following \eqref{eq:column-system-output}
therefore gives \(x\strun y\). This also covers
\(D'\wedge I=0\), when the target output of the second system vanishes.
\end{proof}


\noindent\begin{minipage}[t]{0.79\textwidth}
  \vspace{0in}
Let $a>c\ge1$ and $b\ge d\ge1$, and consider a rank two hook tabloid
\[
T= \Bigl(
T^{t_1,\dots, t_{a-1}}_{i_0,i_1,\dots,i_b},\;
T^{s_1,\dots, s_{c-1}}_{j_1,j_2,\dots,j_d}
\Bigr)
\]
that fails semistandardness in the row direction. Let $k$ be the smallest index with
$i_k>j_k$.
Since both hook columns are strictly increasing and $j_k<i_k$,
$j_1<\cdots<j_k<i_k<\cdots<i_b$; the highlighted entries record this chain.
\end{minipage}\hfill \begin{minipage}[t]{0.18\textwidth}
  \vspace{0in}
  \usetikzlibrary{matrix,backgrounds}
  \resizebox{\linewidth}{!}{
    \begin{tikzpicture}[baseline=(m-7-1.base)]
\matrix (m) [
  matrix of math nodes,
  column sep=0.8em,
  nodes={
    minimum width=1.6em,
    minimum height=1.35ex,
    anchor=center
  },
  ampersand replacement=\&
]{
i_1          \& \leq         \& j_1        \\
\rotatebox{-90}{$<$} \&      \& \rotatebox{-90}{$<$} \\[-2ex]
\vdots       \&              \& \vdots     \\[-1ex]
\rotatebox{-90}{$<$} \&      \& \rotatebox{-90}{$<$} \\
i_{k-1}      \& \leq         \& j_{k-1}    \\
\rotatebox{-90}{$<$} \&      \& \rotatebox{-90}{$<$} \\
i_k          \& >            \& j_k        \\
\rotatebox{-90}{$<$} \&      \& \rotatebox{-90}{$<$} \\
i_{k+1}      \&              \& j_{k+1}    \\
\rotatebox{-90}{$<$} \&      \& \rotatebox{-90}{$<$} \\[-2ex]
\vdots       \&              \& \vdots     \\ };
\begin{scope}[on background layer]
  \draw[
    blue!20,
    line width=14pt,
    line cap=rect,
    line join=miter
  ]
  (m-1-3.center) -- (m-7-3.center) -- (m-7-1.center) -- (m-11-1.center);
\end{scope}
\end{tikzpicture}}
\end{minipage}

\vspace{-.4in}
\begin{definition}\label{def:rank2-row-obstruction-data}
For this index~$k$, put
\[
L_0 := \{i_k,\dots,i_b\},
\qquad
R_0 := \{j_1,\dots,j_k\},
\qquad
I_0 := \{j_{k+1},\dots,j_d\}
\]
(with $I_0 = \varnothing$ if $k=d$), and define the \defi{induced triple}
\[
\begin{aligned}
T_1 = T^{\,t_1,\dots,t_{a-1}}_{\,i_0,i_1,\dots,i_{k-1}}
       \in\bS_{(a,1^{k-1})}(V),
\quad
T_2 = T^{\,s_1,\dots,s_{c-1}}_{\,R_0 \cup L_0}
       \in\bS_{(c,1^{b})}(V),\quad
T_3 = e_{I_0}\in \textstyle\bigwedge^{d-k}(V),
\end{aligned}
\]
where the column of $T_2$ is the displayed sorted union and
$e_{I_0}=e_{j_{k+1}}\wedge\cdots\wedge e_{j_d}$.  Thus $L_0$ is the outer
column block transferred to the second hook, $R_0$ is its retained inner
prefix, and \(I_0\) consists of the remaining inner column entries.
\end{definition}

\ytableausetup{boxsize=1.5em, mathmode}

\begin{proposition}\label{prop:rank2-local-column-elimination}
There exists $\widetilde T \in \bS_{(a,1^{k-1})}(V) \otimes \bS_{(c,1^b)}(V) \otimes \bigwedge^{d-k}(V)$ and scalars $\gamma_{T'}$ such that
\[
\Psi'_{b-k+1,d-k}(\widetilde T)
= T + \sum_{T'\strun T}\gamma_{T'}\,T'.
\]
Hence, modulo $R_{(a-1,b),(c-1,d-1)}$ (using the indexing convention of
Notation~\ref{not:col-section-d-convention}), one obtains the straightening
relation
\(T=-\sum_{T'\strun T}\gamma_{T'}\,T'\).
\end{proposition}

\begin{proof}
  Apply Lemma~\ref{lem:triple-two-stage-elim} with $m=b-k+1$, $r=d-k$,
  and
  \[
    \begin{gathered}
      I=\{j_{k+1},\dots, j_d\},\quad J = \{i_0,\dots,i_{k-1}\},\quad K = L_0 \cup R_0, \quad
      \alpha = \sum_{\nu=1}^{a-1}\eps_{t_\nu},\quad
      \beta = \sum_{\nu=1}^{c-1}\eps_{s_\nu}.
    \end{gathered}
  \]
  The displayed inequalities give $K_+=L_0$ and $K_-=R_0$. Strictness of the
  outer column gives $\max J=i_{k-1}<i_k=\min K_+$; if $I\ne\varnothing$,
  strictness of the inner column gives
  $\max K_-=j_k<j_{k+1}=\min I$.  Moreover,
  $1\le m<|K|=b+1$, and the two hook types are strictly comparable because
  $a>c$ and $b>d-1$. The lemma therefore gives the target term
  \[
    T^\alpha_{J\cup K_+}\otimes T^\beta_{K_-\cup I}=T.
  \]
  Finally, \((d-1)-(d-k)=b-(b-k+1)=k-1\ge0\), so
  Proposition~\ref{prop:relations}(ii) places the image of
  $\Psi'_{b-k+1,d-k}$ in $R_{(a-1,b),(c-1,d-1)}$.
\end{proof}

\begin{example}\label{ex:row-obstruction-straighten}
Consider
\[
\ytableausetup{boxsize=1.1em}
T=
\begin{ytableau}
1 & 1 & 2 & 3\\
*(blue!25) 3 & *(blue!25) 2 & 3 & 4\\
*(blue!25) 4 & 3
\end{ytableau}
\;=\;
\bigl(T^{123}_{134},\,T^{34}_{23}\bigr).
\]
Here \(k=1\), $(a,b,c,d)=(4,2,3,2)$, and the associated data are
\(L_0=\{3,4\}\), \(R_0=\{2\}\), \(I_0=\{3\}\), with triple
$T_1=T^{123}_1$, 
$T_2=T^{34}_{234}$,
$T_3=e_3$.
The first swap set consists of $T^{123}_1\otimes e_{34}$, with
$M_\alpha=(\frac12)$.  In the order
$T^{34}_2\otimes e_3,T^{33}_2\otimes e_4$, the second elimination matrix is
$M_\beta=
\begin{pmatrix}
\frac12 & -\frac12\\[4pt]
-\frac14 & \frac34
\end{pmatrix}$; the two solution vectors are $2$ and
$(3,1)^{\mathsf T}$.  Applying $\Psi'_{2,1}$ and passing modulo
$R_{(3,2),(2,1)}$ gives
\[
\begin{aligned}
\underbrace{\begin{ytableau}
1&1&2&3\\
3&2&3&4\\
4&3
\end{ytableau}}_{T}
&\;\equiv\;
-\frac13
\underbrace{\begin{ytableau}
1&1&3&4\\
2&2&3&4\\
3&3
\end{ytableau}}_{T'_1}
\;+\;\frac13
\underbrace{\begin{ytableau}
1&1&3&3\\
2&2&3&4\\
4&3
\end{ytableau}}_{T'_2}
+\;\frac23
\underbrace{\begin{ytableau}
1&1&2&3\\
2&3&3&4\\
3&4
\end{ytableau}}_{T'_3}.
\end{aligned}
\]

All four tabloids have content $\{1^2,2^2,3^4,4^2\}$ and agree at the
final step $(2,2)$.  Their last disagreements are
\[
\begin{array}{c|c|c|c}
i&(t,s)&V_{t,s}(T)&V_{t,s}(T'_i)\\ \hline
1&(1,2)&(2,3,4,5)&(2,3,5,5)\\
2&(1,1)&(2,2,3,4)&(2,3,3,4)\\
3&(1,3)&(2,3,5,6)&(2,4,6,6)
\end{array}
\]
In each row, $V_{t,s}(T)$ is coordinatewise at most $V_{t,s}(T'_i)$, so $T'_i\strun T$ by Definition~\ref{def:running-order}.
\end{example}

\begin{remark}[Straightening of arbitrary rank]\label{rmk:general-rank-straightening}
The passage from rank two to arbitrary rank is the one used in the spanning
argument in the proof of Theorem~\ref{thm:main-hook}(1): apply the appropriate
relation to adjacent hooks across which semistandardness fails, use
Lemma~\ref{lem:single-hook-insertion} to preserve each strict running
comparison, and use Proposition~\ref{prop:global-tabloid-order} for
termination. This algorithm is valid whenever $\bk$ has characteristic~$0$
or $>|\lambda|$: each rank two step inverts an elimination matrix of
Theorem~\ref{thm:row-elimination-invertible} or
Theorem~\ref{thm:Theta-elim-invertible}, and the parameters $a+b$ and
$a+b+m$ occurring in their characteristic bounds are at most~$|\lambda|$.
\end{remark}

\ytableausetup{smalltableaux}
Hook tabloids give a third straightening model alongside the Schur and Weyl
models of Akin--Buchsbaum--Weyman~\cite{AkinBuchsbaumWeyman1982}. The
algorithms differ, but their normal forms are indexed by the same
semistandard tableaux.

\begin{example}\label{ex:straighten-33}
The straightening relations in the hook basis need not have integer
coefficients: rational coefficients already appear in small examples, in
contrast to the integral Schur and Weyl normal forms.  For
\[
  T \;=\; \smallhooktableau{126,435}\,,
  \qquad\text{hooks }\ \smallhooktableau{126,4}\ \text{ and }\ \smallhooktableau{35},
\]
one finds, in the three respective models,
\begin{align*}
\text{(hook)}\quad  T &= -\tfrac12\,\smallhooktableau{123,456}
                        \;-\;\tfrac12\,\smallhooktableau{124,356}
                        \;-\;\smallhooktableau{125,346}
                       +\tfrac12\,\smallhooktableau{134,256}
                        \;+\;\smallhooktableau{135,246}\,,\\[4pt]
\text{(Schur)}\quad T &= -\,\smallhooktableau{125,346}
                        \;+\;\smallhooktableau{135,246}\,,\\[4pt]
\text{(Weyl)}\quad  T &= -\,\smallhooktableau{123,456}
                        \;-\;\smallhooktableau{124,356}
                        \;-\;\smallhooktableau{125,346}\,. \qedhere
\end{align*}
\end{example}
\ytableausetup{boxsize=1em,nosmalltableaux}

\section{The Structure of the Homology Algebras of Determinantal Rings}\label{sec:applications}

This section proves that the Tor algebras of generic
determinantal, symmetric determinantal, and Pfaffian rings are generated by
their first nonzero strand and are Koszul (Theorem~\ref{thm:koszul-strands}). The Eilenberg--Moore degeneration
then gives explicit rational expressions for the Poincar\'e series of their
residue fields (Corollary~\ref{cor:poincare-applications}). 

Throughout this section, $\bk$ denotes a field of characteristic $0$.

\subsection{Determinantal rings}\label{sec:det-poincare}

We are concerned with the following families.   Let \(V,W\) be vector spaces of dimensions~\(n,m\).

\begin{itemize}[$\bullet$]
\item \textbf{Generic determinantal.}
Set \(S:=\Sym^\bullet(V\otimes W)\), and let \(I_r\subset S\) be the
ideal of \(r\times r\) minors of the generic \(n\times m\) matrix, where \(2\le r\le\min(n,m)\).

\item \textbf{Symmetric determinantal.}
Set \(S:=\Sym^\bullet(S^2V)\), and let \(I_r\subset S\) be the ideal of \(r\times r\) minors of the generic symmetric matrix, where \(2\le r\le n\).

There is also a \(\GL(V)\)-equivariant algebra \(M_r\), finite over \(S/I_r\). For its construction see~\cite[\S4.1]{superresII}. The important takeaway is that \(M_r\) is a Gorenstein integral domain and that \(S/I_r\) is the fixed subalgebra of an involution on \(M_r\), and hence an \(S\)-module summand. The generators of $M_r$ have half the degree of the generators of $S$; in the formulas below, we will renormalize so that the generators of $M_r$ have degree 1.

\item \textbf{Pfaffian.}
Set \(S:=\Sym^\bullet(\bigwedge^2V)\), and let \(I_{2r}\subset S\) be the ideal of \(2r\times 2r\) Pfaffians of the generic skew-symmetric matrix, where \(2\le r\le\lfloor n/2\rfloor\).
\end{itemize}

\begin{theorem}\label{thm:koszul-strands}
  For \(r\ge2\), the Tor algebras of generic determinantal, symmetric determinantal, the algebras $M_r$, and Pfaffian rings are generated by their first nonzero strand and are Koszul with respect to the strand grading.
\end{theorem}

\medskip

\noindent\textbf{Strategy.}
The proof has three ingredients:
\begin{enumerate}[(a)]

\item Each Tor algebra~\(B\) is generated by its first nonzero strand
(\S\ref{subsec:strand-generation}).

\item The Pfaffian and $M_r$ cases degenerate to retracts supported on diagonal
hooks, the
symmetric quotient to the second Veronese of such a retract, and the generic
case to a fine Segre product of two retracts of universal hook algebras
(Proposition~\ref{prop:diagonal-hook-degeneration},
Corollary~\ref{cor:hook-degreewise-compatibility}, and
Proposition~\ref{prop:generic-degeneration}).

\item These associated graded algebras are Koszul
(Propositions~\ref{prop:diagonal-retracts} and
\ref{prop:generic-fine-segre}, together with the standard Veronese result),
and the filtrations preserve strand degree, so Koszulity lifts to~$B$.

\end{enumerate}

For $G=\GL(V)$ or $\GL(V)\times\GL(W)$, we identify equivariant classes of $G$-representations with their characters on the diagonal maximal torus, in variables $x_1,\dots,x_n$ and
$y_1,\dots,y_m$.  See \S\ref{subsec:partitions-schur} for partition notation and Definition~\ref{def:Q-sets} for  $Q_\delta$.

\begin{corollary}\label{cor:poincare-applications}
  The equivariant Poincar\'e series of the residue fields of the three determinantal
  quotient rings and \(M_r\)  admit the following  explicit rational expressions.

\smallskip
\noindent\textup{\textbf{(i) Generic determinantal.}}\quad
For $A=\Sym(V\otimes W)$ and $2\le r\le\min(n,m)$,
\[
\rP_{A/I_r}(t,u)
\;=\;
\frac{\DS \prod_{1\le i\le n,\,1\le j\le m}(1+x_iy_j\,tu)}
{\DS \sum_{\lambda\subseteq (m-r+1)^{n-r+1}}(-1)^{\rank\lambda}\,[\bS_{\lambda^{(r-1)}}(V)\otimes\bS_{(\lambda^{\mathsf t})^{(r-1)}}(W)]\,(tu)^{|\lambda|+\rank\lambda}\,u^{(r-2)\rank\lambda}}.
\]

\smallskip
\noindent\textup{\textbf{(ii) Symmetric determinantal.}}\quad
For $A=\Sym(S^2V)$ with $2\le r\le n$,
\[
\rP_{A/I_r}(t,u)
=\;
\frac{\DS \prod_{1\le i\le j\le n}(1+x_ix_j\,tu)}
      {\DS \sum_{\lambda\in Q_0^{\mathrm{ev}}}(-1)^{\rank\lambda/2}\,[\bS_{\lambda^{(r-2)}}(V)]\,(tu)^{|\lambda|/2}\,u^{(r-2)\rank\lambda/2}}.
\]
Renormalize the generators of $M_r$ to have degree $1$. For $3\le r\le n$,
\[
\rP_{M_r}(t,u)
=\;
\frac{\DS \prod_{1\le i\le j\le n}(1+x_ix_j\,tu^2)}
{\DS \sum_{\lambda\in Q_0}(-1)^{\rank\lambda}\,
[\bS_{\lambda^{(r-2)}}(V)]\,(tu^2)^{(|\lambda|+\rank\lambda)/2}\,
u^{(r-3)\rank\lambda}}.
\]
For $r=2$, one has $M_2=\Sym(V)$ with its standard grading, and hence
\[
  \rP_{M_2}(t,u)=\prod_{1\le i\le n}(1+x_i t u).
\]

\smallskip
\noindent\textup{\textbf{(iii) Pfaffian.}}\quad
For $A=\Sym(\bigwedge^2V)$ with $2\le r\le\lfloor n/2\rfloor$,
\[
\rP_{A/I_{2r}}(t,u)
\;=\;
\frac{\DS \prod_{1\le i<j\le n}(1+x_ix_j\,tu)}
{\DS \sum_{\lambda\in Q_{-1}}(-1)^{\rank\lambda}\,[\bS_{\lambda^{(2r-2)}}(V)]\,(tu)^{|\lambda|/2+\rank\lambda}\,u^{(r-2)\rank\lambda}}.
\]
\end{corollary}

\begin{proof}
\noindent\textit{Numerator.}  The Koszul resolution of $\bk$ over $A=\Sym(E)$ identifies $\Tor^A_i(\bk,\bk) \cong \bigwedge^iE$ in internal degree $i$, whence
\[
\rP_A(t,u)\;=\;\sum_{i\ge 0}[{\bigwedge}^iE]\,(tu)^i
\;=\; \prod_{\chi\in\mathrm{wt}(E)}\bigl(1+\chi\,tu\bigr),
\]
where \(\mathrm{wt}(E)\) is the multiset of maximal torus weights of~\(E\).
Specializing $E=V\otimes W$, $E=S^2V$, and $E=\bigwedge^2V$ produces the displayed
numerators for the quotient rings.

\smallskip
\noindent\textit{Denominator.} The Tor groups of our families are as follows:

\begin{itemize}[$\bullet$]
  
\item \textbf{Generic determinantal.}
By~\cite{Lascoux1978} (see also~\cite[\S6.1]{Weyman2003}),
\[
\Tor_p^S(S/I_r,\bk)
\;\cong\;
\bigoplus_{\substack{\lambda\subset(m-r+1)^{n-r+1}\\|\lambda|=p}}
\bS_{\lambda^{(r-1)}}(V)\otimes\bS_{(\lambda^{\mathsf t})^{(r-1)}}(W).
\]

\item \textbf{Symmetric determinantal.}
  By~\cite{JozefiakPragaczWeyman1981} (see also~\cite[\S6.3]{Weyman2003}),
\begin{align*}
\Tor_p^S(S/I_r,\bk)
&\cong\;
\bigoplus_{\substack{\lambda\in Q_0^{\mathrm{ev}}\\
(|\lambda|-\rank\lambda)/2=p}}
\bS_{\lambda^{(r-2)}}(V),\qquad 
\Tor_p^S(M_r,\bk)
&\cong\;
\bigoplus_{\substack{\lambda\in Q_0\\
(|\lambda|-\rank\lambda)/2=p}}
\bS_{\lambda^{(r-2)}}(V).
\end{align*}

\item \textbf{Pfaffian.}
By~\cite{JozefiakPragacz1979} and \cite{JozefiakPragaczWeyman1981} (see also~\cite[\S6.4]{Weyman2003}),
\[
\Tor_p^S(S/I_{2r},\bk)
\;\cong\;
\bigoplus_{\substack{\lambda\in Q_{-1}\\|\lambda|/2=p}}
\bS_{\lambda^{(2r-2)}}(V).
\]
\end{itemize}

These decompositions are multiplicity-free, with degrees as follows:
\begin{center}
\begin{tabular}{l|l|l|l}
Family & $p$ & $\alpha$ & internal degree \\
\hline
Generic determinantal & $|\lambda|$ & $\rank \lambda$ & $|\lambda|+(r-1)\rank\lambda$ \\
Symmetric determinantal & $(|\lambda|-\rank\lambda)/2$ & $\rank \lambda/2$ & $\bigl(|\lambda|+(r-2)\rank\lambda\bigr)/2$ \\
Pfaffian & $|\lambda|/2$ & $\rank \lambda$ & $\bigl(|\lambda|+(2r-2)\rank\lambda\bigr)/2$
\end{tabular}
\end{center}

For \(r\ge2\), the least positive strand gap is $\ell=r-1$, attained by the ideal
generators of degree~$r$. Theorem~\ref{thm:koszul-strands} therefore
verifies the hypotheses of Proposition~\ref{prop:em-rational-form}.

Substituting $(x,y)=(tu,-tu^{r-1})$ into $H(x,y)=\sum_{p,\alpha}[T_{p,p+(r-1)\alpha}]\,x^p y^\alpha$ gives
\[
H(tu,-tu^{r-1})
\;=\;
\sum_{p,\alpha\ge 0}(-1)^\alpha\,[T_{p,p+(r-1)\alpha}]\,t^{p+\alpha}\,u^{p+(r-1)\alpha}.
\]

For $M_r$, first halve its standard grading, so that
$A=\Sym(S^2V)\to M_r$ is degree preserving. Put $d=\rank\lambda$. The summand
indexed by $\lambda\in Q_0$ has
\[
  p=\frac{|\lambda|-d}{2},\qquad
  j_0=\frac{|\lambda|+(r-2)d}{2},\qquad
  j_0-p=\frac{r-1}{2}d,
\]
where $j_0$ is its internal degree in this $A$-compatible grading. Thus
$\ell=(r-1)/2$ and the strand degree is $d$. For $r\ge3$,
Proposition~\ref{prop:em-rational-form} gives the formula in this grading.
Restoring the standard grading on $M_r$ doubles every internal degree, so
substituting $u^2$ for $u$ gives the displayed formula. When $r=2$, the
invariant-theoretic description gives $M_2=\Sym(V)$ directly.
\end{proof}

\begin{remark}
  \label{rem:koszul-and-maximal-minors}
When \(r=2\), the three quotient rings in
Corollary~\ref{cor:poincare-applications} are standard graded Koszul
algebras. Hence, if
\(\operatorname{Hilb}_R(z)=\sum_{j\ge0}[R_j]z^j\), then  $\rP_R(t,u)=\operatorname{Hilb}_R(-tu)^{-1}$.
Without equivariance, this is the usual identity for a Koszul algebra~\cite{Priddy1970,Froberg1999}.
\end{remark}

\begin{remark}
  For generic maximal minors,
assume \(n\le m\) and \(r=n\). The Eagon--Northcott resolution gives
 $T_{i,j}=0$ for $i>0$ and $j\ne n+i-1$.
 Thus every positive Tor group has strand degree one. The denominator in Proposition~\ref{prop:em-rational-form} becomes
\[
 H(tu,-tu^{n-1})
 =1-\sum_{i=1}^{m-n+1}\binom{m}{n+i-1}\binom{n+i-2}{n-1} t^{i+1}u^{n+i-1}.
\]
After setting \(u=1\), this is the Golod
denominator~\cite{Golod1962Homology}. Avramov proved that
the quotient map defined by the maximal minors is Golod and observed that this
recovers Gover's calculation
\cite[Theorem~6.7 and Remark~6.8]{Avramov1978};
see also~\cite{Gover1973}.
\end{remark}

\begin{remark}
The formulas above also determine the local Poincar\'e series along every
rank stratum. Let \(x\) have rank \(q\) in the generic or symmetric case
and rank \(2q\) in the Pfaffian case. Schur-complement coordinates identify
a neighborhood of \(x\) with the product of a smooth variety of dimension
\[
 d=q(m+n-q),\qquad d=\frac{q(2n-q+1)}2,\qquad
 d=q(2n-2q-1),
\]
respectively, and the residual rank locus \(X_{\mathrm{res}}\) obtained by
replacing \((r,n,m)\) by \((r-q,n-q,m-q)\), \((r,n)\) by
\((r-q,n-q)\), or \((r,n)\) by \((r-q,n-2q)\). Thus, writing \(\rP_x^X(t)\) for the residue-field Poincar\'e series of \(\cO_{X,x}\), we have
\[
 \rP_x^X(t)=(1+t)^d\rP_0^{X_{\mathrm{res}}}(t). \qedhere
\]
\end{remark}

\begin{example}
\label{ex:nonkoszul-pfaffian}
Let $X=(x_{ij})$ be the generic $8\times8$ alternating matrix,
$S=\bk[x_{ij}:i<j]$, $I=I_6(X)$, and $R=S/I$.
The ring $R$ is minimally defined by cubic Pfaffians, so it is not Koszul,
whereas Corollary~\ref{cor:poincare-applications}(iii) gives
\[
  \rP_R(t,u)=
  \frac{(1+tu)^{28}}
  {1-28t^2u^3-63t^3u^4-36t^4u^5+36t^5u^7
    +63t^6u^8+28t^7u^9-t^9u^{12}}.
\]
Here the top Tor bidegree is $(6,12)$, so Poincar\'e
duality~\cite{AvramovGolod1971} gives
$D(t)=-t^9D(t^{-1})$ after setting $u=1$.
\end{example}

For submaximal minors of a generic $n\times n$ matrix, $n\ge3$, part~(i)
similarly recovers Avramov's formula~\cite[Theorem~8.9]{Avramov1978}.

\subsection{Further applications}

\subsubsection{Transposed and nonreduced families} \label{sec:further-koszul-examples}

Assume $\operatorname{char}\bk=0$. We record several transposed
determinantal families, whose defining ideals are generally nonradical.
Let \(\tau\) denote transpose duality on
polynomial functors (written as a superscript), so that $\bS_\lambda^\tau=\bS_{\lambda^{\mathsf t}}$.  This is an exact monoidal equivalence, with the Koszul sign rule, and therefore transports equivariant
resolutions and their Tor multiplication
\cite[\S7.4]{introtca}; see also~\cite{RaicuVandeBogert2025}.  It exchanges
the ambient algebra objects $\Sym(S^2V)$ and $\Sym(\bigwedge^2V)$, and sends $(a+1,1^b)$ to $(b+1,1^{a})$; hence it preserves hook rank and the strand
grading.  Equivalently, transposition corresponds to evaluation on a purely
odd vector superspace.

\begin{example}[Transposed Pfaffian and symmetric families]
  Set $n=\dim V$.  For $q\ge2$, transpose the Pfaffian ideal
  $I_{2q}\subset\Sym(\bigwedge^2{-})$ before evaluating on $V$.  The
  resulting ideal $I_{2q}^\tau\subset A^\tau=\Sym(S^2V)$ is generated by
  $\Sym^{2q}(V)\subset\Sym^q(S^2V)$.  Its Tor groups are
  \[
    \Tor_p^{A^\tau}(A^\tau/I_{2q}^\tau,\bk)
\;\cong\;
\bigoplus_{\substack{\lambda\in Q_{-1}\\|\lambda|/2=p}}
\bS_{(\lambda^{(2q-2)})^{\mathsf t}}(V),
  \]
and therefore
\[
\rP_{A^\tau/I_{2q}^\tau}(t,u)
\;=\;
\frac{\DS \prod_{1\le i \le j\le n}(1+x_ix_j\,tu)}
{\DS \sum_{\lambda\in Q_{-1}}(-1)^{\rank\lambda}\,[\bS_{(\lambda^{(2q-2)})^{\mathsf t}}(V)]\,(tu)^{|\lambda|/2+\rank\lambda}\,u^{(q-2)\rank\lambda}}.
\]
For $r\ge2$, transposing the stable symmetric determinantal family similarly
gives the ideal in $\Sym(\bigwedge^2V)$ generated by
$\bS_{(r,r)}(V)\subset\Sym^r(\bigwedge^2V)$. The algebra $M_r$ introduced in
\S\ref{sec:det-poincare} has a transposed analogue as well; its Tor decomposition
and Poincar\'e series are obtained by transposing every Schur functor in the
symmetric formulas.
\end{example}

\begin{example}[Exterior and permanental transposes]
\label{ex:exterior-det}
   Set $n=\dim V$, $m=\dim W$, and $2\le r\le\min(n,m)$.  The Cauchy
   decompositions
   $\Sym(V\otimes W)\cong\bigoplus_\lambda\bS_\lambda(V)\otimes\bS_\lambda(W)$
   and
   $\bigwedge(V\otimes W)\cong\bigoplus_\lambda\bS_\lambda(V)\otimes
   \bS_{\lambda^{\mathsf t}}(W)$ show that transpose duality on the \(W\) factor
   carries $I_r$ to the ideal
   $J_r^{\rm ext}\subset\bigwedge(V\otimes W)$ generated by
   $\bigwedge^rV\otimes\Sym^rW$.  Its quotient has Poincar\'e series
\[
\rP_{\bigwedge(V\otimes W)/J_r^{\rm ext}}(t,u)=
\frac{\DS \prod_{1\le i\le n,\,1\le j\le m}(1-x_iy_j\,tu)^{-1}}
{\DS \sum_{\lambda}(-1)^{\rank\lambda}\,[\bS_{\lambda^{(r-1)}}(V)\otimes\bS_{((\lambda^{\mathsf t})^{(r-1)})^{\mathsf t}}(W)]\,(tu)^{|\lambda|+\rank\lambda}\,u^{(r-2)\rank\lambda}}.
\]
Here and in the next display, \(\lambda\) ranges over all partitions.

Transposing both variables gives the generally nonradical ideal
$J_r^{\rm per}\subset\Sym(V\otimes W)$ generated by
$\Sym^rV\otimes\Sym^rW$, equivalently the $\GL(V)\times\GL(W)$-stable
ideal generated by an $r\times r$ permanent.  Its quotient has Poincar\'e
series
\[
\rP_{\Sym(V\otimes W)/J_r^{\rm per}}(t,u)=
\frac{\DS \prod_{1\le i\le n,\,1\le j\le m}(1+x_iy_j\,tu)}
{\DS \sum_{\lambda}(-1)^{\rank\lambda}\,[\bS_{(\lambda^{(r-1)})^{\mathsf t}}(V)\otimes\bS_{((\lambda^{\mathsf t})^{(r-1)})^{\mathsf t}}(W)]\,(tu)^{|\lambda|+\rank\lambda}\,u^{(r-2)\rank\lambda}}.
\]
\end{example}

\begin{example}[Skew determinants] \label{ex:skew-dets}
Let $n=\dim V$ and fix $1\le q\le\lfloor n/2\rfloor$.  Set
$A=\Sym(\bigwedge^2V)$, and let $J_q\subset A$ be the nonradical ideal
generated by
$\bS_{(2^{2q})}(V)\subset\Sym^{2q}(\bigwedge^2V)$.  Equivalently,
$J_q$ is the $\GL(V)$-stable ideal generated by the determinant of one
$2q\times2q$ principal submatrix; it is not the ideal of all
minors of size~$2q$.  For $q=1$, $A/J_1$ is the second Veronese subalgebra of
$\bigwedge^\bullet V$, and hence is Koszul~\cite{BackelinFroberg1985}.  In general,
\cite[Proposition~4.1.4]{SamVandeBogertOSP} gives
\[
  \Tor_i^{A}(A/J_q, \bk)
  \cong \bigoplus_{\substack{\lambda \in Q_0^{\rm ev} \\ (|\lambda|-\rank\lambda)/2 = i}} \bS_{\lambda^{(2q-2)}}(V).
\]
In fact, the same construction without the even condition yields a finitely
generated module $M$ over $A/J_q$, which can be shown to be Gorenstein.
Moreover, $A/J_q$ is the invariant subring under an involution, as in the
case of symmetric minors. This is not discussed in
\cite{SamVandeBogertOSP}, but can be explained as follows. The module $M$
is Cohen--Macaulay, and the terms of its resolution are equivariantly
self-dual. Subject to equivariance, its presentation matrix is unique up
to possible non-essential rescaling. Thus its resolution is also self-dual.

Transposition gives the analogous ideal
$J_q^\tau\subset\Sym(S^2V)$ generated by
$\bS_{(2q,2q)}(V)\subset\Sym^{2q}(S^2V)$.  By monoidality, its Tor algebra
is the transpose of the Tor algebra above; in particular, its additive
decomposition is obtained by transposing every Schur functor
\cite[Remark~4.1.5]{SamVandeBogertOSP}.
This gives analogous Poincar\'e series formulas for both families and their
extensions $M$.
\end{example}

\subsubsection{Super homogeneous spaces} \label{sec:superspaces}

We use the notation of \cite[\S6]{SamSnowdenFlagSupervarieties}.  Let
$X=\Gr_{r|s}(\mathbb C^{n|m})$, with $r\ge s$, and set
\[
  a=\min(m,n-r+s),\qquad
  A=\rH^*_{\mathrm{sing}}(\Gr_s(\mathbb C^a),\mathbb C),\qquad
  B=\rH^*(X,\cO_X),
\]
and assume for the moment that $a>0$. Let $A_+$ be the ideal of elements of
positive degree in $A$. Put
$W_0=\operatorname{Hom}(\mathbb C^n,\mathbb C^m)$,
$W_1=\operatorname{Hom}(\mathbb C^m,\mathbb C^n)$, and
$W:=W_0\oplus W_1$ and $S=\Sym(W^*)$. Let $Z=Z_0\times W_1$, where
$Z_0\subset W_0$ is cut out by the
$(a+1)$-minors. Let \(\widetilde Z\to Z\) be the finite flat cover of
degree \(\binom{a}{s}\) defined in
\cite[\S5.1]{SamSnowdenFlagSupervarieties}, and put
\[
  T_Z=\Tor^S_\bullet(\cO_Z,\mathbb C)
\]
with its strand grading.  By
\cite[Theorem~6.1 and Proposition~6.13]{SamSnowdenFlagSupervarieties},
$A=B^{\GL_{n|m}(\mathbb C)}$ and the filtration on $B$ induced by the
spectral sequence
restricts to the degree filtration
\(F^pA=\bigoplus_{j\ge p}A^j\).  We call
$\overline B=B/A_+B$ the \emph{determinantal fiber}.

For a commutative algebra \(A\), a graded \(A\)-algebra
\(C=\bigoplus_{d\ge0}C_d\) with \(C_0=A\) is \defi{Koszul over \(A\)}
if \(A=C/C_+\) has a linear graded projective \(C\)-resolution. This
relative grading places all of \(A\) in degree~\(0\), independently of
its cohomological grading.

\begin{proposition}\label{prop:super-grassmannian-determinantal-fiber-koszul}
The \(A\)-algebra \(B\) is Koszul over \(A\), and
its determinantal fiber \(\overline B\) is Koszul for the grading
\(\overline B_d=\overline B^{ad}\). If \(s=0\), then \(A=\mathbb C\) and
\(\overline B=B\).
\end{proposition}

\begin{proof}
The multiplicative spectral sequence of
\cite[Proposition~6.4 and Remark~2.3]{SamSnowdenFlagSupervarieties}
degenerates at its first page by
\cite[Proposition~6.10]{SamSnowdenFlagSupervarieties}; hence
\[
  \gr B\cong\Tor^S_\bullet(\cO_{\widetilde Z},\mathbb C).
\]
The algebra map \(\cO_Z\to\cO_{\widetilde Z}\), together with
multiplication by \(A=\Tor^S_0(\cO_{\widetilde Z},\mathbb C)\), gives a
homomorphism
\[
 A\otimes_{\mathbb C}T_Z\longrightarrow
 \Tor^S_\bullet(\cO_{\widetilde Z},\mathbb C).
\]
It is an isomorphism of \(A\)-modules by
\cite[Proposition~6.8]{SamSnowdenFlagSupervarieties}, hence an isomorphism
of algebras. A linear \(T_Z\)-resolution of \(\mathbb C\) remains linear
and free after tensoring with \(A\), so \(\gr B\) is Koszul over \(A\).
The May spectral sequence \cite[\S4, Theorem~3]{May1966}
\(\Tor^{\gr B}(A,A)\Longrightarrow\gr\Tor^B(A,A)\) is therefore concentrated
on the diagonal, so \(B\) is Koszul over~\(A\).

Consequently, quotienting by \(A_+\) gives a surjection
\[
T_Z\cong(\gr B)/(A_+\gr B)\twoheadrightarrow\gr\overline B.
\]
It is an isomorphism: by
\cite[Theorem~6.1(b),(c)]{SamSnowdenFlagSupervarieties}, its source and
target have the same Hilbert series.  Finally,
\cite[Theorem~6.6]{SamSnowdenFlagSupervarieties} identifies the regrading
above with the strand grading.  If $a<\min(m,n)$, Koszulity follows from
Theorem~\ref{thm:koszul-strands} and the standard filtration spectral
sequence; if $a=\min(m,n)$, then $Z=W$, so $T_Z=\mathbb C$ and the conclusion
is immediate.
\end{proof}

\begin{remark}
If $a=0$, then $B=\mathbb C$ when $m=0$; otherwise $(r,s)=(n,0)$ and $B$
is the exterior algebra on $(\mathbb C^m)^*\otimes\mathbb C^n$, concentrated
in cohomological degree~$0$ and Koszul in exterior degree. The fiber
isolates an ordinary connected graded algebra over \(\mathbb C\), whereas
\(B\) is treated relatively over \(A\). This distinction is substantive:
if $s=1$, then $A=\mathbb C[h]/(h^a)$, so is not itself Koszul after normalizing $\deg(h)=1$ for $a \ge 3$.
\end{remark}

\begin{remark}
The coherent cohomology algebras of the periplectic Grassmannians in
\cite{superresII} are likewise Koszul over their invariant
cohomology subalgebras, and their fibers modulo the invariants of positive
degree are Koszul. This follows from
\cite[Theorem~5.1(b),(c), Propositions~5.5, 5.6, 5.8, and~5.13, and
Corollary~5.10]{superresII}; the endpoint and double-cover cases are covered
by \cite[Remark~5.2, \S4.1, and Theorem~4.3(d)]{superresII}.
\end{remark}

\subsection{The Filtration Machinery}\label{subsec:filtration}

We return to the proof of Theorem~\ref{thm:koszul-strands}. Fix \(L\in\mathbb Z_{\ge0}\). For \(i\ge0\), let
\(h_i=(i{+}1,1^{L+i})\), with Frobenius pair \((i\mid L{+}i)\), and for a
finite set \(I\subset\mathbb Z_{\ge0}\), let \(h_I\) be its Frobenius union
(Definition~\ref{def:frobenius-union}). Write
\(c^\gamma_{\alpha,\beta}\) for the Littlewood--Richardson coefficient.

\begin{notation}\label{not:multiset-union}
For finite sets \(I,J\), write \(I\uplus J\) for their multiset union: an element of \(I\cap J\) has multiplicity two, and other elements of \(I\cup J\) have multiplicity one. We lexicographically sort multisets of integers by treating them as weakly decreasing sequences.
\end{notation}

\begin{lemma}\label{lem:multi-jump-ladder}
Let $I,J,K\subset\mathbb Z_{\ge0}$ be finite sets with
$\#K=\#I+\#J=:d$. Set $M=I\uplus J$. If $c^{h_K}_{h_I,h_J}\ne0$, then $K\ge_{\mathrm{lex}}M$.  Equality holds if and only if $I\cap J=\varnothing$ and $K=I\cup J$.
\end{lemma}

\begin{proof}
  Let $k_1=\max(K)$ and $m_1=\max(M)$.   We argue by induction on $d$, the case $d=0$ being immediate.  The
Littlewood--Richardson condition implies $h_I,h_J\subseteq h_K$, and hence
$k_1\ge m_1$.  If the inequality is strict, then $K>_{\mathrm{lex}}M$.
Otherwise, after interchanging $I$ and $J$, assume $k_1=m_1\in I$.  Deleting
one copy of $m_1$ from $M$, put
$K'=K\setminus\{k_1\}$, $I'=I\setminus\{m_1\}$, and
$M'=I'\uplus J$. If \(S=(s_1>\cdots>s_p)\), then $h_S\leftrightarrow (s_1,\ldots,s_p\mid L+s_1,\ldots,L+s_p)$.

Deleting the common first row and column of $h_K$ and $h_I$ therefore
identifies, up to translation, the skew diagrams $h_K/h_I\cong h_{K'}/h_{I'}$.
The Littlewood--Richardson rule therefore gives
  $c^{h_K}_{h_I,h_J}   =c^{h_{K'}}_{h_{I'},h_J}\ne 0$.
Induction gives $K'\ge_{\mathrm{lex}}M'$, and hence
$K\ge_{\mathrm{lex}}M$.  Equality is possible precisely when $M$ has no repeated entry and $K=M$, as stated.
\end{proof}

\begin{example}
For \(L=1\), \(I=\{2\}\), and \(J=\{1\}\), size forces a two-element
set~\(K\) to be either \(\{1,2\}\) or \(\{0,3\}\). The first is the Frobenius
union, while the second is lexicographically larger and therefore has
smaller filtration degree below.
\end{example}

Fix a finite set \(\Omega\subset\mathbb Z_{\ge0}\) and a parity function
$\pi\colon\Omega\to\mathbb Z/2$. Let
\[
  \cH^\pi_{\Omega,L}(V):=
  \bigoplus_{I\subseteq\Omega}\bS_{h_I}(V)
\]
be the parity form of the hook subalgebra supported on
\(\{h_i:i\in\Omega\}\), with the generator indexed by \(h_i\) assigned
parity \(\pi(i)\). Write
\(\cH_{\Omega,L}(V):=\cH^0_{\Omega,L}(V)\). The Zhang twist relating the
parity forms in \S\ref{ssec:universal-hook-algebra} restricts to this
subalgebra; hence \(\cH^\pi_{\Omega,L}(V)\) is a Zhang twist of
\(\cH_{\Omega,L}(V)\).

Choose rational numbers $\varepsilon_i>0$, $i\in\Omega$, satisfying
\[
  \varepsilon_k>
  2\sum_{\substack{j\in\Omega\\j<k}}\varepsilon_j,
  \qquad
  \sum_{i\in\Omega}\varepsilon_i<1,
\]
and set
$d(I)=\sum_{i\in I}(i+1-\varepsilon_i)$.

\begin{proposition}[Diagonal hook degeneration]
\label{prop:diagonal-hook-degeneration}

Let \(B=\bigoplus_{I\subseteq\Omega}B_I\) be a supercommutative
$\GL(V)$-algebra generated by the singleton components, such that
\[
  B_I\cong\bS_{h_I}(V),\qquad  \deg_{\mathbb Z/2}B_I=\sum_{i\in I}\pi(i),
\]
and $B$ is graded by $\deg B_I = \#I$.
Assume that $B_I\cdot B_J$ has nonzero projection to $B_{I\cup J}$ whenever
$I\cap J=\varnothing$.  Give $B$ the ascending filtration
\[
  F_{\le q}B:=\bigoplus_{d(I)\le q}B_I.
\]
Then \(\gr_F(B)\cong\cH^\pi_{\Omega,L}(V)\) as graded $\GL(V)$-equivariant superalgebras.
\end{proposition}

\begin{proof}
Let $B_K$ be a component with nonzero projection from $B_I\cdot B_J$. The grading gives $\#K=\#I+\#J$, while equivariance gives $c^{h_K}_{h_I,h_J}\ne0$ and $|h_K|=|h_I|+|h_J|$.  Since
\(|h_S|=\sum_{s\in S}(2s+L+1)=2\sum_{s\in S}(s+1)+(L-1)\#S\) for every
finite \(S\), these equalities give
\(\sum_{k\in K}(k+1)=\sum_{i\in I}(i+1)+\sum_{j\in J}(j+1)\), so the
comparison of \(d(K)\) with \(d(I)+d(J)\) depends only on the
corresponding sums of the \(\varepsilon_i\). With
$M=I\uplus J$, Lemma~\ref{lem:multi-jump-ladder} gives
$K\ge_{\mathrm{lex}}M$.  If \(K>_{\mathrm{lex}}M\), let \(a\) be the largest
index with different multiplicities in \(K\) and \(M\).  After their common
initial segment is removed, $K$ contains one additional copy of $a$, while
every lower index occurs at most twice in $M$.  Hence
\[
  \sum_{k\in K}\varepsilon_k-
  \sum_{m\in M}\varepsilon_m
  \ge \varepsilon_a-
  2\sum_{\substack{j\in\Omega\\j<a}}\varepsilon_j>0.
\]
Thus $d(K)<d(I)+d(J)$.  Equality occurs exactly when $I$ and $J$ are
disjoint and $K=I\cup J$.  This proves that $F_\bullet B$ is
multiplicative.

Write $\overline B_I$ for the image of $B_I$ in $\gr_F B$.  If
$I\cap J=\varnothing$, the only component in filtration degree
$d(I)+d(J)$ is $\overline B_{I\cup J}$, and its projection is nonzero by
hypothesis.  If $I\cap J\ne\varnothing$, every component has strictly
smaller filtration degree, so $\overline B_I\cdot\overline B_J=0$.
The singleton components therefore satisfy precisely the quadratic
multiplication rules defining $\cH^\pi_{\Omega,L}(V)$.  After choosing equivariant isomorphisms on these components (which are unique up to scalar), they induce a surjection
\[
  \cH^\pi_{\Omega,L}(V)\twoheadrightarrow\gr_F B.
\]
Both sides have underlying $\GL(V)$-module $\bigoplus_{I\subseteq\Omega}\bS_{h_I}(V)$, so this is an isomorphism.
\end{proof}

\begin{remark}\label{rem:filtration-subtlety}
The degeneration is generally nontrivial.  For $\dim V=6$, a
\texttt{Macaulay2}~\cite{M2} computation in the $r=2$ Pfaffian Tor algebra
shows that the product
\[
  \bS_{(2,1^4)}(V)\cdot\bS_{(2,1^4)}(V)\subseteq\Tor_2\cdot\Tor_2
\]
contains $\bS_{(3,2^4,1)}(V)\subseteq\Tor_4$, while the corresponding product in $\cH^\pi_{\Omega,L}(V)$ vanishes.
\end{remark}

\begin{corollary}\label{cor:hook-degreewise-compatibility}
If algebras \(B^{(1)},\dots,B^{(s)}\) satisfy the hypotheses of
Proposition~\textup{\ref{prop:diagonal-hook-degeneration}}, with corresponding
algebras \(\cH^{\pi_i}_{\Omega_i,L_i}(V_i)\), and \(E\) is obtained from them
by iterated Veronese and Segre operations in the strand grading, equipped with
the induced filtrations, then \(\gr(E)\) is obtained by applying the same
operations to the algebras \(\cH^{\pi_i}_{\Omega_i,L_i}(V_i)\).
\end{corollary}

\begin{proof}
The filtration preserves strand degree, so
\(F_{\le q}(A^{(d)})=F_{\le q}A\cap A^{(d)}\) and
\(\gr(A^{(d)})\cong(\gr A)^{(d)}\). On a Segre product
\(A\circ C:=\bigoplus_{n\ge0}A_n\otimes_\bk C_n\), set
\(F_{\le q}(A_n\otimes C_n):=\sum_{q_1+q_2\le q}F_{\le q_1}A_n
\allowbreak\otimes_\bk F_{\le q_2}C_n\); exactness of \(\otimes_\bk\) gives
\(\gr(A\circ C)\cong(\gr A)\circ(\gr C)\) degreewise, and the assertion
follows by iteration.
\end{proof}

\subsection{Strand Generation}\label{subsec:strand-generation}

We prove generation by the first nonzero strand for all three
families. The Pfaffian and symmetric degenerations then follow from
Proposition~\ref{prop:diagonal-hook-degeneration}; the generic degeneration
is constructed below. For the symmetric family with $r=2$, this property for the second Veronese
subalgebra of a polynomial ring was proved in \cite{ConcaKatthanReiner2021}.

For each product below, the union summand has Littlewood--Richardson
multiplicity one. By naturality, the projection onto this summand is either
identically zero or nonzero in every dimension in which the summand is
nonzero.

\subsubsection{Pfaffian case}

Fix \(r\ge 2\), set \(A=\Sym(\bigwedge^2 V)\), and let \(I_{2r}\) be the ideal
of Pfaffians of order~\(2r\). 
Set \(B:=\Tor^A(A/I_{2r},\bk)\). As a \(\GL(V)\)-representation, its
degree-\(p\) component decomposes as
\(B_p=\bigoplus_{\Sigma(I)=p}\bS_I(V)\), where \(\Sigma(I):=\sum_{i\in I}(i{+}1)\) and \(\bS_I(V)\) is
the Schur functor indexed by the Frobenius union of the hooks
\(h_i=(i{+}1,1^{2r-1+i})\) for \(i\in I\).

\begin{proposition}\label{prop:pfaffian-strand-generation}
For disjoint \(I,J\), the product \(\bS_I\cdot\bS_J\) surjects
onto~\(\bS_{I\cup J}\).  In particular, \(B\) is generated by the
singleton summands \(\bS_{\{i\}}\) in its first nonzero strand.
\end{proposition}

\begin{proof}
  We may assume $I,J$ are nonempty and that \(\max(I\cup J)=\max I\). Let $N=\max I$ and argue by induction on $N$.  By naturality, we may take \(\dim V=2r+N\).  Then \(A/I_{2r}\) is
Gorenstein~\cite{JozefiakPragacz1979}, so its Tor algebra is a Poincar\'e
duality algebra~\cite{AvramovGolod1971}. Put
\(\Omega_N:=\{0,\dots,N\}\); the top summand is
\(\bS_{\Omega_N}(V)\). The Littlewood--Richardson
rule identifies the unique dual of \(\bS_I\) as \(\bS_K\), where
\(K=\Omega_N\setminus I\).

Since \(J\subseteq K\) and \(\max K<N\), induction gives a nonzero product
\(\bS_J\cdot\bS_{K\setminus J}\to\bS_K\).  Associativity therefore gives a
nonzero product
\(\bS_I\cdot\bS_J\cdot\bS_{K\setminus J}\to\bS_{\Omega_N}\). The unique
Poincar\'e dual of \(\bS_{K\setminus J}\) is
\(\bS_{I\cup J}\), so the projection of \(\bS_I\cdot\bS_J\) onto that summand is nonzero.
Iterating over the singleton decomposition of \(I\) proves generation.
\end{proof}

\subsubsection{Symmetric determinantal case}

Fix \(r\ge 2\), set \(A=\Sym(S^2 V)\), and let \(I_r\) be the ideal of
\(r\times r\) minors. Let \(M_r\) be the algebra introduced in
\S\ref{sec:det-poincare}, which is finite over \(A/I_r\). The Tor algebra
\(B:=\Tor^A(M_r,\bk)\) decomposes as \(B_p=\bigoplus_{\Sigma(I)=p}\bS_I(V)\), indexed by subsets \(I\subset\mathbb Z_{\ge 0}\), where \(\Sigma(I):=\sum_{a\in I}a\) and \(\bS_I(V):=\bS_{\lambda(I)^{(r-2)}}(V)\) where $\lambda(I)$ has Frobenius coordinates $(I|I)$. The involution on \(M_r\) with fixed subalgebra \(A/I_r\) induces one on \(B\), with \(B^{\bZ/2}\cong\Tor^A(A/I_r,\bk)\) the sum of the \(\bS_I(V)\) with
\(\#I\) even.

\begin{proposition}\label{prop:symm-strand-generation}
For all disjoint \(I,J\), the product \(\bS_I\cdot\bS_J\) in $B$ surjects onto \(\bS_{I\cup J}\).  In particular, \(B\) is generated by the summands \(\bS_{\{a\}}\) with \(a\ge 0\), and $B^{\bZ/2}$ is generated by the summands $\bS_{\{a,b\}}$ with $a>b \ge 0$.
\end{proposition}

The proof is similar to the Pfaffian case.

\subsubsection{Generic determinantal case}

Fix \(\dim V=n\), \(\dim W=m\), \(2\le r\le\min(n,m)\), set \(A=\Sym(V\otimes W)\),
and let \(I_r\) be the ideal of \(r\times r\) minors.  The Tor algebra
\(B:=\Tor^A(A/I_r,\bk)\) decomposes as
\(B_p=\bigoplus_{|\lambda|=p}\bS_\lambda(V,W)\), where
\(\bS_\lambda(V,W):=\bS_{\lambda^{(r-1)}}(V)\otimes
\bS_{(\lambda^{\mathsf t})^{(r-1)}}(W)\) and
\(\lambda\subseteq (m{-}r{+}1)^{n-r+1}\). Here \(\lambda^{(r-1)}\) uses the
row insertion of Definition~\ref{def:row-insertion}. The two factors encode
the same multiset of Frobenius hooks, with arms and legs interchanged; the
fine Segre presentation below records both factors simultaneously.
Write \(B_\lambda:=\bS_\lambda(V,W)\).

\begin{proposition}\label{prop:generic-strand-generation}
If \(\lambda,\mu\) are mergeable
\textup{(Definition~\ref{def:frobenius-union})}, then the product
\(\bS_\lambda(V,W)\cdot\bS_\mu(V,W)\) surjects onto
\(\bS_{\lambda\sqcup\mu}(V,W)\).  In particular, \(B\) is generated
by the hook summands \(\bS_{(a+1,1^b)}(V,W)\).
\end{proposition}

\begin{lemma}\label{lem:generic-corner}
Let \(c=x_{nm}\) and let \(M'=(x_{ij}-x_{im}x_{nj}c^{-1})_{i<n,\,j<m}\) be the
Schur complement over \(A[c^{-1}]\).  Then \(M'\) is a generic
\((n-1)\times(m-1)\) matrix, \(I_r\cdot A[c^{-1}]=I_{r-1}(M')\cdot A[c^{-1}]\),
and the subgroup \(\GL(V')\times\GL(W')\) of \(\GL(V)\times\GL(W)\) fixing the
last basis vectors of \(V\) and \(W\) (with \(V'=\bk^{\,n-1}\),
\(W'=\bk^{\,m-1}\)) acts on \(M'\) as the full generic action.
\end{lemma}

\begin{proof}
Over \(A[c^{-1}]\), row and column operations clear row \(n\) and column \(m\)
off the corner, carrying the generic matrix to \(M' \oplus (c)\) up to units; hence \(I_r\) localizes to \(I_{r-1}(M')\).  The substitution
\(x_{ij}\mapsto x_{ij}-x_{im}x_{nj}c^{-1}\) (\(i<n\), \(j<m\)) is unitriangular
with unit Jacobian, so the interior entries are independent coordinates and
\(M'\) is generic; the stated subgroup transforms each interior entry in its
row and column index while fixing \(c\).
\end{proof}

Let \(\xi\) be the matrix with \(x_{nm}(\xi)=1\) and
\(x_{ij}(\xi)=0\) for \((i,j)\ne(n,m)\), and let
\[
\mathfrak m_\xi=
(x_{ij}:(i,j)\ne(n,m)) + ( x_{nm}-1)\subset A,
\qquad
R:=A_{\mathfrak m_\xi},
\qquad
\bk_\xi:=R/\mathfrak m_\xi R\cong\bk.
\]
Introduce independent variables \(y_{ij}\), \(i<n\), \(j<m\), and put
\[
  A':=\bk[y_{ij}:i<n,\ j<m]\cong\Sym(V'\otimes W'),
  \qquad y_{ij}\longmapsto M'_{ij}.
\]
Write \(C\) for the minimal graded free resolution of \(A/I_r\), \(D\) for
the minimal \(\GL(V')\times\GL(W')\)-equivariant \(A'\)-free resolution of
\(A'/I_{r-1}(y)\), and \(K:=D\otimes_{A'}R\).

\begin{lemma}\label{lem:generic-loc}
The complex \(K\) is the minimal \(R\)-free resolution of \(R/I_rR\), and
\[
  C\otimes_A R\cong K\oplus T,
\]
where \(T\) is a bounded contractible complex of finite free \(R\)-modules.
Specialization at \(\xi\) gives a natural
\(\GL(V')\times\GL(W')\)-equivariant isomorphism of graded algebras
\[
  \rH_\bullet(C\otimes_A\bk_\xi)
  \cong \Tor^{A'}_\bullet(A'/I_{r-1}(y),\bk).
\]
\end{lemma}

\begin{proof}
The coordinates obtained above identify \(R\) with a localization of
\[
  A'[x_{1m},\ldots,x_{n-1,m},x_{n1},\ldots,x_{n,m-1},c,c^{-1}]
\]
at the maximal ideal generated by \(A'_+=(y_{ij}:i<n,\ j<m)\), the border
variables, and \(c-1\).
Thus \(A'\to R\) is flat, and Lemma~\ref{lem:generic-corner} gives
\(I_rR=I_{r-1}(y)R\).  Hence \(K\) resolves \(R/I_rR\); it is minimal because
its differentials have entries in \(A'_+\subset\mathfrak m_\xi R\).

Over a local ring, every bounded finite free resolution is the direct sum of
its minimal resolution and a bounded contractible complex. Applying this to
\(C\otimes_A R\) gives the decomposition. Base change along the Schur
complement also gives natural equivariant quasi-isomorphisms of augmented
differential graded \(\bk\)-algebras
\[
 (A/I_r)\otimes_A^{\mathrm L}\bk_\xi
 \simeq (R/I_rR)\otimes_R^{\mathrm L}\bk_\xi
 \simeq (A'/I_{r-1}(y))\otimes_{A'}^{\mathrm L}\bk.
\]
Taking homology gives the displayed algebra isomorphism.
\end{proof}

\begin{samepage}
Let \(A_+=(x_{ij})\subset A\) be the homogeneous maximal ideal. Since \(C\)
is minimal, reduction modulo \(A_+\) kills its differential. Thus
\[
F_p:=C_p/A_+C_p
\cong\Tor_p^A(A/I_r,\bk)
=B_p
\cong\bigoplus_{|\rho|=p}\bS_\rho(V,W).
\]
\end{samepage}

\begin{lemma}\label{lem:generic-mult-one}
For each \(\lambda\), set
\[
  U_\lambda:=
  \bS_{\lambda^{(r-2)}}(V')\otimes
  \bS_{(\lambda^{\mathsf t})^{(r-2)}}(W').
\]
The representation \(U_\lambda\) occurs with multiplicity one in the
restriction of \(F_{|\lambda|}\), and only in the restriction of its
summand \(\bS_\lambda(V,W)\). Its unique copy in
\(C_{|\lambda|}\otimes_A\bk_\xi\) represents the summand indexed by
\(\lambda\) in
\(\Tor^{A'}_\bullet(A'/I_{r-1}(y),\bk)\) under
Lemma~\ref{lem:generic-loc}.
\end{lemma}

We use below that $\bS_\alpha(V')$ occurs in the restriction of $\bS_\beta(V)$ if and only if $\alpha \subseteq \beta$ and $\beta/\alpha$ is a horizontal strip, i.e., does not contain more than one box in any column (and similar statements apply to $W'$ and $W$), see e.g., \cite[(3.11)]{introtca}. 

\begin{proof}
Within the single irreducible \(\bS_\lambda(V,W)\), $U_\lambda$ occurs exactly once by the branching rule stated above since  \(\lambda^{(r-1)}/ \lambda^{(r-2)}\) is a single row of length \(\rank(\lambda)\).  As the term in homological degree \(p\) is
\(\bigoplus_{|\nu|=p}\bS_\nu(V,W)\), it suffices to show that \(\nu=\lambda\) is
the only partition with \(|\nu|=|\lambda|\) for which both
\(\nu^{(r-1)}/\lambda^{(r-2)}\) and \((\nu^{\mathsf t})^{(r-1)}/(\lambda^{\mathsf t})^{(r-2)}\) are
horizontal strips. 

Choose $\lambda,\nu$ satisfying these two horizontal strip conditions.
Write Frobenius coordinates \(\nu=(a_i\mid b_i)_{i\le d}\) and
\(\lambda=(a'_i\mid b'_i)_{i\le e}\).
First, \(\lambda^{(r-2)}\subseteq\nu^{(r-1)}\) implies that \(e\le d\).  If \(e<d\), then by Definition~\ref{def:row-insertion}, column \(e{+}1\) of \(\nu^{(r-1)}\) has height \((e{+}1)+b_{e+1}+(r{-}1)\ge
e+r\), while column \(e{+}1\) of \(\lambda^{(r-2)}\)
has height \(\le e\). Hence column $e+1$ of $\nu^{(r-1)}/\lambda^{(r-2)}$ has at least \(r\ge2\) boxes, a contradiction.

Thus \(e=d\), and comparing \(i\)-th column heights gives
\(0\le(b_i+r-1)-(b'_i+r-2)\le1\), i.e.\ \(b_i\le b'_i\) for \(i\le d\).  The
second strip condition is the first for the transposed partitions, so the same
computation on rows gives \(a_i\le a'_i\).  Then
\[
  0=|\lambda|-|\nu|=\sum_{i=1}^{d}\bigl[(a'_i-a_i)+(b'_i-b_i)\bigr]
\]
is a sum of nonnegative terms, forcing \(a_i=a'_i\) and \(b_i=b'_i\) for
all \(i\); hence \(\nu=\lambda\).

Thus \(U_\lambda\) occurs once in
\(C_{|\lambda|}\otimes_A\bk_\xi\). The algebra
\(\Tor^{A'}_\bullet(A'/I_{r-1}(y),\bk)\) contains it once in homological
degree \(|\lambda|\). Since \(\GL(V')\times\GL(W')\) is reductive, this unique copy
must survive to homology and maps isomorphically onto that summand.
\end{proof}

\begin{proof}[Proof of Proposition~\textup{\ref{prop:generic-strand-generation}}]
Set \(\nu=\lambda\sqcup\mu\), and let
\[
  m_r(\lambda,\mu)\colon B_\lambda\otimes B_\mu\longrightarrow B_\nu
\]
be multiplication followed by projection onto \(B_\nu\). We prove by induction
on \(r\) that this map is nonzero.

For \(r=1\), \(A/I_1=\bk\) and \(B=\bigwedge(V\otimes W)\) is an exterior
algebra. The maps \(m_1(\lambda,\mu)\) are
nonzero by the results in \cite[\S 2]{GrahamHunziker2009}, after applying the monoidal
transpose functor to the \(W\) factor as in \S\ref{sec:further-koszul-examples}; see \cite[\S7.4]{introtca}.

Let \(r\ge2\) and assume the result for \(r-1\). Choose a homogeneous
\(\GL(V)\times\GL(W)\)-equivariant chain map
\(C\otimes_A C\to C\) lifting multiplication on \(A/I_r\), and put
\[
  j_r(\rho):=|\rho|+(r-1)\rank(\rho).
\]
By the Tor decomposition in \S\ref{sec:det-poincare}, \(j_r(\rho)\) is the
internal degree of the free generators in the summand indexed by~\(\rho\).
Both size and rank are additive under Frobenius union, whence
\(j_r(\nu)=j_r(\lambda)+j_r(\mu)\).  The component from the
\(\lambda,\mu\) free generators to the \(\nu\) free generator therefore has
matrix entries in \(A_0=\bk\). Its reduction modulo \(A_+\) is
\(m_r(\lambda,\mu)\). Hence, if \(m_r(\lambda,\mu)=0\), this chain map
component is zero and remains zero after specialization at \(\xi\).

By Lemmas~\ref{lem:generic-loc} and~\ref{lem:generic-mult-one}, the unique
copies of \(U_\lambda,U_\mu,U_\nu\) in the specialized resolution represent
the corresponding summands of
\(\Tor^{A'}_\bullet(A'/I_{r-1}(y),\bk)\). Moreover, after restriction to
\(\GL(V')\times\GL(W')\), \(U_\nu\) occurs only in
\(\bS_\nu(V,W)\subset F_{|\nu|}\).
If \(m_r(\lambda,\mu)=0\), the \(U_\nu\)-component of the specialized
product on \(U_\lambda\otimes U_\mu\) would therefore vanish. The algebra
isomorphism of Lemma~\ref{lem:generic-loc} identifies this with
\(m_{r-1}(\lambda,\mu)\), contradicting the inductive hypothesis. Thus
\(m_r(\lambda,\mu)\ne0\), and this map is surjective because \(B_\nu\) is
irreducible.

For generation, write \(\lambda\) as the Frobenius union of its diagonal
hooks
\[
H_i(\lambda)=(a_i{+}1,1^{b_i}). 
\]
Applying the nonvanishing successively to these hooks shows that the product of the corresponding
summands in the first nonzero strand has nonzero projection onto
\(\bS_\lambda(V,W)\). Hence \(B\) is generated by its first nonzero strand.
\end{proof}

\medskip
Write \(s:=r-1\), \(\pi(a,b):=a+b+1\pmod 2\), and
\begin{equation}\label{eq:generic-base-hook-set}
  H=\{(a\mid b):0\le a\le m-r,\ 0\le b\le n-r\}.
\end{equation}
We identify \(h=(a\mid b)\in H\) with the hook partition
\((a+1,1^b)\).
For a partition \(\rho\), let \(\mathsf H(\rho)\) denote the multiset of its
Frobenius pairs from Definition~\ref{def:frobenius-coords}.
For \(h=(a\mid b)\in H\), put
\[
 E_h:=\bS_{(a+1,1^{b+s})}(V)\otimes
      \bS_{(b+1,1^{a+s})}(W),
\]
and set \(E=\bigoplus_{h\in H}E_h\), with \(E_h\) of parity \(\pi(a,b)\).
Let \(\Sym_\pi(E)\) be the supersymmetric algebra for this parity, and let
$q_\pi\colon E\otimes E\twoheadrightarrow\Sym_\pi^2(E)$
be the canonical quotient. If \(h,h'\) are mergeable, let
\(\eta(h,h')\) be the rank two partition having Frobenius hooks \(h,h'\),
and let
\[
p_{h,h'}\colon E_h\otimes E_{h'}\twoheadrightarrow
\bS_{\eta(h,h')^{(s)}}(V)\otimes
\bS_{(\eta(h,h')^{\mathsf t})^{(s)}}(W)
\]
be the tensor product of the projections onto these two unique
Littlewood--Richardson summands. Its kernel is independent of all scalar
choices. Define the
\defi{fine hook algebra} by
\[
G^\pi(V,W):=
\Sym_\pi(E)\Big/
\left\langle
\begin{array}{ll}
q_\pi(E_h\otimes E_{h'})&\text{if \(h,h'\) are not mergeable},\\
q_\pi(\ker p_{h,h'})&\text{if \(h,h'\) are mergeable}.
\end{array}
\right\rangle.
\]
Equivalently,
\[
  G^\pi(V,W)=\bigoplus_\lambda
  \bS_{\lambda^{(s)}}(V)\otimes\bS_{(\lambda^{\mathsf t})^{(s)}}(W),
  \qquad \lambda\subseteq(m-s)^{n-s}.
\]
Its product is zero for a nonmergeable pair and, for a mergeable pair
\(\lambda,\mu\), is the projection onto the
summand indexed by \(\lambda\sqcup\mu\). Write \(G^0(V,W)\) for its purely
even form.

Unlike the Pfaffian and symmetric cases, the generic hooks $(a\mid b)$ vary over
a two-dimensional grid rather than a single diagonal, so the combinatorial input
takes the following form involving both factors. Order the finite hook set \(H\) by
\begin{equation}\label{eq:generic-hook-order}
  (a\mid b)\prec_H(c\mid d)
  \quad\Longleftrightarrow\quad
  a+b<c+d,\ \text{ or }\ a+b=c+d\text{ and }a<c.
\end{equation}
Multisets are compared lexicographically after sorting in decreasing order.
The second criterion
only makes the order total; the hooks of equal size compared below always
coincide.

\begin{lemma}\label{lem:2d-jump-ladder}
Assume \(s>0\). Let \(\lambda,\mu\) be partitions and let \(M\) be the
multiset union of \(\mathsf H(\lambda)\) and \(\mathsf H(\mu)\). Suppose a
partition \(\nu\) satisfies
\(|\nu|=|\lambda|+|\mu|\) and
\[
c^{\nu^{(s)}}_{\lambda^{(s)},\mu^{(s)}}\ne0,\qquad
c^{(\nu^{\mathsf t})^{(s)}}_{(\lambda^{\mathsf t})^{(s)},(\mu^{\mathsf t})^{(s)}}\ne0.
\]
If \(\mathsf H(\nu)\ne M\), then
\(\mathsf H(\nu)>_{\mathrm{lex}}M\).
\end{lemma}

\begin{proof}
Since \(|X^{(s)}|=|X|+s\,\rank(X)\), comparison of sizes in the first
Littlewood--Richardson coefficient gives
\[
\rank(\nu)=\rank(\lambda)+\rank(\mu).
\]
Thus \(\nu\) has the same number of Frobenius hooks as \(M\). We argue by
induction on this number, as in the proof of
Lemma~\ref{lem:multi-jump-ladder}. If \(M=\varnothing\), then
\(\lambda=\mu=\nu=\varnothing\), and there is nothing to prove.

The first Littlewood--Richardson condition implies
\(\lambda^{(s)},\mu^{(s)}\subseteq \nu^{(s)}\). By
Definition~\ref{def:row-insertion}, row insertion leaves the Frobenius arms
unchanged, so the largest arm of \(\nu\) is at least the largest arm in
\(M\). Applying the same argument to the transposed condition shows that the
largest leg of \(\nu\) is at least the largest leg in \(M\).

Let \((k_a\mid k_b)\) be the Frobenius pair of the first diagonal hook
\(H_1(\nu)\) from Definition~\ref{def:frobenius-coords}, and let
\((m_a\mid m_b)\) be the largest element of \(M\) under \(\prec_H\).
The preceding discussion implies that \(k_a+k_b\ge m_a+m_b\). If this inequality is strict, the
first hook of \(\nu\) exceeds the largest hook of \(M\) in size, so
\(\mathsf H(\nu)>_{\mathrm{lex}}M\).

It remains to treat the case that \(k_a+k_b=m_a+m_b\).
Thus the maximal arm and maximal leg in \(M\) occur in
the same hook \((m_a\mid m_b)=(k_a\mid k_b)\), which belongs to one of
\(\lambda,\mu\). After interchanging them, assume it is the first diagonal
hook of \(\lambda\), and let \(\nu^-,\lambda^-\) be obtained by deleting
this common hook. Removing the shared first row and column identifies
\[
\nu^{(s)}/\lambda^{(s)}\ \cong\ (\nu^-)^{(s)}/(\lambda^-)^{(s)},
\]
and likewise after transposition. The same Littlewood--Richardson tableaux
therefore witness both nonvanishing conditions for
\((\lambda^-,\mu,\nu^-)\). Let \(M^-\) be obtained from \(M\) by deleting
the common hook. If \(\mathsf H(\nu^-)=M^-\), then
\(\mathsf H(\nu)=M\), contrary to hypothesis. Otherwise induction gives
\(\mathsf H(\nu^-)>_{\mathrm{lex}}M^-\), and restoring the common first
hook gives \(\mathsf H(\nu)>_{\mathrm{lex}}M\).
\end{proof}

\begin{proposition}[Generic degeneration]\label{prop:generic-degeneration}
  Choose positive rational weights
\(\varepsilon_h>0\) for $h\in H$ with \(\varepsilon_h>2\sum_{h'\prec_H h}\varepsilon_{h'}\) and
\(\sum_h\varepsilon_h<1\).  For a partition \(\lambda\), put
\[
  d(\lambda)=|\lambda|-\sum_{h\in\mathsf H(\lambda)}\varepsilon_h.
\]
The degrees \(d(\lambda)\) define an ascending multiplicative filtration
\(F_\bullet B\), and \(\gr_F B\cong G^\pi(V,W)\) as
\(\GL(V)\times\GL(W)\)-equivariant superalgebras graded by strands.
\end{proposition}

\begin{proof}
The proof is identical to that of
Proposition~\ref{prop:diagonal-hook-degeneration}, using
Lemma~\ref{lem:2d-jump-ladder} in place of
Lemma~\ref{lem:multi-jump-ladder}; the nonzero Frobenius-union projection is
given by Proposition~\ref{prop:generic-strand-generation}.
\end{proof}

\smallskip\noindent
Recall the set \(H\) from \eqref{eq:generic-base-hook-set}, and set
\[
 H_V=\{(a\mid b+s):(a\mid b)\in H\},\qquad
 H_W=\{(b\mid a+s):(a\mid b)\in H\}.
\]
Let
\[
A_V:=\bigoplus_{\substack{\rho:\ \mathsf H(\rho)\subseteq H_V}}\bS_\rho(V)
\subseteq\cH(V),
\]
and define \(A_W\subseteq\cH(W)\) analogously using \(H_W\). Since
Frobenius union retains every input hook, the componentwise argument in the
proof of Proposition~\ref{prop:diagonal-retracts} shows that the projections
onto \(A_V\) and \(A_W\) are algebra retractions. The quadratic Gr\"obner
basis of Theorem~\ref{thm:quad-gb} is homogeneous for the hook multigrading
of Definition~\ref{def:hook-multigrading}, so its restrictions to these hook
types give quadratic Gr\"obner bases for \(A_V\) and \(A_W\).

\smallskip\noindent\emph{The common multigrading.}\quad
Let \(\{e_h\}_{h\in H}\) be the standard basis of \(\mathbb N^{H}\).  Grade \(A_V\) by \(\mathbb N^{H}\), assigning the generator summand \(\bS_{(a+1,1^{b+s})}(V)\) the multidegree \(e_{(a\mid b)}\); then
\(\deg \bS_{\lambda^{(s)}}(V)= \sum_i e_{(a_i\mid b_i)}\), where
\((a_i\mid b_i)\) are the Frobenius hooks of \(\lambda\).  Grade \(A_W\) by
the \emph{same} \(\mathbb N^{H}\), assigning \(\bS_{(b+1,1^{a+s})}(W)\) the
multidegree \(e_{(a\mid b)}\); then
\(\deg \bS_{(\lambda^{\mathsf t})^{(s)}}(W) = \sum_i e_{(a_i\mid b_i)}\).

\begin{definition}\label{def:fine-segre}
Let \(A=\bigoplus_{\alpha\in\mathbb N^{H}}A_\alpha\) and
\(B=\bigoplus_{\alpha\in\mathbb N^{H}}B_\alpha\) be \(\mathbb N^{H}\)-graded
\(\bk\)-algebras with \(A_{\mathbf0}=B_{\mathbf0}=\bk\), each generated by
its components in the degrees \(\{e_h\}_{h\in H}\).
Their \defi{fine \textup{(}multigraded\textup{)}
Segre product} is the subalgebra
\[
  A\boxtimes B:=\bigoplus_{\alpha\in\mathbb N^{H}}A_\alpha\otimes_\bk B_\alpha
  \ \subseteq\ A\otimes_\bk B ,
\]
with componentwise multiplication.
\end{definition}

Thus \(G^0(V,W)=A_V\boxtimes A_W\), while \(G^\pi(V,W)\) is the Zhang twist
determined by \(\pi\) and the common \(\mathbb N^H\)-grading. To apply
Sullivant's theorem, we present this fine Segre product as a toric fiber
product. Fix
\(\bk\)-bases \(\{x^h_j\}_j\) of \(A_{V,e_h}\) and \(\{y^h_k\}_k\) of
\(A_{W,e_h}\), write \(A_V=\bk[x^h_j]/I_V\) and
\(A_W=\bk[y^h_k]/I_W\), and set
\(z^h_{jk}:=x^h_j\otimes y^h_k\).
Write \(\mathcal D=\{e_h:h\in H\}\). Since the degree vectors \(e_h\) are
linearly independent, Sullivant's toric fiber product
\(I_V\times_{\mathcal D}I_W\)~\cite{Sullivant2007} has codimension zero. By
\cite[Lemma~7]{KahleRauh2014},
\[
  G^0(V,W)=A_V\boxtimes A_W
  \cong \bk[z^h_{jk}]/(I_V\times_{\mathcal D}I_W).
\]

\begin{proposition}\label{prop:generic-fine-segre}
\(G^\pi(V,W)\) is Koszul.
\end{proposition}

\begin{proof}
The factors \(A_V\) and \(A_W\) have quadratic Gr\"obner bases, and the
degree vectors \(\{e_h\}\) are linearly independent.  Sullivant's
theorem for toric fiber products of codimension zero
\cite[Thms.~2.8--2.9, Prop.~2.6]{Sullivant2007} gives a quadratic
Gr\"obner basis for \(I_V\times_{\mathcal D}I_W\). Thus \(G^0(V,W)\) is
Koszul.  Since \(G^\pi(V,W)\) is a graded Zhang twist of \(G^0(V,W)\), it is
Koszul as well~\cite{Zhang1996}.
\end{proof}

\begin{proof}[Proof of Theorem~\textup{\ref{thm:koszul-strands}}]
Propositions~\ref{prop:pfaffian-strand-generation},
\ref{prop:symm-strand-generation}, and
\ref{prop:generic-strand-generation} give generation by the first nonzero
strand. In each case, write $\gr(B)$ for the associated graded algebra of the
filtration under consideration and $\pi$ for the parity induced by Tor
homological degree.

\smallskip
\noindent\textbf{Pfaffian:} For \(n=\dim V\), apply
Proposition~\ref{prop:diagonal-hook-degeneration} with
\(\Omega=\{0,\dots,n-2r\}\) and \(L=2r-1\). It identifies \(\gr(B)\) with
the parity form of the Koszul retract in
Proposition~\ref{prop:diagonal-retracts}(i), so \(\gr(B)\) is Koszul.

\smallskip
\noindent\textbf{Symmetric determinantal:} For
$\widetilde B=\Tor^A(M_r,\bk)$, apply
Proposition~\ref{prop:diagonal-hook-degeneration} with
\(\Omega=\{0,\dots,n-r+1\}\) and \(L=r-2\), after row insertion with
parameter~\(r-2\) (Definition~\ref{def:row-insertion}). Thus
\(\gr(\widetilde B)\) is Koszul.
The summands indexed by sets of even cardinality form
$B=\Tor^A(A/I_r,\bk)=\widetilde B^{\bZ/2}$; Corollary~\ref{cor:hook-degreewise-compatibility}
identifies \(\gr(B)\) with the second Veronese of \(\gr(\widetilde B)\),
which is Koszul~\cite{BackelinFroberg1985}.

\smallskip
\noindent\textbf{Generic determinantal:} By
Proposition~\ref{prop:generic-degeneration},
\(\gr(B)\cong G^\pi(V,W)\), and the latter is Koszul by Proposition~\ref{prop:generic-fine-segre}.

In every case the filtration preserves strand degree and is finite and
exhaustive in each such degree. The induced filtration on the reduced bar
complex gives a convergent spectral sequence compatible with the strand grading,
\[
E^1\cong\Tor^{\gr(B)}(\bk,\bk)
\Longrightarrow\gr\Tor^B(\bk,\bk).
\]
The left side is concentrated where strand and homological degrees agree,
so the same is true of the abutment. Thus \(B\) is Koszul.
\end{proof}

\bibliographystyle{alpha}
\bibliography{biblio}

\end{document}